\documentclass[a4paper, 11pt]{amsart}   
\usepackage[colorlinks=true,linkcolor=blue,urlcolor=blue]{hyperref}
\usepackage[onehalfspacing]{setspace}

\pdfoutput=1
\usepackage{latexsym, amsmath, amssymb, amsthm, mathrsfs, bm, mathtools,comment}
\usepackage[mathscr]{eucal}
\usepackage{array}
\usepackage{tikz}
\usetikzlibrary{arrows,shapes,positioning,calc,patterns,cd}
\usepackage{graphicx} 
\usepackage[T1]{fontenc} 
\usepackage{mathptmx}
\usepackage{microtype}

 \usepackage[inner=2.4cm,outer=2.4cm, bottom=3.2cm]{geometry}

\usepackage[initials]{amsrefs}
\allowdisplaybreaks

\usepackage{xcolor, color, soul}

\renewcommand{\leq}{\leqslant}
\renewcommand{\geq}{\geqslant}

\theoremstyle{plain}
\newtheorem{theorem}{Theorem}[section]
\newtheorem{theoremx}{Theorem}
 
\newtheorem{corollary}[theorem]{Corollary}
\newtheorem{lemma}[theorem]{Lemma}
\newtheorem{proposition}[theorem]{Proposition}

\theoremstyle{definition}
\newtheorem{definition}[theorem]{Definition}
\newtheorem{setup}[theorem]{Setup}

\newtheorem{example}[theorem]{Example}

\newtheorem{remark}[theorem]{Remark}
\numberwithin{equation}{subsection}

\newtheorem{question}[theorem]{Question}

\newtheorem{notation}[theorem]{Notation}

\newcommand{\R}[1]{R^{1/p^{#1}}}
\newcommand{\m}{\mathfrak{m}}
\newcommand{\n}{\mathfrak{n}}

\newcommand{\RR}{\mathbb{R}}
\newcommand{\NN}{\mathbb{Z}_{\geq 0}}
\newcommand{\ZZ}{\mathbb{Z}}
\newcommand{\QQ}{\mathbb{Q}}

\newcommand{\fn}{\mathbf{n}}
\newcommand{\fx}{\mathbf{x}}
\newcommand{\fa}{\mathbf{a}}

\newcommand{\J}{\mathcal{J}}
\newcommand{\IN}{\operatorname{in}}

\newcommand{\fp}{\mathfrak{p}}

\newcommand{\cM}{\mathcal{M}}

\newcommand{\Cech}{ \check{\rm{C}}}

\DeclareMathOperator{\mcm}{{mcm}}

\newcommand{\reg}{\operatorname{reg}}

\newcommand{\Spec}{\operatorname{Spec}}

\newcommand{\Hom}{\operatorname{Hom}}

\newcommand{\gr}{\operatorname{gr}}
\newcommand{\pf}{\operatorname{pf}}

\newcommand{\Ker}{\operatorname{Ker}}

\newcommand{\Max}{\operatorname{Max}}

\newcommand{\Ass}{\operatorname{Ass}}	
\DeclareMathOperator{\chara}{{{char}}}

\newcommand{\bh}{\operatorname{bigheight}}	
	
\newcommand{\depth}{\operatorname{depth}}

\newcommand{\height}{\operatorname{height}}	
\newcommand{\rk}{\operatorname{rk}}	
\newcommand{\Ht}{\operatorname{ht}}

\makeatletter
\@namedef{subjclassname@2020}{\textup{2020} Mathematics Subject Classification}
\makeatother

\def \cJ{\mathcal J}
\def \cI{\mathcal I}

\def \R{\mathcal R}

\newcommand{\HH}[3]{\operatorname{H}^{#1}_{#2}\left(#3\right)}

\DeclareMathOperator{\grade}{{grade }}
\DeclareMathOperator{\lcm}{{lcm}}

\newcommand{\ls}{\leqslant}%
\newcommand{\gs}{\geqslant}

\newcommand{\ov}[1]{\overline{#1}}

\newcommand{\II}{\mathbb{I}}

\author[A. De Stefani]{Alessandro De Stefani}
\address{Dipartimento di Matematica, Universit{\`a} degli Studi di Genova, Via Dodecaneso 35, 16146 Genova, Italy}
\email{destefani@dima.unige.it}

\author[J. Monta{\~n}o]{Jonathan Monta{\~n}o}
\address{School of Mathematical and Statistical Sciences, Arizona State University, P.O. Box 871804, Tempe, AZ 85287-18041}
\email{montano@asu.edu}

\author[L. N{\'u}{\~n}ez-Betancourt]{Luis N{\'u}{\~n}ez-Betancourt}
\address{Centro de Investigaci{\'o}n en Matem{\'a}ticas, Guanajuato, Gto., M{\'e}xico}
\email{luisnub@cimat.mx}



\begin{document}

\title[Blowup algebras of determinantal ideals in prime characteristic]{Blowup algebras of determinantal ideals in prime characteristic}

\dedicatory{In memory of Professor Wolmer Vasconcelos}

\begin{abstract} 
We study when blowup algebras are $F$-split or strongly $F$-regular.
Our main focus is on algebras given by 
 symbolic and ordinary powers of ideals of minors of a generic matrix, a symmetric matrix, and a Hankel matrix.  We also study ideals of Pfaffians of a skew-symmetric matrix.
We use these results to obtain bounds on the degrees of the defining equations for these algebras. We also prove that the limit of the normalized regularity of the symbolic powers of these ideals exists and that their depth stabilizes. Finally, we show that, for determinantal ideals, there exists a monomial order for which taking initial ideals commutes with taking symbolic powers. To obtain these results we develop the notion of $F$-split filtrations and symbolic $F$-split ideals.
\end{abstract}

\maketitle
\setcounter{secnumdepth}{1}
\setcounter{tocdepth}{1}
 \tableofcontents

\section{Introduction}

Let $R$ be a Noetherian ring. 
A {\it filtration}  $\II=\{I_n\}_{n\in\NN}$ is a sequence of ideals such that $I_0=R$, $I_{n+1}\subseteq I_n$ for every $n\in \NN$, and $I_m I_n\subseteq I_{m+n}$ for every $n,m\in\NN$. Classical examples of filtrations include ordinary and symbolic powers. By taking the initial ideals of a filtration under a monomial order, one obtains a filtration of monomial ideals.
 Given a filtration, one can construct its Rees algebra $\R(\II)$ and associated algebra $\gr(\II)$.
Notably, the Rees algebra of  the ordinary  powers of an ideal $I$ gives the coordinate ring of the blowup of $\Spec(R)$ along the variety defined by $I$.

In this manuscript we provide several results regarding ordinary and symbolic powers of determinantal ideals, and their Rees and associated graded algebras. Specifically, we study ideals of minors of generic, symmetric, and Hankel matrices of variables. 
We also study ideals of Pfaffians of a skew-symmetric matrix of variables. These objects have been intensively studied together with the varieties that they define, and they have connections with other areas of mathematics. For more information on this topic, we refer the interested reader to Bruns and Vetter's book 
\cite{BookDet}, and to the more recent book of Bruns, Conca, Raicu and Varbaro \cite{BCRV}.

In what follows, $I_t(-)$ denotes the ideal generated by $t$-minors, and $P_{2t}(-)$ the ideal generated by $2t$-Pfaffians. In our first set of results, we  show that the Rees and associated graded algebras of determinantal ideals have mild singularities from the perspective of Frobenius \cites{HoHu1,HoHu2,HoHu3,HoHuStrong}. In particular, we show that several of them are strongly $F$-regular, or at least $F$-split.
 These singularities are regarded as the characteristic $p$ analogue of log-terminal and log-canonical singularities \cites{H98,HW02,MS97,SmithFrational,SmithComm,Smith97} (see also \cite{SurveyTestIdeals}).
We recall that strongly $F$-regular rings are Cohen-Macaulay and normal \cite{HoHuStrong}. They are also simple as modules over their ring of differential operators \cite{DModFSplit}. We point out that the local cohomology modules of $F$-split rings satisfy desirable vanishing theorems \cites{HRFpurity,DSGNB} and their defining ideals satisfy Harbourne's conjecture on symbolic powers \cites{HC1,HC2,GrifoHuneke}. In the following result, we denote by $\R(I)$ the Rees algebra corresponding to the ordinary powers of the ideal $I$, and by $\R^s(I)$ and $\gr^s(I)$ the Rees and associated graded algebras corresponding to the  symbolic powers of the ideal $I$.

\begin{theoremx}\label{MainThmFsing}
Let $K$ be an $F$-finite field of prime characteristic $p>0$. Let $X$ be a 
generic matrix, $Y$ be  a 
generic symmetric matrix,  $Z$  be 
a generic skew-symmetric matrix, and $W$ be a 
generic Hankel matrix. 
For an integer $t>0$ we have
\begin{enumerate}
\item 
$\R^s(I_t(X))$ and $\gr^s(I_t(X))$ are strongly $F$-regular (Theorem  \ref{ThmSymbFregGen}).
\item 
If $p \gg 0$, then $\R(I_t(X))$ is  $F$-split (Theorem \ref{ThmReesFpureGen}).
\item 
$\R^s(I_t(Y))$ and $\gr^s(I_t(Y))$ are   $F$-split (Theorem \ref{ThmSFPSym}).
\item 
If $p \gg 0$, then $\R(I_t(Y))$ is  $F$-split (Theorem \ref{ThmReesFpureSym}).
\item 
$\R^s(P_{2t}(Z))$ and $\gr^s(P_{2t}(Z))$ are strongly $F$-regular (Theorem \ref{ThmSymbFregPf}).
\item 
If $p \gg 0$, then  $\R(P_{t}(Z))$ is  $F$-split (Theorem \ref{ThmReesFpurePf}).
\item 
$\R^s(I_t(W))$ and $\gr^s(I_t(W))$ are   $F$-split (Theorem \ref{ThmSFPHankel}).
\item 
$\R(I_t(W))$ is  $F$-split (Theorem \ref{ThmReesFpureHankel}).
\end{enumerate}
\end{theoremx} 

The wide variety of determinantal objects that we are able to cover in Theorem \ref{MainThmFsing} highlights the fact that our techniques have a broad range of applications. 

Note that, since $K[X]/I_t(X)$ and $K[Z]/P_{t}(Z)$ are direct summands of $\gr^s(I_t(X))$ and $\gr^s(P_{2t}(Z))$, respectively, Theorem \ref{MainThmFsing} (1) and (5) imply the known results that  $K[X]/I_t(X)$ and $K[Z]/P_{t}(Z)$ are strongly $F$-regular. In fact, the proofs of  Theorem \ref{MainThmFsing} (1) and (5) can be specialized  to give alternative proofs for the strong $F$-regularity of $K[X]/I_t(X)$ and $K[Z]/P_{t}(Z)$. 

Although it was already known that  $\R^s(I_t(X))$ is $F$-rational \cite{FRatGeneric}, hence Cohen-Macaulay and normal,
$F$-rationality does not imply that the ring is $F$-split.  Therefore, Theorem \ref{MainThmFsing}(1) improves this result, as strong $F$-regularity implies that the ring is both $F$-rational and $F$-split. Cohen-Macaulayness was also known for symbolic Rees algebras of ideals of Pfaffians of a generic skew-symmetric matrix \cite{baetica98}, and this is now also a consequence of Theorem \ref{MainThmFsing}(5). 
 We point out that the new techniques we use to study $F$-singularities of blowup algebras are neither based on the theory of Sagbi  bases \cites{S1,S2,Sagbi} nor on that of straightening laws  \cites{DRS,FRatGeneric}. We only invoke known results that use Sagbi bases in order to have that some blowup algebras we consider are Noetherian.  
 Our strategy uses the new notion of  $F$-split filtrations (Definition \ref{DefFpureFilt}), classical methods in tight closure theory \cite{HoHuStrong}, and the choice of certain polynomials inspired by  Seccia's work on Knutson ideals \cites{SecciaGen,SecciaHankel}.

Since all the Rees  and associated graded algebras in Theorem \ref{MainThmFsing} are  $F$-split, their $a$-invariants are not  positive \cite{HRFpurity}. As a consequence, we obtain bounds for the Castelnuovo-Mumford regularity and  the degrees of  the defining equations of such algebras; see Theorems \ref{ThmDegGenOrd}, \ref{ThmDegGenSymb}, \ref{ThmDegSymOrd}, \ref{ThmDegSymSymb},
\ref{ThmDegPfOrd}, \ref{ThmDegPfSymb},
\ref{ThmDegHankelOrd}, and  \ref{ThmDegHankelSymb}.
We point out that, even for monomial ideals, it was generally not known how to bound the degrees of the defining equations 
of these Rees algebras in terms of the generators of the ideal. Significant work has been done over the years in order to find such equations via different methods 
\cites{FK2015,Ha02,Ha2002,HSV12,Huckaba,KPU2011,LP2014,MP2013,Morey97,MU96,STV98,SUV95,V91}. 

A related question is whether the limit of normalized Castelnuovo-Mumford regularities, $\lim\limits_{n\to\infty} \frac{\reg(R/I^{(n)})}{n}$, always exists \cite{HHT}. See also the work of Cutkosky on the subject \cite{Cutkosky_irrational}. 
Several authors have approached this question in a variety of cases;  however, it remains widely open in general. Some classes of ideals for which this limit is known to exist are square-free monomial ideals \cite{HoaTrung} and ideals of small dimension \cite{HHT}. 
We obtain this property for determinantal ideals in prime characteristic.

\begin{theoremx}
Let $K$ be an $F$-finite field of prime characteristic. Let $X$ be a 
generic matrix, $Y$ be a 
generic symmetric matrix,  $Z$ be a 
generic skew-symmetric matrix, and $W$ be a 
generic Hankel matrix. 
For an integer $t>0$ we have that
\begin{enumerate}
\item 
$\displaystyle \lim\limits_{n\to\infty} \frac{\reg(K[X]/I_t(X)^{(n)})}{n}$ exists (Theorem \ref{mainDetReg}).
\item 
$\displaystyle \lim\limits_{n\to\infty} \frac{\reg(K[Y]/I_t(Y)^{(n)})}{n}$ exists (Theorem \ref{mainDetRegSym}).
\item 
$\displaystyle \lim\limits_{n\to\infty} \frac{\reg(K[Z]/P_{2t}(Z)^{(n)})}{n}$ exists (Theorem \ref{mainDetRegPf}).
\item 
$\displaystyle \lim\limits_{n\to\infty} \frac{\reg(K[W]/I_t(W)^{(n)})}{n}$ exists (Theorem \ref{mainDetRegHankel}).
\end{enumerate}
\end{theoremx} 

If the ground field is the field of complex numbers,  there are  linear formulas for  $\reg(R/I^{(n)})$ when $I$ is the ideal of $t$-minors of a generic matrix \cite{Raicu} or $2t$-Pfaffians of a generic skew-symmetric  matrix \cite{Perlman}. These results were obtained  using representation theory in characteristic zero. The case of ideals of $t$-minors of a generic matrix was recently further extended to fields of any characteristic \cite{BCRV}. It is worth mentioning that, in general, the function $\reg(R/I^{(n)})$ is not eventually  linear, not even for square-free monomial ideals \cite{ExNotLin}. In particular, this linearity may fail even if   the symbolic Rees algebra, $\R^s(I)$,   is Noetherian.

We also obtain that the depth of symbolic powers of determinantal ideals stabilizes, and in some cases, we obtain the stable value.
Our approach shows that the stable value equals the minimum of the depths among all the symbolic powers.
This minimum value was already computed \cite{BookDet}; however,  to the best of our knowledge, it was not shown that the stable and minimum values coincide. 

\begin{theoremx}
Let $K$ be an $F$-finite field of prime characteristic. Let $X$ be a 
generic matrix, $Y$ be a 
generic symmetric matrix,  $Z$ be a 
generic skew-symmetric matrix, and $W$ be a 
generic Hankel matrix. 
For an integer $t>0$ we have that
  \begin{enumerate}
\item 
$\lim\limits_{n\to\infty} \depth(K[X]/I_t(X)^{(n)})=t^2-1$ (Theorem \ref{mainDetReg}).
\item 
$ \depth(K[Y]/I_t(Y)^{(n)})$ stabilizes for $n \gg 0$ (Theorem \ref{mainDetRegSym}).
\item  
$\lim\limits_{n\to\infty} \depth(K[Z]/P_{2t}(Z)^{(n)})=t(2t-1)-1$ (Theorem \ref{mainDetRegPf}).
\item 
$\depth(K[W]/I_t(W)^{(n)})$ stabilizes for $n \gg 0$ (Theorem \ref{mainDetRegHankel}).
\end{enumerate}
\end{theoremx}

It is known that the initial ideals of the determinantal ideals treated in this work are radical with respect certain monomial orders (see Section \ref{SubsectionDeterminantal}). Then, it is natural to compare the initial ideal of their symbolic powers and the symbolic powers of their initial ideals. Sullivant showed that $\IN_<(I^{(n)})\subseteq \IN_<(I)^{(n)}$ if $K$ is algebraically closed and   $\IN_<(I)$ is  radical \cite{Sull}. In the case of ideals of minors of generic matrices \cite{BCInitial}, and of Hankel matrices of variables \cite{conca98}, not only the containment, but in fact equality is known to hold. As a consequence of the techniques introduced in this article we recover these results, and we also obtain equality in the case of Pfaffians.

\begin{theoremx}\label{MainThmInitial}
Let $K$ be a perfect field of prime characteristic. Let $X$ be a 
generic matrix,  $Z$ be a 
generic skew-symmetric matrix, and $W$ be a 
generic Hankel matrix. 
In each case, let $<$ be the monomial order introduced in Section \ref{SubsectionDeterminantal}. For an integer $t>0$ we have that
  \begin{enumerate}
\item  
$\gr\big(\{\IN_<\big(I_t(X)^{(n)}\big)\}_{n\in \NN}\big)$ is   $F$-split, therefore
$$\IN_<(I_t(X)^{(n)})=\IN_<(I_t(X))^{(n)}$$ for every $n\in\NN$ (Theorems \ref{ThmInitialGen}, and  \ref{ThmInitialGen}).
\item 
$\gr\big(\{\IN_<\big(P_{2t}(Z)^{(n)}\big)\}_{n\in \NN}\big)$  is   $F$-split, therefore
$$\IN_<(P_{2t}(Z)^{(n)})=\IN_<(P_{2t}(Z))^{(n)}$$ for every $n\in\NN$ (Theorems  \ref{ThmInitialPf}, and \ref{ThmInitialPf}).
\item  
$\gr\big(\{\IN_<\big(I_t(W)^{(n)}\big)\}_{n\in \NN}\big)$ is   $F$-split, therefore
$$\IN_<(I_t(W)^{(n)})=\IN_<(I_t(W))^{(n)}$$ for every $n\in\NN$ (Theorems \ref{ThmInitialHankel}, and \ref{ThmInitialHankel}).
\end{enumerate}
\end{theoremx}

If $K$ is a field of characteristic zero, then we obtain the equalities 
$\IN_<(I_t(X)^{(n)})=\IN_<(I_t(X))^{(n)}$, 
$\IN_<(P_{2t}(Z)^{(n)})=\IN_<(P_{2t}(Z))^{(n)}$, and
$\IN_<(I_t(W)^{(n)})=\IN_<(I_t(W))^{(n)}$ 
for every $n\in\NN$ (see Corollaries \ref{CorInitialGen}, \ref{CorInitialPf}, and \ref{CorInitialHankel}), via reduction to prime characteristic. We point out that we do not obtain analogous results for generic symmetric matrices $Y$ because we do not know whether the Rees algebra associated to the filtration $\{\IN_<(I_t(Y)^{(n)})\}_{n \in \NN}$ is Noetherian.

Theorem \ref{MainThmInitial} allows us to find bounds for numerical invariants 
of determinantal ideals in terms of  their initial ideals (see Remarks \ref{RemW} and  \ref{RemResurgence}).
We also  obtain that the regularity and depth of the initial ideals of such determinantal ideals satisfy desirable properties (see Corollary \ref{CorInitialTHms}). 
In this context, we provide new examples of existence of limits of normalized Castelnuovo-Mumford regularities for filtrations given by initial ideals. This is closely related to  questions  previously  asked by Herzog, Hoa, and Trung \cite{HHT}. 

We stress that our strategy to show Theorem \ref{MainThmInitial} makes no use of the standard techniques employed before to obtain 
results about initial ideals of determinantal rings and their ordinary and symbolic powers \cites{baetica98,BCInitial,dCEP,dCP,Hankel,SKRS}. In particular, we use neither the straightening laws \cite{DRS} nor the Knuth-Robinson-Schensted correspondence.  Indeed, our techniques to prove Theorem \ref{MainThmInitial} rely on methods in prime characteristic, and a test for the equality $\IN_<(I^{(n)})=\IN_<(I)^{(n)}$ (see Theorem \ref{Theorem initial torsion free} and Corollary \ref{CorInSymbEq}) inspired by the work of Huneke, Simis and Vasconcelos \cite{HSV89}.

Our main tool in this manuscript is our new  notion of  $F$-split filtration (Definition \ref{DefFpureFilt}). 
 If the  $F$-split filtration is given by symbolic powers of an ideal, we say that the ideal is symbolic $F$-split (Definition \ref{DefSymbolicFpure}). 
Ideals that are symbolic $F$-split produce symbolic Rees algebras and symbolic associated graded algebras that  are  $F$-split (see Theorem \ref{mainCharP}).  As for the classical notion of $F$-purity, there exists a criterion that allows us to test when an ideal is symbolic $F$-split  (see Theorem \ref{ThmFedders}) which resembles the one given by Fedder \cite{FedderFputityFsing}. We note that if an ideal is symbolic $F$-split, then its quotient ring is $F$-split. However, the converse is not true: in Example \ref{ExFpureNotSymb} we show that even strong $F$-regularity  does not imply that the ideal is symbolic $F$-split.
Examples of symbolic $F$-split ideals include  square-free monomial  ideals  (see Example \ref{SqFree_Example}) and determinantal ideals (see Theorems \ref{ThmSFPGen}, \ref{ThmSFPSym}, \ref{ThmSFPPf}, and \ref{ThmSFPHankel}).
We refer to  Corollary \ref{CorH}  and Example \ref{ExZhibek} for additional examples. Using these ideas, we are able to answer a question raised by Huneke\footnote{ BIRS-CMO workshop on {\it Ordinary and Symbolic Powers of Ideals} Summer of 2017,  Casa Matem\'atica Oaxaca, Mexico.} regarding $F$-K\"onig ideals (see Example \ref{negativeHuneke}) which arose in connection to the Conforti-Cornu\'ejols conjecture \cite{CC}. 
 We also show that $a$-invariants and depths of symbolic $F$-split ideals have good behavior (see Proposition \ref{PropSymbFpureDepthAinv}).
In addition, there is a finite test to verify that their symbolic and ordinary powers coincide (see Theorem \ref{ThmCompare}).


\section{Notations and preliminaries}\label{Notations and Prelmininaries}

Throughout this manuscript, all rings commutative with identity.
 We begin this section by recalling some notation and preliminary results that we  use in the article.

\subsection{Graded algebras}
A $\NN$-graded ring is a ring $A$ which admits a direct sum decomposition $A= \bigoplus_{n \gs 0} A_n$ of Abelian groups, with $A_{i} \cdot A_{j} \subseteq A_{i+j}$ for all $i$ and $j$.

Assume that $A$ is a $\NN$-graded algebra over a Noetherian ring $A_0$ and let $M = \bigoplus_{n\in \ZZ} M_n$ and $N=\bigoplus_{n \in \ZZ} N_n$ be graded $A$-modules. An $A$-homomorphism $\varphi:M \to N$ is called {\it homogeneous of degree $c$} if $\varphi(M_n) \subseteq N_{n+c}$ for all $n \in \ZZ$. The set of all graded homomorphisms $M \to N$ of all degrees form a graded submodule of $\Hom_R(M,N)$. In general, these two modules are not the same, but they coincide when $M$ is finitely generated \cite{BrHe}.

Given a Noetherian $\NN$-graded algebra $A$, there exist $f_1,\ldots, f_r \in A$  homogeneous elements   such that $A = A_0[f_1,\ldots, f_r]$, which is equivalent to  $\oplus_{n>0}A_n=(f_1,\ldots, f_r)$ \cite[Proposition 1.5.4]{BrHe}. Therefore, if $A_0$ is local, or $\NN$-graded over a field, there is a minimal set of integers $d_1,\ldots, d_r$ such that there exist such $f_1,\ldots, f_r$  of degree $d_1,\ldots, d_r$ respectively. We call these numbers the {\it generating degrees of $A$ as $A_0$-algebra}. 

Let $S=A_0[y_1,\ldots, y_r]$ be a polynomial ring over $A_0$ with $\deg(y_i)=d_i$ for  $1\ls i\ls r$, and let $\phi: S\to A$ be an $A_0$-algebra homomorphism defined by $\phi(y_i)=f_i$ for  $1\ls i\ls r$. Consider the ideal $\cI=\Ker(\phi)$. We call any minimal set of homogeneous generators of $\cI$ the {\it defining equations of $A$} over $A_0$.


 \subsection{Methods in prime characteristic}\label{secMprime}

In this subsection we assume that $A$ is reduced and that it has prime characteristic $p>0$. For $e \in \NN$, let $F^e:A\to A$ denote the $e$-th iteration of the Frobenius endomorphism on $A$. If $A^{1/p^e}$ denotes the ring of $p^e$-th roots of $A$ taken in the total field of fractions of $A$, we can identify $F^e$ with the natural inclusion $\iota: A \hookrightarrow A^{1/p^e}$. Throughout this article, any $A$-linear map $\phi:A^{1/p^e} \to A$ such that $\phi \circ \iota = {\rm id}_A$ is called a splitting of Frobenius, or just a splitting. 

Given an $A$-module $M$, we let $M^{1/p^e}$ denote the $A$-module which has the same additive structure as $M$ and scalar multiplication defined by $a \cdot m^{1/p^e} := (a^{p^e}m)^{1/p^e}$, for all $a \in A$ and $m^{1/p^e} \in M^{1/p^e}$. 

For an ideal $I$ generated by $\{f_1,\ldots, f_u\}$ we denote by $I^{[p^e]}$ the ideal generated by $\{f_1^{p^e},\cdots, f_u^{p^e}\}$. We note that $IA^{1/p^e}=(I^{[p^e]})^{1/p^e}.$

In the case in which $A=\oplus_{n\gs 0}A_n$ is $\ZZ_{\gs 0}$-graded, we can view $A^{1/p^e}$ as a $\frac{1}{p^e}\NN$-graded module in the following way:   we  write $f \in A$ as $f=f_{d_1}+\ldots+f_{d_n}$, with $f_{d_j} \in A_{d_j}$. Then, $f^{1/p^e} = f_{d_1}^{1/p^e}+\ldots+f_{d_n}^{1/p^e}$ where each $f_{d_j}^{1/p^e}$ has degree $d_j/p^e$. Similarly, if $M$ is a $\ZZ$-graded $A$-module, we have that $M^{1/p^e}$ is a $\frac{1}{p^e}\ZZ$-graded $A$-module. As a submodule of $A^{1/p^e}$, $A$  inherits a natural $\frac{1}{p^e}\NN$ grading, which is compatible with its original grading. In other words, if $f \in A$ is homogeneous of  degree $d$ with respect to its original grading, then it has degree $d=dp^e/p^e$ with respect to the inherited $\frac{1}{p^e}\NN$ grading.

\begin{definition} \label{DefnFSing} 
Let $A$ be a Noetherian ring of positive characteristic $p$. We say that $A$ is  {\it $F$-finite} if it is a finitely generated $A$-module via the action induced by the Frobenius endomorphism $F: A \to A$ or, equivalently, if $A^{1/p}$ is a finitely generated $A$-module.  If $(A,\m,K)$ is a $\NN$-graded $K$-algebra, then $A$ is $F$-finite if and only if $K$ is $F$-finite, that is, if and only if $[K:K^p]< \infty$. A ring $A$ is called {\it $F$-pure} if $F$ is a pure homomorphism, that is, if and only if the map $A \otimes_A M \to A^{1/p} \otimes_A M$ induced by the inclusion $\iota$ is injective for all $A$-modules $M$. A ring $A$ is  called {\it $F$-split} if $F$ is a split monomorphism. Finally, an $F$-finite ring $A$ is called {\it strongly $F$-regular} if for every $c \in A$ not in any minimal prime, the map $A \to A^{1/p^e}$ sending $1 \mapsto c^{1/p^e}$ splits for some (equivalently, all) $e \gg 0$.
\end{definition}

\begin{remark}\label{FpureFsplit}
 We have that $A$ is $F$-split if and only if $A$ is a direct summand of $A^{1/p}$. If $A$ is an $F$-finite ring, then $A$ is $F$-pure if and only $A$ is $F$-split  \cite[Corollary $5.3$]{HRFpurity}. 
\end{remark}

\begin{remark}\label{Trace}
 Assume $A$ is an $F$-finite regular  local ring, or a polynomial ring over an $F$-finite field, then $\Hom_A(A^{1/p^e},A)$ is a free $A^{1/p^e}$-module \cite[Lemma 1.6]{FedderFputityFsing}. If $\Phi$ is a generator  (homogeneous in the graded case)  of this module as an $A^{1/p^e}$-module, then for ideals $I,J\subset A$ we have that the map $\phi:=f^{1/p^e} \cdot \Phi=\Phi(f^{1/p^e}-)$ satisfies 
$\phi\big(J^{1/p^e}\big)\subseteq I$ 
if and only if
$
f^{1/p^e}\in  \big(IA^{1/p^e} :_{A^{1/p^e}} J^{1/p^e}\big)
$
or, equivalently,
$
f\in  \big(I^{[p^e]} :_{A} J\big)
$
\cite[Proposition 1.6]{FedderFputityFsing}. In particular, 
 $\phi$ is surjective if and only if $f^{1/p^e}\not\in \m A^{1/p^e}$, that is  $f\not\in \m^{[p^e]}$.

Now, assume that  $A=K[x_1,\ldots,x_d]$ is a polynomial ring and that $\gamma:K^{1/p^e}	\to K$ is a splitting. Let
$\Phi: A^{1/p^e}\to A$ be the $A$-linear map  defined by
$$\Phi\left(c^{1/p^e} x_1^{\alpha_1/p^e}\cdots x_d^{\alpha_d/p^e}\right)=
\begin{cases}
\gamma(c^{1/p^e}) x_1^{(\alpha_1-p^e+1)/p^e}\cdots x_d^{(\alpha_d-p^e+1)/p^e}& \text{if } p^e| (\alpha_i-p^e+1) \ \ \forall  i,\\
0& \text{otherwise.}
\end{cases}$$
We have that $\Phi$ is a generator of   $\Hom_A(A^{1/p^e},A)$ as an $ A^{1/p^e}$-module \cite[Page 22]{Brion_Kumar_book}.  The map  $\Phi$ is often called the {\it trace map} of $A$. We point out that, if $K$ is not perfect,   $\Phi$ depends on $\gamma$, but this is usually omitted from the notation.
\end{remark}

\subsection{Local cohomology and Castelnuovo-Mumford regularity}

For an ideal $I\subseteq A$, we define the {\it $i$-th local cohomology of $M$ with support in $I$} as
$H^i_I(M):=H^i(\Cech^\bullet(\underline{f};A)\otimes_A M)$,
where $\Cech^\bullet(\underline{f};A)$ is the {\v C}ech complex on a set of generators $\underline{f}=f_1,\ldots,f_\ell$ of $I$. We note that  $H^i_I(M)$ does not depend on the choice of generators of $I$.
Moreover, it only depends on the radical of $I$.
We recall that the $i$-th local cohomology functor $H^i_I(-)$ can also be defined as the $i$-th right derived functor of $\Gamma_I(-)$, where $\Gamma_I(M) = \{m \in M \mid I^n m = 0$ for some $n \in \NN\}$. 
If $I = \m$ is a maximal ideal  and $M$ is finitely generated, then $H^i_\m(M)$ is Artinian.

If  $M = \bigoplus_{\frac{n}{p^e} \in \frac{1}{p^e} \ZZ} M_{\frac{n}{p^e}}$ is a $\frac{1}{p^e}\NN$-graded $R$-module, 
and we let $A_+=\bigoplus_{n> 0} A_n$, then $H^i_{A_+}(M)$ is a $\frac{1}{p^e}\ZZ$-graded $A$-module. Moreover, $[H^i_{A_+}(M)]_{\frac{n}{p^e}}$ is a finitely generated $A_0$-module for every $n \in \ZZ$, and $H^i_{A_+}(M)_{\frac{n}{p^e}} = 0$ for $n \gg 0$ \cite[Theorem 16.1.5]{BroSharp}. 
 We define the {\it $a_i$-invariant of $M$} as
$$
a_i(M)=\max\left\{ \frac{n}{p^e} \; \bigg| \;  [H^i_{A_+}(M)]_{\frac{n}{p^e}} \neq 0\right\}
$$
if $H^i_{A_+}(M)\neq 0$, and $a_i(M)=-\infty$ otherwise.

\begin{remark} 
Given a finitely generated $\ZZ$-graded $A$-module $M$, we have that $a_i(M^{1/p^e}) = a_i(M)/p^e$ for all $i \in \NN$. In fact, $H^i_{A_+}(M^{1/p^e}) \cong H^i_{A_+}(M)^{1/p^e}$ since the functor $(-)^{1/p^e}$ is exact. 
\end{remark}



\begin{remark} \label{FpureNeg} 
If $A$ is an  $F$-split $\NN$-graded ring, then $a_i(A) \ls 0$ for all $i \in \ZZ$ \cite[Lemma $2.3$]{HRFpurity}. 
\end{remark}

Given a finitely generated $\ZZ$-graded $A$-module, the {\it Castelnuovo-Mumford regularity of $M$} is defined as
$$
\reg(M)=\max\{a_i(M)+i \mid i \in \NN\}.
$$

\begin{remark}\label{remRegPR}
If $A=A_0[x_1,\ldots, x_r]$ is a polynomial ring over $A_0$, such that $x_i$ has degree $d_i>0$ for every $1\ls i\ls r$, then $\reg(A)= r-\sum_{i=1}^r d_i$.
\end{remark}



 \subsection{Filtrations and blowup algebras}\label{subBA}
Let $R$ be a commutative ring. We say that a sequence of ideals $\{I_n\}_{n\in \NN}$ of $R$ is a {\it filtration} if $I_0=R$, $I_{n+1}\subseteq I_n$ for every $n\in \NN$, and $I_nI_m\subseteq I_{n+m}$ for every $n,m\in \NN$.

\begin{definition} \label{blowup} 
Let $R$ be a ring. Consider the following graded algebras associated to a filtration  $\II = \{I_n\}_{n \in \NN}$:
\begin{enumerate}
\item[(i)] The the {\it Rees algebra} of $\II$: $\R(\II)=\bigoplus_{n\in \NN}I_{n}T^n \subseteq R[T]$, where $T$ is a variable.
\item[(ii)] The {\it  associated graded algebra} of $\II$: $\gr(\II)=\bigoplus_{n\in \NN}I_{n}/I_{n+1}$.
\end{enumerate} 
\end{definition}
We generally refer to the above as the blowup algebras associated to the filtration $\II$ \cite{Vasconcelos}.

If the Rees algebra is Noetherian, we can compute the dimensions of the blowup algebras. We show this in the next proposition.

\begin{proposition}\label{dimOfblow}
Assume that the ideal $I_1$ has positive height and that $\R(\II)$ is finitely generated as an $R$-algebra. Then, $\dim(\R(\II))=\dim(R)+1$ and $\dim(\gr(\II))=\dim(R).$
\end{proposition}
\begin{proof}

Consider the extended Rees algebra $B:=R[\II T,T^-1 ]=\oplus_{n\in \ZZ} I_n T^n$ where $I_n=R$ for $n\ls0$. Since $\R(\II)$ is Noetherian, there exists $\ell \in \ZZ_{>0}$ such that
\begin{equation}\label{eventSG}
 I_{n+\ell}=I_{\ell}I_{n} \text{ for every } n\gs \ell \, \text{ \cite[Remark 2.4.3]{ratliff1979notes}.}
\end{equation}
Thus,   $B$ is an integral extension of $R[I_\ell T,T^{-1}]=\oplus_{n\in \ZZ} I_\ell^n T^n$, and $\R(\II)$ is an integral extension of $R[I_\ell T]=\oplus_{n\in \ZZ} I_\ell^n T^n$, and hence  they both have  dimension $\dim(R) +1$ \cite[Theorem 2.2.5, Theorem 5.1.4(1)(2)]{huneke2006integral}. Now, $T^{-1}$ is  homogeneous regular element of $B$, thus  $B/(T^{-1})\cong \gr(\II)$ has dimension $\dim(R)$, finishing the proof.
\end{proof}


\section{Generators of defining equations of $F$-split blowup algebras}


This section is devoted to find bounds for the degrees of the defining equations of the algebras introduced in Definition \ref{blowup} when they are  $F$-split. The main results of this section are  Theorems \ref{defEqblowup1} and \ref{defEqblowup2}.  



Let $A=\oplus_{n\gs 0}A_n$ be a $\NN$-graded Noetherian ring. Given a finitely generated graded $A$-module $M=\oplus_{n\in \ZZ}M_n$, we let  
$$\beta_A(M)=\inf\{i\mid M = A \cdot (\oplus_{n\ls i}M_n)\},$$
that is, the largest degree of a minimal homogeneous generator of $M$. The following lemma states an upper bound on $\beta_A(M)$  in terms of the Castelnuovo-Mumford regularity. This statement can be found in the literature when $A$ is generated as an algebra in degree one \cite[Theorem $16.3.1$]{BroSharp}, or  when $A_0$ is a field \cite[Theorem $3.5$]{DalzottoSbarra} \cite[Theorem $2.2$]{CutkoskyKurano}. While the same result in our setup may be well-known to experts, we could not find a reference in the literature. We include its proof here  for the sake of completeness.

\begin{proposition}\label{PropDegGenReg}
Let $A=\oplus_{n\gs 0}A_n$ be a $\NN$-graded Noetherian ring.
Let  $d_1,\ldots, d_r>0$ be the generating degrees of $A$ as an $A_0$-algebra. Let $M$ be a finitely generated $\ZZ$-graded $A$-module. Then, $ \beta_A(M)\ls \reg(M) + \sum_{i=1}^r (d_i-1)$.
\end{proposition}
\begin{proof}
Let  $f_1,\ldots,f_r$  be homogeneous  generators of $A$ as an $A_0$-algebra  of degree $d_1,\ldots, d_r$, respectively. Without loss of generality, we may assume that $1 \leq d_1\leq \ldots \leq d_r$. Observe that ${A_+} = \bigoplus_{n>0} A_n = (f_1,\ldots,f_r)$. 

We now proceed by induction on $\sum_{i=1}^r d_i \geq r$. The base case $d_i=1$ for all $1 \leq i \leq r$ is known \cite[Theorem 16.3.1]{BroSharp}. Let $A'=A[y]$, where $\deg(y)=1$ and $M'=M\otimes_A A'$. Since $y$ is regular on $M'$, a standard argument via the long exact sequence of local cohomology of $$0\to M'(-1)\xrightarrow{y} M'\to M'/yM\cong M\to 0$$ shows that $\reg(M) = \reg(M')$, where   the regularity of $M'$ is computed with respect to the ideal ${A_+}' =\oplus_{n>0} A'_n$. 
We observe that $f=f_r - y^{d_r}$ is a homogeneous element of degree $d_r$ which is regular on $M'$. The short exact sequence $$0 \to M'(-d_r) \xrightarrow{f} M' \to M'/fM' \to 0$$ gives  $\reg(M'/fM') \leq \max\{\reg(M')+d_r-1,\reg(M')\} = \reg(M')+d_r-1 \ls \reg(M)+d_r-1$. Note that $M'/fM'$ is an  $A'/fA'$-module  and that $A'/fA' \cong A_0[f_1,\ldots,f_{r-1},y]$. Since  $\sum_{i=1}^{r-1} d_i+1 < \sum_{i=1}^{r} d_i$,  by induction, we have that $$\beta_{A'/fA'}(M'/fM') \leq \reg(M'/fM') + \sum_{i=1}^{r-1}(d_i-1) \leq \reg(M) +\sum_{i=1}^{r}(d_i-1).$$ 
Let $N$ be the $A'$-submodule of $M'$ generated by elements of degree at most $\reg(M) +\sum_{i=1}^{r}(d_i-1)$. We have just shown that $M'=N+fM'$, and therefore $M'=N + \m' M'$, where $\m' = \m_0 + {A_+}'$. Thus, from the graded Nakayama's lemma it follows that $M'=N$. In particular, $\beta_{A'}(M') \leq \reg(M) + \sum_{i=1}^r (d_i-1)$. Since $\beta_{A'}(M') = \beta_A(M)$, the proof is complete.
\end{proof}

We need one more lemma before stating the main result of this section. 

\begin{lemma}\label{lemmaDefEq}
Let $A=\oplus_{n\gs 0}A_n$ be a Noetherian $F$-finite and $F$-split graded ring.
Let  $d_1,\ldots, d_r$ be the generating degrees of $A$ as an $A_0$-algebra. Then the defining equations of $A$ over $A_0$ have degree at most  
\[
\dim(A)+\sum^r_{i=1}d_i - \max\{ \dim(A), r\}.
\]
\end{lemma}
\begin{proof}
Let  $f_1,\ldots,f_r$  be homogeneous  generators of $A$ as an $A_0$-algebra  of degree $d_1,\ldots, d_r$, respectively.  
Let $S=A_0[y_1,\ldots, y_r]$ be a polynomial ring over $A_0$ with $\deg(y_i)=d_i$ for  $1\ls i\ls r$, and let $\phi: S\to A$ be the graded $A_0$-algebra homomorphism defined by
  $\phi(y_i)=f_i$ for  $1\ls i\ls r$.  Let $\cI=\Ker(\phi)$. 
  
Set $S_+=(y_1,\ldots, y_r)\subseteq S$ and consider the  homogeneous  short exact sequence: $$
 0\to \cI\to S\to A\to 0.
 $$
 From the long exact sequence of local cohomology modules with support in $S_+$, we obtain 
$H^i_{S_+} (A)\cong H^{i+1}_{S_+}(\cI)$ for $i\leq r-2$, and an exact sequence
$$
0\to H^{r-1}_v (A)\to
H^{r}_{S_+} (\cI )\to
H^{r}_{S_+} (S)\to
H^{r}_{S_+} (A) \to 0.
$$ 

Since $a_i(A) \leq 0$ for every $i \in \NN$ by Remark \ref{FpureNeg}, and $a_r (S)=-\sum^r_{i=1}d_i$ 
by Remark \ref{remRegPR}, 
we have that $a_i(\cI)\leq 0$ 
for every $i\in \NN$.
Thus, 
\[
\reg(\cI)=\max\{a_i(\cI)+i\}\leq \min\{r,\dim(A)\},
\]
as $a_i(\cI)=-\infty$ for $i>\min\{r,\dim(A)\}$ \cite[Theorems 3.3.1 and 6.1.2]{BroSharp}. The result now follows by Proposition \ref{PropDegGenReg}, after performing some easy calculations.
\end{proof}

The following is the main theorem of this section, as it provides bounds for the degrees of generators of defining equations of  $F$-split blowup algebras.

\begin{theorem}\label{defEqblowup1}
Let $R$ be a Noetherian $F$-finite and  $F$-split ring  of  characteristic $p>0$.
Let $\II=\{I_n\}_{n\in \NN}$ be a filtration such that  $\R(\II)$ is a finitely generated  $F$-split $R$-algebra. Let $e_1,\ldots, e_\ell$ be the generating degrees of $\R(\II)$ as an $R$-algebra, that is, $\R(\II)=R[I_{e_1}T^{e_1}, \ldots,  I_{e_\ell}T^{e_\ell}]$, and let $v_1,\ldots,v_\ell$ be the number of generators of $I_1,\ldots,I_\ell$, respectively. Further assume that $I_1$ has positive height. The defining equations of $\R(\II) = \bigoplus_{n \geq 0} I_nT^n$ over $R$ have degree at most  
\[
\dim(R)+1+ \sum^\ell_{i=1}e_i v_i- \max\left\{ \dim(R)+1,\sum_{i=1}^\ell v_i\right\}.
\]
 Moreover, if $\gr(\II)$ is  $F$-split, then 
 the defining equations of $\gr(\II) = \bigoplus_{n \geq 0} I_n/I_{n+1}$ over $R/I_1$ have degree at most  
 \[
 \dim(R)+\sum^\ell_{i=1}e_iv_i - \max\left\{\dim(R), \sum_{i=1}^\ell v_i\right\}.
 \]
 \end{theorem}
\begin{proof} 
This follows from Lemma \ref{lemmaDefEq} and Proposition \ref{dimOfblow}.
\end{proof}

\begin{theorem}\label{defEqblowup2}
Let $K$ be an $F$-finite field, and $R$ be an  $F$-split graded $K$-algebra, generated over $K$ by $\mu$ elements of degree one.
Let $\II=\{I_n\}_{n\in \NN}$ be a filtration such that  $\R(\II)$ is a finitely generated  $F$-split $R$-algebra. Let $e_1,\ldots, e_\ell$ be the generating degrees of $\R(\II)$ as an $R$-algebra, that is, $\R(\II)=R[I_{e_1}T^{e_1}, \ldots,  I_{e_\ell}T^{e_\ell}]$. Set $w_i=\beta_R(I_{e_i})$ for  $1\ls i\ls \ell$ and  let $\nu_1,\ldots, \nu_\ell$  be the number of generators of $I_1,\ldots, I_\ell$ respectively. 
 Further assume that $I_1$ has positive height. The defining equations of $\R(\II)$ over $K$ have total degree at most  
  \[
\dim(R) + 1 + \mu + \sum^\ell_{i=1}\nu_i(w_i+e_i) - \max\left\{\dim(R)+1, \mu+\sum_{i=1}^\ell v_i \right\}.
 \]
Moreover, if $\gr(\II)$ is  $F$-split, then 
 the defining equations of $\gr(\II)$ over $K$ have total degree at most  
  \[
\dim(R) + \mu + \sum^\ell_{i=1}\nu_i(w_i+e_i) - \max\left\{\dim(R), \mu+\sum_{i=1}^\ell v_i \right\}.
 \]
 \end{theorem}
\begin{proof} 
Both parts of the result follow from Lemma \ref{lemmaDefEq} and Proposition \ref{dimOfblow}.
\end{proof}


\section{$F$-split filtrations }

Throughout this section  we assume the following setup.

\begin{setup}\label{setupFpure}
Let $R$ be a Noetherian $F$-finite and  $F$-split ring  of  characteristic $p>0$ which is either local or $\NN$-graded. In the local case, we let $\m$ denote its unique maximal ideal, and $K=R/\m$ its residue field. In the graded case, we assume that $R=\bigoplus_{n \geq 0} R_0$ is a finitely generated $R_0$-algebra, where $(R_0,\m_0)$ is a local ring. We let $R_+ = \bigoplus_{n >0} R_n$, and $\m=\m_0+R_+$. We further assume that $R$ generated in degree one, that is, $R=R_0[R_1]$. In the graded case, every object we consider is homogeneous with respect to the given grading.
\end{setup}

\subsection{$F$-split filtrations of ideals} 

We introduce the main object of study of this paper:  {\it  $F$-split filtrations}. For a related notion in the case of ordinary powers, see \cite{LauritzenThomsen}.

\begin{definition}\label{DefFpureFilt}
Assume Setup \ref{setupFpure}. We say that a sequence of $R$-ideals $\II=\{I_n\}_{n\in \NN}$ is an {\it  $F$-split filtration} if  $I_0=R$, $I_{n+1}\subseteq I_n$ for every $n\in \NN$, $I_nI_m\subseteq I_{n+m}$ for every $n,m\in \NN$, and there exists a splitting $\phi:R^{1/p}\to R$ such that
$\phi\big((I_{np+1})^{1/p}\big)\subseteq I_{n+1}$ for every $n\in \NN$.  
\end{definition}

We now study properties regarding $F$-splittings, depth and regularity for ideals appearing in  these filtrations (see  Theorem  \ref{mainCharP} and Theorem \ref{mainRegDepth}). In particular, we will show that $F$-split filtrations yield  $F$-split blowup algebras.

\begin{remark} 
Suppose that $\phi:R^{1/p}\to R$ is a surjective map such that
$\phi\big((I_{np+1})^{1/p}\big)\subseteq I_{n+1}$ for every $n\in \NN$.
Let $g\in R$ such that $\phi (g^{1/p})=1$. Then, $\varphi(-)=\phi(g^{1/p}-)$ induces a splitting such that
$\varphi\big((I_{np+1})^{1/p}\big)\subseteq I_{n+1}$ for every $n\in \NN$. Then, it suffices to assume that $\phi$ is surjective in Definition \ref{DefFpureFilt}
\end{remark}

\begin{remark}\label{symbFpareFp} 
We observe that if $\II=\{I_n\}_{n\in \NN}$ is an  $F$-split filtration, then $R/I_1$ is  $F$-split. In fact, by considering $n=0$ in Definition \ref{DefFpureFilt}, one gets an induced splitting $\phi:(R/I_1)^{1/p} \to R/I_1$. In particular, $I_1$ is a  radical ideal.
\end{remark}

\begin{proposition}\label{from1toall}
Assume Setup \ref{setupFpure} and  let $\II = \{I_n\}_{n \in \NN}$  be a filtration. The following statements are equivalent:
\begin{enumerate}
\item $\II$ is an  $F$-split filtration.
\item There exists a splitting $\phi_e:R^{1/p^e}\to R$ such that
$\phi_e\big((I_{np^e+s})^{1/p^e}\big) = I_{n+1}$ for every $e\in \ZZ_{>0}$, $n\in\NN$, and $s\in \ZZ$ such that $1\ls s\ls p^e$.
\item There exists a splitting $\phi_e:R^{1/p^e}\to R$ such that
$\phi_e\big((I_{np^e+1})^{1/p^e}\big)\subseteq I_{n+1}$ for some $e\in \ZZ_{>0}$ and every $n\in\NN$.
\end{enumerate}
\end{proposition}
\begin{proof}
We consider the implication $(1)\Rightarrow (2)$. Let $\phi$ be as in Definition \ref{DefFpureFilt}.
For every $j>0$ we consider the $R$-linear map $\varphi_j:R^{1/p^j}\to R^{1/p^{j-1}}$ defined as $\varphi_j(r^{1/p^j})=(\phi(r^{1/p}))^{1/p^{j-1}}$.
We observe that 
$\varphi_j\big((I_{np^j+s})^{1/p^j}\big)\subseteq \varphi_j\big((I_{np^j+1})^{1/p^j}\big)\subseteq (I_{np^{j-1}+1})^{1/p^{j-1}}$ for every $n\in \NN$ and $j,s\in \ZZ_{>0}$ .
Then,  we have
$$
\varphi_1\circ \varphi_{2}\circ\cdots \circ \varphi_e\big((I_{np^e+s})^{1/p^e}\big)\subseteq I_{n+1}
$$
for every  $e>0$, $n\geq 0$, and $s>0$. Set $\phi_e := \varphi_1\circ \varphi_{2}\circ\cdots \circ \varphi_e$. It remains to show that $I_{n+1}\subseteq \phi_e\big((I_{np^e+s})^{1/p^e}\big) $ for $s\ls p^e$.
But this inclusion follows by noticing that
  $$I_{n+1}\subseteq \phi_e\big(I_{n+1}R^{1/p^e}\big)\subseteq \phi_e\big((I_{np^e+s})^{1/p^e}\big)$$
for $s\ls p^e$.

Since $(2)\Rightarrow (3)$ is clear, it remains to show the implication $(3)\Rightarrow (1)$. 
We consider the natural inclusion $\iota:R^{1/p}\to R^{1/p^e}$ and set $\phi:= \phi_e\circ\iota$.
We note that, 
$\iota\big( (I_{np+1})^{1/p}\big)\subseteq (I_{np^e+p^{e-1}})^{1/p^e}\subseteq (I_{np^e+1})^{1/p^e}$.
As a consequence we have that
$$
\phi\big( (I_{np+1})^{1/p}\big) \subseteq \phi_e\big( (I_{np^e+1})^{1/p^e}\big) \subseteq  I_{n+1}$$ for every $n\in \NN$, and the result follows.
\end{proof}

For ideals in a regular ring    $R$, we state an effective criterion for  $F$-split filtrations  analogous to the classical one by Fedder \cite{FedderFputityFsing}.
\begin{proposition}\label{PropFedder} Assume Setup \ref{setupFpure} with $R$ regular. In the graded case, we further assume that $R_0$ is a field, so that $\m=R_+$.  
We have that
$\II=\{I_n\}_{n\in \NN}$ is an  $F$-split filtration if and only 
$$\bigcap_{n\in\NN} \big(( I_{n+1})^{[p]}:_R I_{np+1}\big) \not\subseteq\m^{[p]}.$$
\end{proposition}
\begin{proof}
Since $R$ is regular, 
we can pick the trace $\Phi$, which is a generator of $\Hom_R(R^{1/p},R)$ as a free $R^{1/p}$-module described, see Remark \ref{Trace}. Then, for $f\in R$ and $\phi:=f^{1/p} \cdot \Phi=\Phi(f^{1/p}-)$ we have
$\phi\big((I_{np+1})^{1/p}\big)\subseteq I_{n+1}$  for every $n\in \NN$
if and only if
$$f\in\bigcap_{n\in\NN} \left(( I_{n+1})^{[p]}:_R I_{np+1}\right),$$
by Remark \ref{Trace}.  
In addition,
 $\phi$ is surjective if and only if $f \not\in \m^{[p]}$, and the result follows.
\end{proof}

\subsection{ $F$-split blowup algebras} In this subsection we obtain the first  significant property for $F$-split filtrations. In the following theorem  we prove  that if $\II$ is an  $F$-split filtration,  then the algebras in Definition \ref{blowup} are  $F$-split. This is one of the main motivations for introducing  $F$-split filtrations.

\begin{theorem}\label{mainCharP}
Assume Setup \ref{setupFpure}. If $\II$ is an  $F$-split filtration,
then $\R(\II)$ and $\gr(\II)$ are  $F$-split.
\end{theorem}
\begin{proof}
Let $\phi$ be  such that $\II$ is   $F$-split with respect to $\phi$, see Definition \ref{DefFpureFilt}.
We note that $\R(\II)$ is reduced. Then, we can consider the ring of $p$-th roots $\R(\II)^{1/p}=\bigoplus_{n \in \NN}(I_{n})^{1/p}T^{n/p}$. 
We define $\varphi:\R(\II)^{1/p}\to  \R(\II)$ as the homogeneous homomorphism of $\R(\II)$-modules  induced  by
 $\varphi(r^{1/p} T^{n/p}) =\phi(r^{1/p})T^{n/p}$ if $p$ divides $n$, and $\varphi(r^{1/p} T{n/p}) =0$ otherwise. 
The map $\varphi$ is well-defined because
$$
\phi \big(( I_{(n+1)p} )^{1/p}\big) \subseteq 
\phi \big(( I_{np+1} )^{1/p}\big) \subseteq   I_{n+1} 
$$
for every $n\in \NN$, and it is $\R(\II)$-linear since $\phi$ is $R$-linear.
If $r\in I_{n}\subseteq ( I_{np})^{1/p}$, then $\varphi(rt^{np/p})=\phi(r) T^n= rT^n$ because $\phi$ is 
a splitting. We conclude that $\varphi$ is a splitting of the inclusion $\R(\II) \to \R(\II)^{1/p}$, and hence $\R(\II)$ is  $F$-split.

Consider the ideal $\cJ=\bigoplus_{n\in \NN} I_{n+1}T^n\subseteq  \R(\II)$. Since 
$\phi\big( ( I_{np+1})^{1/p}\big)\subseteq I_{n+1}$ for every $n\in\NN$, we obtain
$\varphi\big(( I_{np+1})^{1/p} T^{np/p}\big)\subseteq  I_{n+1} T^n$ for every $n\in\NN$.
Therefore, $\varphi(\cJ^{1/p})\subseteq \cJ$. This induces a splitting $\overline{\varphi}: (\R(\II)/\cJ)^{1/p} \to \R(\II)/\cJ$ and then $\R(\II)/\cJ \cong \gr(\II)$ is also  $F$-split. 
\end{proof}

\begin{remark}\label{remOnlyR(I)F-pure}
We remark that in order to prove that $\R(\II)$ is  $F$-split, we only need a splitting $\phi':R^{1/p}\to R$ such that $\phi'\big((I_{(n+1)p})^{1/p}\big)\subseteq I_{n+1}$ for every $n\in \NN$. The stronger requirement in the definition of  $F$-split filtrations is to ensure that $\gr(\II)$ is  $F$-split as well.
\end{remark}


\subsection{Depth and regularity of  $F$-split filtrations}

In this subsection we study the asymptotic behavior the depth and Castelnuovo-Mumford regularity of  $F$-split filtrations. We assume Setup \ref{setupFpure}. In the graded case, by depth of a graded $R$-module we mean its grade with respect to the maximal ideal $\m$, i.e., the length of a maximal regular sequence for $M$ inside $\m_0 + {R_+}$. On the other hand, by Castelnuovo-Mumford regularity we mean the regularity computed with respect to the ideal ${R_+}$.

In  Theorem \ref{mainRegDepth} we show that the sequences $\{\depth(I_n)\}_{n\in \NN}$ and $\{\frac{\reg(I_n)}{n}\}_{n\in \NN}$ converge to a limit under mild assumptions. We begin with the following technical result.

 \begin{proposition}\label{PropSymbFpureDepthAinv}
Assume Setup \ref{setupFpure}, and let $\II = \{I_n\}_{n \in \NN}$ be an  $F$-split filtration. Then
 \begin{enumerate}
  \item $\depth(I_n)\ls \depth(I_{\lceil\frac{n}{p^e}\rceil})$ for every $n,\, e\in \NN$.
 \item If $R$ is graded, then $a_i(I_{n})\gs p^ea_i(I_{\lceil\frac{n}{p^e}\rceil})$ for every $n,\, e\in \NN$ and $0\ls i\ls \dim(R/I_1)$. 
 \end{enumerate}
 \end{proposition}
\begin{proof}
By Proposition \ref{from1toall} the natural map $\iota :I_{n+1}\to (I_{np^e+s})^{1/p^e}$ splits for every $n,\, e\in \NN $ and $1\ls s\ls p^e$ via a splitting $\phi_e$. 
Therefore,  the module $\HH{i}{\m}{I_{n+1}}$ is a direct summand of $\HH{i}{\m}{(I_{np^e+s})^{1/p^e}}$ for every $1\ls i\ls\dim(R/I_1)$. 
We note that  
\begin{equation*}\label{eqds}
\HH{i}{\m}{(I_{np^e+s})^{1/p^e}}=\left( \HH{i}{\m}{ I_{np^e+s}}\right)^{1/p^e}.
\end{equation*}
Thus, 
$\HH{i}{\m}{I_{np^e+s}}=0$
implies that  $\HH{i}{\m}{(I_{np^e+s})^{1/p^e}}=0$, and hence $\HH{i}{\m}{R/I_{n+1}}=0$.
Therefore, we have that $\depth(I_{n+1})\gs \depth(  I_{np^e+s})$, which proves the first part.

We note that  $\HH{i}{{R_+}}{I_{n+1}}$ is also a direct summand of $\HH{i}{{R_+}}{(I_{np^e+s})^{1/p^e}}$ and 
\begin{equation*}
\HH{i}{{R_+}}{(I_{np^e+s})^{1/p^e}}=\left( \HH{i}{{R_+}}{ I_{np^e+j}}\right)^{1/p^e},
\end{equation*}
thus we obtain
$$
a_i(I_{n+1})\ls a_i \left( (I_{np^e+s})^{1/p^e}\right)=\frac{1}{p^e} a_i \left(  I_{np^e+s}\right),
$$
and the second part follows.
\end{proof}

The following is the main result  of this section.

\begin{theorem}\label{mainRegDepth}
Assume Setup \ref{setupFpure}, and let $\II$ be an  $F$-split filtration such that $\R(\II)$ is Noetherian. Then
\begin{enumerate}
\item  $ \depth(I_{n})$ stabilizes  and the stable value  is equal to $\min\{ \depth(I_{n})\}$.
\item If $R$ is graded, then $\lim\limits_{n\to\infty} \frac{\reg(I_{n})}{n}$ exists. As a consequence, if $R$ is regular, then $\lim\limits_{n\to\infty} \frac{\reg(R/I_{n})}{n}$ exists.
\end{enumerate}

\end{theorem}
\begin{proof}

We begin with (1). Since $\R(\II)$ is Noetherian, there exists $\ell \in \ZZ_{>0}$ such that 
$ I_{n+\ell}=I_{\ell}I_{n}$  for every  $n\gs \ell$  \cite[Remark 2.4.3]{ratliff1979notes}. 
 Hence, for every $j=0,\ldots, \ell-1$, there exist
$d_j,\ell_j\in\NN$ such that 
$$\depth(I_{(n+1)\ell+j})=\depth((I_{\ell})^n I_{\ell+j})=d_j,$$
for $n\gs \ell_j$
  \cite[Theorem 1.1]{HeHi}.

 Let $\delta=\min\{\depth(I_{n})\}_{n\in\NN}$ and fix $s\in\ZZ_{>0}$  such that
 $\delta=\depth(I_{s})$. Let $q=p^e$ be such that $q>\ell$, and $q(s-1)>(\ell_j+1)\ell$ for every $j=0,\ldots, \ell-1$.
From Proposition \ref{PropSymbFpureDepthAinv} it follows that 
$$\depth(I_{q(s-1)+i})\ls \depth(I_{s})$$
for every $i=1,\ldots, q$. 
By our choice for $q$, for each $j=0,\ldots,\ell-1$ there exist  natural numbers $m\gs \ell_j+1$ and $1\ls i\ls q$
such that $q(s-1)+i=m\ell+j$.
Then,
\begin{align*}
\delta=\depth(I_{s}) & \gs \depth(I_{q(s-1)+i})\\
&=\depth(I_{mt+j} )=d_j.
\end{align*}
We conclude that $\delta=d_j$ for every $j=0,\ldots,\ell-1$ because $\delta$ is the minimum depth. 
Then,  
$$
\depth(I_{n})=\delta
, \mbox{ for } n\gg 0.$$

We proceed to prove (2). Since  $\R(\II)$ is  Noetherian and by Equation \eqref{eventSG}, the sequence $\reg(I_{n})$ eventually agrees with a linear quasi-polynomial \cite[Theorem 3.2]{Trung_Wang}. Then, 
 there exists $w\in\NN$ and $c_1,\ldots,c_w,b_1,\ldots,b_w\in\NN$
such that 
$
\reg(I_{n})=c_j n+b_j$  for $ n\equiv j \pmod w
$
and $n\gg 0$.
We want to show that $c_1=\ldots=c_w$.
Set $\alpha_{n} =\max\{ a_i (I_{n}) \}$ and notice that  $\alpha_n\neq -\infty$ for every $n\in\NN$. 
We have that 
$$\lim\limits_{m\to \infty} \frac{\alpha_{wm+j}}{wm+j}=c_j$$
for every $j=0,\ldots, w-1$. 
We fix $i,j\in \{1,\ldots, w\}$, and $e\in \NN$  such that $q=p^e>w$.
Fix $\varepsilon\in \RR_{>0}$ and  let $r\in \NN$ be such that $c_j-\frac{\alpha_{wm+j}}{wm+j}<\varepsilon$ for every $m\gs r$. 
From Proposition \ref{PropSymbFpureDepthAinv}, we obtain 
$$
\alpha_{wr+j}\ls \frac{\alpha_{q^\theta (wr+j-1)+b}}{q^\theta}
$$
for every $\theta\in\NN$ and $b=1,\ldots,q^\theta.$ 
Then,
$$
c_j-\varepsilon\ls \frac{\alpha_{wr+j}}{wr+j}\ls \frac{\alpha_{q^\theta (wr+j-1)+b}}{q^\theta (wr+j)}\ls
 \frac{\alpha_{q^\theta (wr+j-1)+b}}{q^\theta (wr+j-1)+b}.
$$
Since this inequality holds for every $b=1,\ldots,q^\theta$ and since $w<q^\theta$,  
there exists infinitely many pairs $\theta, b$ such that $q^\theta (wr+j-1)+b\cong i \pmod w$.
We conclude that $c_j-\varepsilon \ls c_i$ for every $\epsilon,$ and then $c_j\ls c_i$.
Since $i,j$ were chosen arbitrarily, we conclude that $c_1=\ldots=c_w$. 
\end{proof}

Under some extra assumptions we can say more about the stable value  $\lim\limits_{n\to\infty} \depth(I_{n})$.

\begin{corollary}\label{CM}
Assume Setup \ref{setupFpure}, and let $\II$ be an  $F$-split filtration such that $\R(\II)$ is a Noetherian  Cohen-Macaulay algebra. Then, $$\lim\limits_{n\to\infty} \depth(I_{n})=\dim(R)-\dim(\R(\II)/
\m \R(\II))+1.$$
\end{corollary}
\begin{proof}
There exists $\ell\in \ZZ_{>0}$ such that $\R(\ell)$ is generated in degree one as an algebra \cite[2.4.4]{ratliff1979notes}. 
If $\R=\R(\II)$ is Cohen-Macaulay, then so is $\R(\ell)=\sum_{n\in \NN}I_{n\ell}T^n$ as it is a direct summand of $\R$.  Therefore, 
\begin{align*}
\lim\limits_{n\to\infty} \depth(I_{n}) 
&= \dim(\R(\ell)) - \dim(\R(\ell)/  
\m \R(\ell))  \ \ \  \hbox{  \cite[Theorem 1.1]{HeHi}}   \\
&= \dim(R)- \dim(\R(\ell)/
\m \R(\ell))+1\\
&=\dim(R)- \dim(\R/\m \R) + 1. \qedhere
\end{align*}
\end{proof}


\section{Symbolic $F$-split ideals and  symbolic powers}\label{sectionSymbFpure}

We now focus on  $F$-split filtrations that are given by symbolic powers. We recall that, given a ring $R$ and an ideal $I \subseteq R$, for $n \in \NN$ the {\it $n$-th symbolic power of $I$} is defined as $I^{(n)} = I^nR_W$, where $W$ is the complement of the union of the minimal primes of $I$.

Throughout this section we assume Setup \ref{setupFpure} with $R$ regular. In the graded case, we further assume that $R_0$ is a field, i.e., $R$ is standard graded. We recall it here more explicitly for future reference:

\begin{setup}\label{setupSymbFpure}
Let $(R,\m,K)$ be an $F$-finite regular  ring  of  characteristic $p>0$ which is either local, or $R= \bigoplus_{n \geq 0} R_n = K[R_1]$ is a standard graded polynomial ring over the field $K$, with homogeneous maximal ideal $\m =  \bigoplus_{n > 0} R_n$.  We denote by
$$\R^s(I)=\bigoplus_{n \in \NN}I^{(n)}T^n\qquad \text{and}\qquad \gr^s(I)=\bigoplus_{n\in \NN}I^{(n)}/I^{(n+1)},$$
the {\it symbolic Rees algebra} and {\it symbolic associated graded algebra} of $I$, respectively. We also let
$$\R(I)=\bigoplus_{n \in \NN}I^{n}T^n\qquad \text{and}\qquad \gr(I)=\bigoplus_{n\in \NN}I^{n}/I^{n+1},$$
be the  {\it Rees algebra} and {\it \ associated graded algebra} of $I$, respectively.
\end{setup}

The interest in symbolic blowups has significantly increased, even in recent years \cites{Roberts,GrifoSeceleanu}.


\begin{definition}\label{DefSymbolicFpure}
We say that an ideal $I$ is {\it symbolic $F$-split} if $\II = \{I^{(n)}\}_{n \in\NN}$ is an  $F$-split filtration.
\end{definition}

We start by studying equality between ordinary and symbolic powers for symbolic $F$-split ideals. We use the  following remark to study symbolic powers.

\begin{remark} \label{remark torsion free} Assume Setup \ref{setupSymbFpure}. We note that $\gr(I)$ is torsion-free over $R/I$ if and only if $I^n=I^{(n)}$ for every $n\in\NN$. In fact, if $\gr(I)$ is torsion-free over $R/I$, then $\Ass_R (I^n/I^{n+1})\subseteq \Ass_R (R/I)$ for every $n\in \NN$. From $0\to I^n/I^{n+1}\to R/I^{n+1}\to R/I^{n}\to 0$ we obtain $\Ass_R (R/I^{n+1})\subseteq \Ass_R (I^n/I^{n+1})\cup \Ass_R (R/I^{n})$. Therefore, proceeding by induction on $n$  we obtain $\Ass_R (R/I^{n+1})\subseteq \Ass_R (R/I)$  for every $n\in \NN$, which implies $I^n=I^{(n)}$ for every $n\in\NN$. Conversely, if  $I^n=I^{(n)}$ for every $n\in\NN$, then $\Ass_R (I^n/I^{n+1})\subseteq \Ass_R (R/I^{n+1})=\Ass_R (R/I)$ for every $n\in\NN$. This implies that $\Ass_R (\gr(I)) \subseteq \Ass_R (R/I)$, i.e., $\gr(I)$ is torsion-free over $R/I$. 
\end{remark}

We can now rephrase a result due to Huneke, Simis and Vasconcelos as follows:
\begin{lemma}\cite[Corollary 1.10]{HSV89} \label{reducedAndTF}
Assume Setup \ref{setupSymbFpure}. 
Let $I\subseteq R$ be a radical ideal. Then $I^n=I^{(n)}$ for every $n\in\NN$ if and only if the associated graded algebra
 $ \gr (I)$
is reduced.
\end{lemma}

\begin{proposition}\label{PropOrdSymFpure}
Assume Setup \ref{setupSymbFpure}. Let $I\subseteq R$ be a symbolic $F$-split ideal.
Then,
 $I^n=I^{(n)}$ for every $n\in\NN$ if and only if 
 $ \gr(I)$
is an  $F$-split ring.
\end{proposition} 
\begin{proof}
If $I^n=I^{(n)}$ for every $n\in\NN$, then $ \gr(I)=\gr^s_{I} (R)$. Hence, 
$ \gr(I)$ is   $F$-split by Theorem \ref{mainCharP}. 
Conversely, if $ \gr(I)$ is   $F$-split, then it is reduced. Therefore,  $I^n=I^{(n)}$ for every $n\in\NN$ by Lemma \ref{reducedAndTF}.
\end{proof}

For the proof of Theorem \ref{ThmCompare} we need the following  well-known  lemma. Here we denote by $\mu(I)$ the minimal number of (homogeneous) generators of $I$.

\begin{lemma}\label{Lemma Obs pe u(I)}
Assume Setup \ref{setupSymbFpure}.  Let $I\subseteq R$ be any ideal. If $r\gs \mu(I)( p-1)+1$, 
then  $I^r= I^{r- p} I^{[p]}$.
\end{lemma}
\begin{proof}
Let $u=\mu(I)$ and $f_1,\ldots, f_u$  a minimal set of generators of
$I.$
Let $\alpha_1,\ldots,\alpha_u\in\NN$ be such that 
$\alpha_1+\ldots+\alpha_u=r,$ 
then by assumption there must exist $\alpha_i$ such that $\alpha_i\gs p.$
Therefore, 
$
f^{\alpha_1}_1\cdots f^{\alpha_u}_u=
f^{\alpha_1}_1\cdots f^{\alpha_i-p}_i f^{\alpha_u}_u\cdot f^{p}_i
\in I^{r-p}I^{[p]}. 
$
This shows that $I^r\subseteq I^{r- p} I^{[p]}$.
To obtain the other containment, we observe that $I^{[p]}\subseteq I^p$.
\end{proof}

The following result gives  a finite test to verify whether all the symbolic and ordinary powers of a symbolic $F$-split ideal coincide. 

  \begin{theorem}\label{ThmCompare}
Assume Setup \ref{setupSymbFpure}. Let $I\subseteq R$ be a symbolic $F$-split ideal.
If  $I^n=I^{(n)}$ for every $n\ls\lceil\frac{\mu(I)(p-1)}{p}\rceil,$ then $I^n=I^{(n)}$ for every $n \in \NN.$
 \end{theorem}
 \begin{proof}
 
By  Proposition \ref{PropOrdSymFpure}, it suffices to show $\gr(I)$ is  $F$-split.  Let $\phi$ be  such that $I$ is symbolic $F$-split with respect to $\phi$. Proceeding as in Theorem \ref{mainCharP} it suffices to prove $\phi \big(( I^{np+1} )^{1/p}\big) \subseteq   I^{n+1}$
for every $n\in\NN$. By assumption this inclusion holds for $n< \lceil\frac{\mu(I)(p-1)}{p}\rceil$, as for these values  $I^{(n+1)}=I^{n+1}$. We fix $n\gs  \lceil\frac{\mu(I)(p-1)}{p}\rceil$. Then, $I^{n p+1}= I^{(n-1)p+1}I^{[p]}$ by Lemma \ref{Lemma Obs pe u(I)}. The latter  is equivalent to
$( I^{np+1})^{1/p}=( I^{(n-1)p+1})^{1/p}I$.
Therefore, by induction on $n$
\[
\phi\big(  ( I^{np+1})^{1/p}   \big)
=\phi\big(( I^{(n-1)p+1})^{1/p} I   \big)
=\phi\big(( I^{(n-1)p+1})^{1/p}       \big)I
\subseteq   I^{n}I=I^{n+1}. \qedhere
\]
 \end{proof}

We continue with a version of Fedder's Criterion for symbolic $F$-split ideals. This improves Proposition \ref{PropFedder} as it only requires to verify that a finite  intersection of colon ideals is not contained in $\m^{[p]}$. We recall that the {\it big height} of an ideal $I$, denoted by $\bh(I)$, is the largest height of a minimal prime of $I$.

\begin{theorem}\label{ThmFedders}
Assume Setup \ref{setupSymbFpure}. Let $I\subseteq R$ be a radical ideal and set $H=\bh(I)$. Let $\delta=1$ if $p\ls H$ and $\delta=0$ otherwise.
Then, $I$ is symbolic   $F$-split if and only if

\begin{equation*}\label{nonCont}
\bigcap^{ \max\{0, H-1-\delta\}}_{n=0}\big(( I^{(n+1)})^{[p]}:_R I^{(np+1)}\big) \not\subseteq\m^{[p]}.
\end{equation*}

\end{theorem}
\begin{proof}

By Proposition \ref{PropFedder}, it suffices to show that 
$$\bigcap_{n\in\NN} \big(( I^{(n+1)})^{[p]}:_R I^{(np+1)}\big) = \bigcap^{ \max\{0, H-1-\delta\}}_{n=0}\big(( I^{(n+1)})^{[p]}:_R I^{(np+1)}\big).$$ 
We note that if $H = 0$, then $I=0$ and the result follows.  If $H=1,$ then $I$ is a principal ideal. Therefore,
$$
 ( I^{(n+1)})^{[p]}:_R I^{(np+1)}
=  I^{np+p}:_R I^{np+1}
 =I^{p-1} = I^{[p]}:_R I
 $$
and the result  follows.
Hence, for the rest of the proof we may assume $H\gs 2.$
Let $J$ be the ideal $( I^{(H-\delta)})^{[p]}:_R I^{((H-1-\delta)p+1)}$.
We claim that $J\subseteq \left( I^{(n+1)}\right)^{[p]}:_R I^{(np+1)}
$ for every $n\gs H-1-\delta$.
We proceed by induction on $n$. The base of induction follows from the definition of $J$.
Suppose that $J\subseteq ( I^{(n+1)})^{[p]}:_R I^{(np+1)}$, we need to show that $J I^{((n+1)p+1)}\subseteq ( I^{(n+2)})^{[p]}$, and it suffices to show this containment locally at every prime ideal in $\Ass_R  ( R/{I^{(n+2)}}^{[p]})=\Ass_R (R/ I^{(n+2)})= \Ass_R (R/I)$, where the first equality holds by flatness of Frobenius.   
Let $Q\in\Ass_R(R/I)$ and set $\widetilde{Q}=QR_Q$.
We observe that since $n+1\gs H-\delta,$ we have $(n+1)p+1\gs Hp-\delta p+1\gs H(p-1)+1$. Then,
Lemma \ref{Lemma Obs pe u(I)} implies
$$
( \widetilde{Q}^{n+2})^{[p]}:_R \widetilde{Q}^{(n+1)p+1}
=\widetilde{Q}^{[p]}( \widetilde{Q}^{n+1})^{[p]}:_R \widetilde{Q}^{[p]}\widetilde{Q}^{(np+1)}\supseteq ( \widetilde{Q}^{n+1})^{[p]}:_R \widetilde{Q}^{(np+1)}.
$$
Since $JR_Q\subseteq ( \widetilde{Q}^{n+1})^{[p]}:_R \widetilde{Q}^{(np+1)}$ by induction hypothesis, the proof of our claim follows. We conclude that
$$
\bigcap_{n\in\NN} \big(( I^{(n+1)})^{[p]}:_R I^{(np+1)}\big) =
\bigcap^{H-1-\delta}_{n=0} \big(( I^{(n+1)})^{[p]}:_R I^{(np+1)}\big),
$$
and the result follows.
\end{proof}

\begin{remark} 
The correction given by $\delta=0$ in the case $p>H$ of Theorem \ref{ThmFedders} is needed. Indeed, if we could always use $\delta=1$, that is, if the condition
$$\bigcap^{ \max\{0, H-2\}}_{n=0}\big(( I^{(n+1)})^{[p]}:_R I^{(np+1)}\big) \not\subseteq\m^{[p]}$$
implies that $I$ is symbolic $F$-split, then every  $F$-split ideal $I$ such that $\bh(I)=2$ would be symbolic $F$-split. This is not the case, as we show in Example \ref{ExFpureNotSymb}.
\end{remark}

\begin{corollary}\label{CorH}
Assume Setup \ref{setupSymbFpure}.  Let $I\subseteq R$ be a radical ideal and set $H=\bh(I)$.
If $I^{(H(p-1))}\not\subseteq \m^{[p]}$, then $I$ is symbolic $F$-split.
\end{corollary}
\begin{proof}
For every $n\in \NN$ we have 
$I^{(H(p-1))} I^{(np+1)}\subseteq I^{(H(p-1)+np+1)} \subseteq  ( I^{(n+1)})^{[p]}$ \cite[Lemma 2.6]{GrifoHuneke}. Therefore,
 $I^{(H(p-1))}\subseteq ( I^{(n+1)})^{[p]}:_R I^{(np+1)}$. The result now follows from Theorem \ref{ThmFedders}.
\end{proof}

\begin{example}\label{SqFree_Example}
Let $I$ be a square-free monomial ideal in a polynomial ring $K[x_1,\ldots, x_r]$.
We observe that 
$$
(x_1\cdots x_r)^{p-1}\in\left( \bigcap^{r} _{n=0}\big(( I^{(n+1)})^{[p]}:_R I^{(np+1)}\big)\right) \smallsetminus \m^{[p]}
$$
Then, by  Theorem \ref{ThmFedders}, 
$I$ is symbolic $F$-split. This can also be proven via monomial valuations (see Example \ref{monExamp}(2)).
\end{example}


We also obtain the following sufficient condition for  $\R^s(I)$ to be $F$-split.

\begin{proposition}\label{PropH2}
Assume Setup \ref{setupSymbFpure}.  Let $I\subseteq R$ be a radical ideal and set $H=\bh(I)$. Let $\delta'=1$ if $p\ls H-1$ and $\delta'=0$ otherwise.
If
$$
\bigcap^{H-2-\delta'}_{n=0}\big(( I^{(n+1)})^{[p]}:_R I^{((n+1)p)}\big) \not\subseteq\m^{[p]},
$$
then  $\R^s(I)$ is  $F$-split. In particular, $\R^s(I)$ is  $F$-split  if $I^{((H-1)(p-1))}\not\subseteq \m^{[p]}$.
\end{proposition}
\begin{proof}
We note that $\R^s(I)$ is  $F$-split if there is a splitting $\phi':R^{1/p}\to R$ for which  the inclusion $$\phi'\big((I^{((n+1)p)})^{1/p}\big)\subseteq I^{(n+1)}$$ holds for every $n\in \NN$ (see Remark \ref{remOnlyR(I)F-pure}).  
Let  $\delta'=1$ if $p\ls H-1$ and $\delta'=0$ otherwise.
We can adapt Lemma \ref{PropFedder} and the proof of Theorem \ref{ThmFedders} to obtain that
 $$\bigcap^{H-2-\delta'}_{n=0}\big(( I^{(n+1)})^{[p]}:_R I^{((n+1)p)}\big) \not\subseteq\m^{[p]}$$
  implies that $\R^s(I)$ is  $F$-split.  
  
For the last statement, we can proceed as in Corollary \ref{CorH} to obtain. 
$$
I^{((H-1)(p-1))}\subseteq \big(( I^{(n+1)})^{[p]}:_R I^{((n+1)p)}\big), 
$$
whence the conclusion follows.
\end{proof}

Since every symbolic $F$-split ideal is  $F$-split, one may ask whether these two conditions are equivalent.
The following  example  shows that this is not the case.

\begin{example}\label{ExFpureNotSymb}
Let $R=K[a,b,c,d]$ be a polynomial ring  and $\chara p \gs 3$. Consider the following matrix 
$$A=\begin{bmatrix}a^2&b&d\\c&a^2&b-d\end{bmatrix}.$$
Let $I=I_2(A)$ be the ideal generated by the $2\times 2$ minors of $A$. The ring $R/I$ is   $F$-split, in fact strongly $F$-regular \cite[Proposition 4.3]{singh99}. The ideal $I$ is prime of height 2. Moreover, the symbolic and ordinary powers of $I$ coincide \cite[Corollary 4.4]{GrifoHuneke}. Considering $p=3$ we verify with Macaulay2  \cite{grayson2002macaulay} that  $( I^{2})^{[3]}:_R I^{4} \subseteq\m^{[3]}$. Therefore, $I$ is not symbolic $F$-split by Theorem \ref{ThmFedders}.
\end{example}

The  following definition due to Huneke  provides a sufficient condition for an ideal to be symbolic $F$-split.

\begin{definition}[Huneke]
Assume Setup \ref{setupSymbFpure}.  Let $I\subseteq R$ be a radical ideal of height $h$. We say that $I$ is {\it $F$-K\"{o}nig} 
if there exists a regular sequence $f_1,\ldots, f_h\in I$ such that $R/(f_1,\ldots, f_h)$ is  $F$-split. 

In particular, if $I$ is an ideal generated by a regular sequence and such that $R/I$ is F-pure, then $I$ is $F$-K{\"o}nig.
\end{definition}

\begin{proposition}\label{FkoningFsymb}
Assume Setup \ref{setupSymbFpure}. If  $I\subseteq R$  is equidimensional and $F$-K\"{o}nig,
then it is symbolic $F$-split.
\end{proposition}
\begin{proof}
Let $h=\Ht(I)$  and  $f_1,\ldots, f_h\in I$ a regular sequence such that $R/(f_1,\ldots, f_h)$ is  $F$-split.
We consider $J=(f_1,\ldots, f_h)$.  Since $J$ is  $F$-split, we have $f_1^{p-1}\cdots f_h^{p-1}\in J^{h(p-1)}\smallsetminus \m^{[p]}$
\cite[Proposition 2.1]{FedderFputityFsing}.
The result now follows from Corollary \ref{CorH} since $J^{h(p-1)}\subseteq I^{(h(p-1))}$. 
\end{proof}

\begin{example}\label{ExZhibek}
Let $A$ and $B$ be two generic matrices of size $n\times n$ with entries in disjoint sets of variables. Let $J$ be the ideal generated by the entries of $AB-BA$ and $I$ the ideal generated by the off-diagonal entries of this matrix. Then if $n= 2$, or $3$, the ideals $I$ and $J$ are  $F$-K\"{o}nig \cite{Kady18}, and hence symbolic $F$-split.
\end{example}

We now mention and answer a question that was raised by  Huneke at the BIRS-CMO workshop on {\it Ordinary and Symbolic Powers of Ideals} during the summer of 2017 at  Casa Matem\'atica Oaxaca which arose in connections with the Conforti-Cornu\'ejols conjecture \cite{CC}.

\begin{question}[Huneke]\label{Huneke}
Let $Q\subseteq R$ be a prime  ideal such that $R/Q$ is  $F$-split, and $Q^{(n)}=Q^n$ for every $n\in\NN.$ Is $Q$ $F$-K\"onig?
\end{question}

The following  example shows that the answer to this question is negative.

\begin{example}\label{negativeHuneke}
Let $I$ and $R$ be as in Example \ref{ExFpureNotSymb} with $p\gs 3$. Then, $I$ is a prime ideal of height 2.
As noted before, $I^{(n)}=I^n$ for every $n\in \NN$; however, $I$ is not $F$-K\"onig. Indeed, by Proposition \ref{FkoningFsymb} and its proof it suffices to show $I^{(2(p-1))}=I^{2(p-1)}\subseteq \m^{[p]}$. Assume that the variable $a$ has degree $1$ and that the variables $b$, $c$, and, $d$ have degree $2$. Hence, $I$ is generated in degree $4$ and then $I^{2(p-1)}$ is generated in degree $8(p-1)$. On the other hand, if $f:=a^{n_1}b^{n_2}c^{n_3}d^{n_4}\not\in \m^{[p]}$, we must have $n_i\ls p-1$ for each $i$. Therefore, such an $f$ has degree at most $7(p-1)$.
\end{example}

\section{Symbolic and ordinary powers of determinantal ideals} 

In this section we prove our main results on symbolic powers of several types of determinantal ideals. A key point in our proofs is the construction of specific polynomials that allow us to directly apply Fedder's Criterion for an ideal to be symbolic $F$-split, Theorem \ref{ThmFedders};  this construction is inspired by the 
work of Seccia in the context of Knutson ideals 
 \cites{SecciaGen,SecciaHankel}.

\begin{notation}
Let $A$ be an $r\times s$ matrix, and $i,j,k,\ell\in\ZZ$ be such that $1\ls i \ls k\ls r$ and $1\ls j\ls \ell\ls s$.
We denote by $A^{[i,k]}_{[j,\ell]}$
the submatrix
of $A$ with row indices $i,\ldots,k$ and column indices $j,\ldots,\ell$.
\end{notation}

In the next subsections we repeatedly use the following lemma.
\begin{lemma} \label{Lemma initial square-free} Let $R$ be a $\NN$-graded polynomial ring over an $F$-finite field $K$, and let $\m$ be the homogeneous maximal ideal of $R$. Let $I$ be a radical homogeneous ideal with $H={\rm bigheight}(I)$. Assume that there exist a homogeneous polynomial $f \in R$ and a monomial order $<$ such that ${\rm in}_<(f)$ is square-free. 
\begin{enumerate}
\item If $f \in I^{(H)}$, then $I$ is symbolic $F$-split.
\item If $f^{p-1} \in (I^n)^{[p]}:I^{np}$ for every $n \in\NN$, then the Rees algebra $\R(I)$ is  $F$-split.
\end{enumerate}
\end{lemma}
\begin{proof}
We first prove (1). The assumption that ${\rm in}_<(f)$ is a square-free monomial implies that $f^{p-1} \notin \m^{[p]}$. Since $f^{p-1} \in (I^{(H)})^{p-1} \subseteq I^{(H(p-1))}$, we conclude by Corollary  \ref{CorH} that $I$ is symbolic $F$-split.

In order to prove (2), we let $\Phi$ be the trace map (see Remark \ref{Trace}), and we consider $\phi=\Phi(f^{p-1}-)$. Because of our assumptions, the map $\phi$ induces an $\R(I)$-linear map $\Psi: \left(\R(I)\right)^{1/p}\to \R(I)$. As above, we have that $f^{p-1} \notin \m^{[p]}$, and therefore $\Psi$ is surjective. It follows that $\R(I)$ is  $F$-split.
\end{proof}


\subsection{Ideals of minors of a generic matrix} \label{SubsectionDeterminantal} In this subsection we use the following setup.

\begin{setup}\label{setupDetGen}
Let $r,s\in \ZZ_{>0}$ be such that $r\leq s$. Let $X=(x_{i,j})$ be a generic $r\times s$ matrix of variables, $K$ be an $F$-finite field of characteristic $p>0$, $R=K[X]$, and $\m=(x_{i,j})$. For $t \in \ZZ_{>0}$ such that $t \leq r$, we let $I_t(X)$ be the ideal generated by the $t\times t$ minors of $X$. We let
\[
f_u(X)=\left( \prod^{r-1}_{\ell=u} \det\left( X^{[r-\ell+1,r]}_{[1,\ell]}\right) \det\left( X^{[1,\ell]}_{[s-\ell+1 ,s]}\right)\right) \cdot \left(\prod^{s-r+1}_{\ell=1} 
\det\left( X^{[1,r]}_{[\ell,r+\ell-1]}\right)\right),
\]
 for $u\leq r$.
We consider the lexicographical monomial order on $R$ induced by
$$x_{1,1}>x_{1,2}>\ldots >x_{1,s}>x_{2,1}>x_{2,2}>\cdots >x_{r,s-1}>x_{r,s}.$$
\end{setup}

\begin{remark}\label{sqft}
For any $t\in \NN$ we note that  the initial form $\IN_<(f_t(X))$ is a square-free monomial.
\end{remark}

We begin by showing that generic determinantal ideals are symbolic $F$-split.

\begin{theorem}\label{ThmSFPGen}
Assuming Setup \ref{setupDetGen}, 
the  ideal $I_t(X)$ is symbolic $F$-split.
\end{theorem}
\begin{proof}
Let $h=(s-t+1)(r-t+1)=\Ht(I_t(X))$. Let $f= f_t(X)$, and note that 
\begin{align*}
f& \in \left( \prod^{r-1}_{\ell=t} I_\ell(X)\right)^2     I_r(X)    ^{s-r+1}\\
&\subseteq \left( \prod^{r-1}_{\ell=t} \ I_t(X)^{(\ell-t+1)}\right)^2\left(     I_t(X)    ^{(r-t+1)}\right)^{s-r+1} \quad \hbox{\cite[Proposition 10.2]{BookDet}}\\
&\subseteq  I_t(X)^{((r-t)(r-t+1))}   I_t(X)^{((r-t+1)(s-r+1))}\\
&\subseteq I_t(X)^{(s-t+1)(r-t+1)}= I_t(X)^{(h)}.
\end{align*}
The conclusion follows from Remark \ref{sqft} and Lemma \ref{Lemma initial square-free} (1).
\end{proof}

From the previous proposition we obtain the following  consequences.

\begin{theorem}\label{mainDetReg}
Assuming Setup \ref{setupDetGen}, the limit $$\lim\limits_{n\to\infty} \frac{\reg\big(R/I_t(X)^{(n)}\big)}{n}$$
  exists and $$ \depth\big(R/I_t(X)^{(n)}\big)$$ stabilizes for $n\gg 0$.
  Furthermore, if 
  $1\ls t\ls \min\{m,r\}$, then  $$\lim\limits_{n\to\infty} \depth\big(R/I_t(X)^{(n)}\big)=\dim(R)- \dim\left(\R^s\left(I_t(X)\right)/\m \R^s\left(I_t(X)\right)\right)=t^2-1.$$
\end{theorem}
\begin{proof}
We know  that  $\R^s\left(I_t(X)\right)$ is Noetherian \cite[Proposition 10.2, Theorem 10.4]{BookDet}. 
  Hence, the result follows by combining  Theorem \ref{ThmSFPGen}  and Theorem \ref{mainRegDepth}.
Since  $\R^s\left(I_t(X)\right)$ is Cohen-Macaulay  \cite[Corollary 3.3]{brunsconca98}, we have that  
\begin{align*}
\lim\limits_{n\to\infty} \depth\big(R/I_t(X)^{(n)}\big) &=\dim(R)- \dim\left(\R^s\left(I_t(X)\right)/\m \R^s\left(I_t(X)\right)\right)&\quad \hbox{ by Corollary \ref{CM},}\\
&=\min\{ \depth\big(R/I_t(X)^{(n)}\big)\}&\quad \hbox{ by Theorem \ref{mainRegDepth},}\\
 & =\grade \big(\m\gr^s(I_t(X))\big) &\quad \hbox{\cite[Proposition 9.23]{BookDet}},\\
&=t^2-1&  \quad  \hbox{\cite[Proposition 10.8]{BookDet}.}
\end{align*}
\end{proof}

We now show that  $\R^s(I_t(X))$ and  $\gr^s(I_t(X))$ are strongly $F$-regular. This strengthens a result of Bruns and Conca  \cite{FRatGeneric} showing that $\R^s(I_t(X))$ is $F$-rational using techniques from the theory of Sagbi bases \cite{Sagbi}. 

\begin{theorem}\label{ThmSymbFregGen}
Assuming Setup \ref{setupDetGen}, the algebras $\R^s(I_t(X))$ and $\gr^s(I_t(X))$ are strongly $F$-regular.
\end{theorem}
\begin{proof}
We know that 
 $\R^s(I_t(X))$ and $\gr^s(I_t(X))$ are Noetherian  \cite[Proposition 10.2, Theorem 10.4]{BookDet}.
We proceed by induction on $t$.
If $t=1$, then the result follows because  $I_1(X) =\m$.
We now assume the result is true for $(t-1)\times (t-1)$-minors of a generic matrix.
If we let $f=f_t(X)$, then $\IN_<(f)$ is a square-free monomial which is not divisible by $x_{r,1}$, because $t\geq 2$.
Let $g=\frac{\prod_{i,j} x_{i,j}}{x_{r,1}\IN_<(f)}$.
We note that $\IN_< (f^{p-1} g^{p-1})=\frac{\prod_{i,j} x_{i,j}^{p-1}}{x_{r,1}^{p-1}}$,  
and as a consequence $f^{p-1} g^{p-1}\not\in \m^{[p]}$.
Let  $\phi=\Phi\left(f^{(p-1)/p} g^{(p-1)/p}-\right),$ where $\Phi: R^{1/p}\to R$ denotes the trace map introduced in Remark \ref{Trace}.
We note that $\phi\left(x_{r,1}^{(p-1)/p}\right)=1$.

We have 
$f\in I_t(X)^{(\Ht(I_t(X))}$  as shown in the proof of Theorem \ref{ThmSFPGen}.
It follows that $f^{p-1}\in \left( I_t(X)^{(n+1)}\right)^{[p]}:I_t(X)^{(np+1)}$ for every $n\in\ZZ_{\geq 0}$  \cite[Lemma 2.6]{GrifoHuneke} (see also  proof of Corollary \ref{CorH}). As a consequence, $\phi$ induces  maps $\Psi: \R^s(I_{t}(X))^{1/p} \to \R^s(I_{t}(X)) $ and $\overline{\Psi}: \gr^s(I_{t}(X))^{1/p} \to \gr^s(I_{t}(X)) $ which satisfy 
\begin{equation}\label{xr1}
\Psi\left(x_{r,1}^{(p-1)/p}\right)=1\quad  \text{and} \quad \overline{\Psi}\left(\overline{x_{r,1}}^{(p-1)/p}\right)=1.
\end{equation}

Let $A=K[U]$, where $U=(u_{i,j})_{\substack{1\ls i\ls r-1,\\ 2\ls j\ls s}}$ is a generic matrix of size $(r-1)\times (s-1)$, and let 
$$S= A[x_{1,1},\ldots,x_{r-1,1}, x_{r,1},\ldots,x_{r,s}].$$
We have an isomorphism  $\gamma: R[x^{-1}_{r,1}]\to S[x^{-1}_{r,1}]$ defined by
$x_{i,j}\mapsto u_{i,j } + x_{r,j}x_{i,1} x^{-1}_{r,1}$, 
$x_{i,1}\mapsto x_{i,1}$, 
and  
$x_{r,j}\mapsto x_{r,j}$ for $i\leq r-1$ and $j\geq 2$.
Furthermore, we have  $\gamma\left(I_t(X)R[x^{-1}_{r,1}]\right)=I_{t-1}(U) S[x^{-1}_{r,1}]$, and then $\gamma(I_t(X)^{(n)}R[x^{-1}_{r,1}])=I_{t-1}(U)^{(n)} S[x^{-1}_{r,1}]$  for every $n\in\ZZ_{\geq 0}$ \cite[Lemma 10.1]{BookDet}.
By the induction hypothesis, 
$\R^s(I_{t-1}(U))$ and $\gr^s(I_{t-1}(U))$ are strongly $F$-regular. It follows that $\R^s(I_{t-1}(U))\otimes_A S[x^{-1}_{r,1}]$ and $\gr^s(I_{t-1}(U))\otimes_A S[x^{-1}_{r,1}]$ are strongly $F$-regular,  
because strong $F$-regularity is preserved by adding variables and localizing. Therefore, thanks to the isomorphism $\gamma$, the rings $\R^s(I_{t}(X))\otimes_R  R[x^{-1}_{r,1}]$ and $\gr^s(I_{t}(X))\otimes_R R[x^{-1}_{r,1}]$
 are  also strongly $F$-regular. From this and Equation \eqref{xr1}, we  conclude that $\R^s(I_t(X))$ and $\gr^s(I_t(X))$ are strongly $F$-regular \cite[Theorem 3.3]{HoHuStrong}.
\end{proof}

We now show that the ordinary Rees algebra of a generic determinantal ideal is  $F$-split. We note that it was already known that $\R(I_t(X))$ is $F$-rational \cite{FRatGeneric}. However, $F$-rationality does not imply that the ring is $F$-split.

\begin{theorem}\label{ThmReesFpureGen}
In addition to assuming Setup \ref{setupDetGen}, suppose that $p>\min\{t,r-t\}$. Then the Rees algebra $\R(I_t(X))$ is  $F$-split.
\end{theorem}
\begin{proof}
Let $f=f_1(X)$, and note that  
$f\in I_\ell(X)^{(\Ht(I_\ell(X))}$  for every $\ell\leq r$, as shown in the proof of Theorem \ref{ThmSFPGen}.
It follows that $f^{p-1}\in \left( I_\ell(X))^{(n+1)}\right)^{[p]}: I_\ell(X) ^{(np+1)}$ for every $\ell\leq r$ and $n\in\ZZ_{\geq 0}$  \cite[Lemma 2.6]{GrifoHuneke} (cf.  proof of Corollary \ref{CorH}).
 Thus,
$$
f^{p-1} I_\ell (X)^{((n+1)p)}\subseteq f^{p-1} I_\ell (X)^{(np+1)}\subseteq \left( I_\ell(X)^{(n+1)}\right)^{[p]}
$$
for every $\ell \leq r$ and $n\in\ZZ_{\geq 0}$. Then,
\begin{align*}
f^{p-1}   I_t(X)^{np}& =f^{p-1}\left(   \bigcap^{ t}_{\ell=1}     I_\ell(X) ^{((t-\ell+1)np)}  \right) &\quad \hbox{\cite[Corollary 10.13]{BookDet}}\\\\
& \subseteq   \bigcap^{ t}_{\ell=1} f^{p-1}\left(    I_\ell(X)^{((t-\ell+1)np)}  \right)&\\
& \subseteq   \bigcap^{ t}_{\ell=1} \left(    I_\ell(X)^{((t-\ell+1)n)}  \right)^{[p]}&\\
& =   \left(  \bigcap^{ t}_{\ell=1}   I_\ell(X)^{((t-\ell+1)n)}  \right)^{[p]} &\\
& =  \left( I_t(X)^{n}\right)^{[p]} &\quad \hbox{\cite[Corollary 10.13]{BookDet}}. 
\end{align*}
The conclusion follows from Remarks \ref{sqft} and Lemma \ref{Lemma initial square-free} (2).
\end{proof}

We end this subsection with the following results which  provide bounds for the degrees of the defining equations of symbolic and ordinary  Rees algebras and  associated graded algebra of determinantal ideals of generic matrices.

\begin{theorem}\label{ThmDegGenOrd}
Assume Setup \ref{setupDetGen}.  
 Set  $\mu = \binom{r}{t}\binom{s}{t}$. 
\begin{enumerate} 
\item  Suppose $\deg(x_{i,j})=0$ for every $i,j$, 
then the defining equations of $\R(I_t(X))$  over $R$ have degree at most $\min\{ rs+1, \mu\}$.
\item   Suppose $\deg(x_{i,j})=1$ for every $i,j$, 
then the defining equations of $\R(I_t(X))$  over $K$ have total degree at most  $rs+\mu (t+1).$
\end{enumerate}
\end{theorem}
\begin{proof}
The result follows from Theorems  \ref{ThmReesFpureGen} and \ref{defEqblowup1}, and Proposition \ref{dimOfblow}.
\end{proof}

\begin{theorem}\label{ThmDegGenSymb}
Assume Setup \ref{setupDetGen} and that $t<r$.  
For $j=t,\ldots, r$, set  $\mu_j= \binom{r}{j}\binom{s}{j}$. 
\begin{enumerate} 
\item  Suppose $\deg(x_{i,j})=0$ for every $i,j$,
then the defining equations of $\R^s(I_t(X))$  over $R$ have degree at most 
$\min\{rs+1+\sum^{r}_{j=t+1}\mu_{j}(j-t),\sum^{r}_{j=t}\mu_{j}(j-t+1) \}$, and of $\gr^s(I_t(X))$ over $R/I_t(X)$, have degree at most $\min\{rs+\sum^{r}_{j=t+1}\mu_{j}(j-t),\sum^{r}_{j=t}\mu_{j}(j-t+1) \}$.
\item Suppose  $\deg(x_{i,j})=1$  for every $i,j$,
then the defining equations of $\R^s(I_t(X))$  and $\gr^s(I_t(X))$  over $K$ have total degree at most  
$rs+\sum^{r}_{j=t}\mu_{j}(2j-t+1)$.
\end{enumerate}
\end{theorem}
\begin{proof}
By Theorem \ref{ThmSFPGen} and Theorem \ref{mainCharP}, the algebras $\R^s(I_t(X))$   and  $\gr^s(I_t(X))$ are  $F$-split. Both parts of the result now follow from Theorem \ref{defEqblowup2}, Proposition \ref{dimOfblow}, and the equality 
$$ \R^s\left(I_t(X)\right)=R[I_t(X) T,I_{t+1}(X) T^2,\ldots, I_{r}(X) T^{r-t+1}] \quad \text{ \cite[Proposition 10.2, Theorem 10.4]{BookDet}.}$$
\end{proof}

\subsection{Ideals of minors of a symmetric matrix} In this subsection we use the following setup.

\begin{setup} \label{setup symmetric} Let $r \in \ZZ_{>0}$, and $Y=(y_{i,j})$ be a generic symmetric matrix of size $r\times r$. Let $K$ be an $F$-finite field of  characteristic $p>0$, $R=K[Y]$, and $\m=(y_{i,j})$. For $t \in \ZZ_{>0}$ with $t \leq r$ we let $I_t(Y)$ be ideal generated by the $t\times t$ minors of $Y$. We let 
\[
f_u(Y)= \prod^{r}_{\ell=u} \det\left( Y^{[1,\ell]}_{[r-\ell+1 ,r]}\right),
\]
for $u\leq r.$
We consider the lexicographical monomial order on $R$ induced by
$$y_{1,1}>y_{1,2}>\ldots >y_{1,r}>y_{2,2}>\ldots >y_{r,r}.$$
\end{setup}

\begin{remark}\label{sqftSym}
For any $t\in \NN$ we note that  the initial form $\IN_<(f_t(Y))$ is a square-free monomial.
\end{remark}

We now show that ideals of minors of a generic symmetric matrix is symbolic $F$-split. 
\begin{theorem}\label{ThmSFPSym}
Assuming Setup \ref{setup symmetric}, the ideal $I_t(Y)$ is symbolic $F$-split. In particular, the rings $\R^s(I_t(Y))$ and $\gr^s(I_t(Y))$ are  $F$-split.
\end{theorem}
\begin{proof}
Let $h=\frac{(r-t+1)(r-t+2)}{2}=\Ht(I_t(Y ))$ \cite[Corollary 2.4]{HtSymm}.
Let $f=f_t(Y)$, and note that
\begin{align*}
f & \in  \prod^{r}_{\ell=t} I_\ell (Y)\\
&\subseteq  \prod^{r}_{\ell=t} \left( I_t(Y)\right)^{(\ell-t+1)} \quad \hbox{\cite[Theorem 4.4]{JMV}}\\
& \subseteq  I_t(Y)   ^{(h)}.
\end{align*}
The first statement now follows from Remark \ref{sqftSym} and Lemma \ref{Lemma initial square-free} (1), and the second statement from Theorem \ref{mainCharP}.
\end{proof}

\begin{lemma}\label{LemmaSymNoeth}
Assuming Setup \ref{setup symmetric}, the rings $\R^s(I_t(Y))$ and $\gr^s(I_t(Y))$ are Noetherian. Moreover, $ \R^s\left(I_t(Y)\right)=R[I_t(Y) T,I_{t+1}(Y) T^2,\ldots, I_{r}(Y) T^{r-t+1}]$.
\end{lemma}
\begin{proof}
We have that
$I_{t+n-1}(Y)\subseteq I_t(Y)^{(n)}  \mbox{ for every  } n\gs 1$ \cite[Theorem 4.4]{JMV}.
Moreover,   we have 
that
$I_t(Y)^{(n)}=\sum I_{t+n_1-1}(Y)\cdots I_{t+n_s-1}(Y),$ 
where the sum ranges over the integers  $n_1,\ldots, n_s\gs 1$, such that $s\ls n$ and  $n_1+\cdots+n_s\gs n$ \cite[Proposition 4.3]{JMV}. The conclusion clearly follows.
\end{proof}

We obtain the following homological consequences.

\begin{theorem}\label{mainDetRegSym}
Assuming Setup \ref{setup symmetric}, the limit $$\lim\limits_{n\to\infty} \frac{\reg(R/I_t(Y)^{(n)})}{n}$$
  exists and $$ \depth(R/I_t(Y)^{(n)})$$ stabilizes for $n\gg 0$.
\end{theorem}
\begin{proof}
Since   $\R^s\left(I_t(Y)\right)$ is Noetherian by Lemma \ref{LemmaSymNoeth}, 
 the result follows by combining  Theorems \ref{ThmSFPSym}  and  \ref{mainRegDepth}.
\end{proof}

We now show that the symbolic Rees algebra of a determinantal ideal of a generic symmetric matrix is strongly $F$-regular.
\begin{theorem}
Assuming Setup \ref{setup symmetric},   $\R^s(I_t(Y))$ is strongly $F$-regular. 
\end{theorem}
\begin{proof}
We set $h=\Ht(I_t(Y))$.
We know that 
 $\R^s(I_t(Y))$  Noetherian by Lemma \ref{LemmaSymNoeth}.

If $t=r$, then $I_t(Y)$ is principal, and so, $\R^s(I_t(Y))=R[\det(Y)T]$, which is isomorphic to a polynomial ring over $K$. Then, $\R^s(I_t(Y))$ is strongly $F$-regular.

We proceed by induction on $t$.
If $t=1$, then the result follows because  $I_1(Y) =\m$.
We now assume the result is true for $(t-1)\times (t-1)$-minors of a generic symmetric matrix.
If we let
\[
f= \det\left( Y^{[2,r]}_{[2,r]}\right) \cdot \prod^{r-1}_{\ell=1} \det\left( Y^{[1,\ell]}_{[r-\ell+1 ,r]}\right),
\]
as in the proof of Theorem \ref{ThmSFPSym}, we have that $f\in I_t(Y)^{(h-1)}$. In addition, 
 then $\IN_<(f)$ is a square-free monomial which is not divisible by $y_{1,1}$, because $t\geq 2$.
We have that $f^{p-1} \not\in \m^{[p]}$.
Let  $\phi=\Phi\left(f^{(p-1)/p}-\right),$ where $\Phi: R^{1/p}\to R$ denotes the trace map introduced in Remark \ref{Trace}.
We note that $\phi\left(y_{1,1}^{(p-1)/p}\right)=1$.
It follows that $f^{p-1}\in \left( I_t(Y)^{(n+1)}\right)^{[p]}:I_t(Y)^{((n+1)p)}$ for every $n\in\ZZ_{\geq 0}$ by Proposition \ref{PropH2}. 

As a consequence, $\phi$ induces  maps $\Psi: \R^s(I_{t}(X))^{1/p} \to \R^s(I_{t}(X))$ which satisfy 
\begin{equation}\label{yr1}
\Psi\left(y_{1,1}^{(p-1)/p}\right)=1\quad  \text{and} \quad \overline{\Psi}\left(\overline{y_{1,1}}^{(p-1)/p}\right)=1.
\end{equation}

Let $A=K[U]$, where $U=(u_{i,j})_{\substack{2\ls i\ls r-1,\\ 2\ls j\ls r}}$ is a generic symmetric matrix of size $(r-1)\times (r-1)$, and let 
$$S= A[y_{1,1},\ldots,y_{2,1},\ldots,y_{r,1}].$$
We have an isomorphism  $\gamma: R[y^{-1}_{1,1}]\to S[y^{-1}_{y,1}]$ defined by
$y_{i,j}\mapsto u_{i,j } + y_{1,j}x_{i,1}y^{-1}_{1,1}$, 
$y_{i,1}\mapsto y_{i,1}$, 
and  
$y_{1,j}\mapsto y_{1,j}$ for $j\geq 2$
\cite{MVSymm} (see also \cite[Lemma 1.1]{HtSymm}.

Furthermore, we have  $\gamma\left(I_t(Y)R[y^{-1}_{1,1}]\right)=I_{t-1}(U) S[y^{-1}_{1,1}]$, and then $\gamma(I_t(Y)^{(n)}R[x^{-1}_{1,1}])=I_{t-1}(U)^{(n)} S[y^{-1}_{1,1}]$  for every $n\in\ZZ_{\geq 0}$ \cite[Lemma 10.1]{BookDet}.
By the induction hypothesis, 
$\R^s(I_{t-1}(U))$ and $\gr^s(I_{t-1}(U))$ are strongly $F$-regular. It follows that $\R^s(I_{t-1}(U))\otimes_A S[y^{-1}_{1,1}]$ is strongly $F$-regular,  
because strong $F$-regularity is preserved by adding variables and localizing. Therefore, thanks to the isomorphism $\gamma$, the ring $\R^s(I_{t}(Y))\otimes_R  R[y^{-1}_{1,1}]$ 
is also strongly $F$-regular. From this and Equation \eqref{yr1}, we  conclude that $\R^s(I_t(Y))$ is strongly $F$-regular \cite[Theorem 3.3]{HoHuStrong}.
\end{proof}

We now show that the ordinary  Rees algebra of a generic symmetric determinantal ideal is  $F$-split.
\begin{theorem}\label{ThmReesFpureSym}
In addition to Setup \ref{setup symmetric}, suppose that $p>\min\{t,r-t\}$. Then the Rees algebra $\R(I_t(Y))$ is  $F$-split.
\end{theorem}
\begin{proof}

Let $f=f_1(Y)$, and note that  
$f\in I_\ell(Y)^{(\Ht(I_\ell(Y))}$  for every $\ell\leq r$, as shown in the proof of Theorem  \ref{ThmSFPSym}.
It follows that $f^{p-1}\in \left( I_\ell(Y))^{(n+1)}\right)^{[p]}:I_\ell(Y) ^{(np+1)}$ for every $\ell\leq r$ and $n\in\ZZ_{\geq 0}$  \cite[Lemma 2.6]{GrifoHuneke} (cf.  proof of Corollary \ref{CorH}).
 Thus,
$$
f^{p-1} I_\ell (Y)^{((n+1)p)}\subseteq f^{p-1} I_\ell (Y)^{(np+1)}\subseteq \left( I_\ell(Y)^{(n+1)}\right)^{[p]}
$$
for $n\in\ZZ_{\geq 0}$. 
Then,
\begin{align*}
f^{p-1}   I_t(Y)^{np}& =f^{p-1}\left(   \bigcap^{ t}_{\ell=1}     I_\ell(Y) ^{((t-\ell+1)np)}  \right) &\quad \hbox{\cite[Theorem 4.4]{JMV}}\\\\
& \subseteq   \bigcap^{ t}_{\ell=1} f^{p-1}\left(    I_\ell(Y)^{((t-\ell+1)np)}  \right)&\\
& \subseteq   \bigcap^{ t}_{\ell=1} \left(    I_\ell(Y)^{((t-\ell+1)n)}  \right)^{[p]}&\\
& =   \left(  \bigcap^{ t}_{\ell=1}   I_\ell(Y)^{((t-\ell+1)n)}  \right)^{[p]} &\\
& =  \left( I_t(Y)^{n}\right)^{[p]} &\quad \hbox{\cite[Theorem 4.4]{JMV}}. 
\end{align*}
The result follows from Remark \ref{sqftSym} and Lemma \ref{Lemma initial square-free} (2).
\end{proof}

We end with the following results about degrees of defining equations for ordinary Rees and associated graded algebras in the case of generic symmetric matrices.

\begin{theorem}\label{ThmDegSymOrd}
Assume Setup \ref{setup symmetric}.  
 Set  $\mu = \frac12 \binom{r}{t}^2$. 
\begin{enumerate} 
\item  Suppose $\deg(y_{i,j})=0$ for every $i,j$, 
then the defining equations of $\R(I_t(Y)$  over $R$ have degree at most $\min\{\binom{r+1}{2}+1, \mu\}$.
\item   Suppose $\deg(y_{i,j})=1$ for every $i,j$, 
then the defining equations of $\R(I_t(Y))$  over $K$ have total degree at most  $\binom{r+1}{2}+\mu (t+1).$
\end{enumerate}
\end{theorem}
\begin{proof}
The result follows from Theorems  \ref{ThmReesFpureSym} and \ref{defEqblowup1}, and Proposition \ref{dimOfblow}.
\end{proof}

\begin{theorem}\label{ThmDegSymSymb}
Assume Setup \ref{setup symmetric}.  
For $j=t,\ldots, r$, set  $\mu_j = \frac12 \binom{r}{j}^2$. 
\begin{enumerate} 
\item  Suppose $\deg(y_{i,j})=0$ for every $i,j$,
then the defining equations of $\R^s(I_t(Y))$  over $R$ have degree at most 
$\min\{\binom{r+1}{2}+1+\sum^{r}_{j=t+1}\mu_{j}(j-t),\sum^{r}_{j=t}\mu_{j}(j-t+1) \}$, and of $\gr^s(I_t(Y))$ over $R/I_t(Y)$, have degree at most $\min\{\binom{r+1}{2}
+\sum^{r}_{j=t+1}\mu_{j}(j-t),\sum^{r}_{j=t}\mu_{j}(j-t+1) \}$.
\item Suppose  $\deg(y_{i,j})=1$  for every $i,j$,
then the defining equations of $\R^s(I_t(Y))$  and $\gr^s(I_t(Y))$  over $K$ have total degree at most  
$\binom{r+1}{2}+\sum^{r}_{j=t}\mu_{j}(2j-t+1)$.
\end{enumerate}
\end{theorem}
\begin{proof}
The result  follows from Theorem \ref{ThmSFPSym}, Theorem \ref{defEqblowup2}, Proposition \ref{dimOfblow}, and Lemma \ref{LemmaSymNoeth}.
\end{proof}

\subsection{Ideals of Pfaffians of a skew-symmetric matrix}

For the convenience of the reader, we recall the definition of Pfaffians.
\begin{definition}
Let $Z=(z_{i,j})$ be a generic $r\times r$ skew symmetric matrix, i.e., $z_{i,j}=-z_{j,i}$ for every  $1\ls i<j\ls r$, and $z_{i,i}=0$ for every $1\ls i\ls m$. A minor of the form $\det\left(Z^{[i_1,\ldots, i_{2t}]}_{[i_1,\ldots, i_{2t}]}\right)$  is the square of a polynomial $\pf\left(Z_{[i_1,\ldots, i_{2t}]}\right) \in R=K[Z]$. Such a polynomial is called a {\it $2t$-Pfaffian} of $Z$.   
\end{definition}

 In this subsection we use the following setup.

\begin{setup}\label{setup pfaffians}
Let $r\in \ZZ_{>0}$, and $Z=(z_{i,j})$ be a generic skew symmetric matrix of size $r\times r$. Let $K$ be an $F$-finite field of  characteristic $p>0$, $R=K[Z]$, and $\m=(z_{i,j})$. For $t \in \ZZ_{>0}$ such that $2t \leq r$, we let $P_{2t}(Z)$ be the ideal generated by the $2t$-Pfaffians of $Z$. If $r$ is odd, we set $b=\lfloor r/2 \rfloor$, and we let
\begin{align*}
&f_{2u}(Z)= \\
&\left( \prod^{b-1}_{\ell=u}
 \pf\left(Z_{[1,\ldots,2\ell]}\right)
\pf\left(Z_{[1,\ldots,\ell, \ell+2, \ldots, 2\ell+1]}\right) 
 \pf\left(Z_{[r+1-2\ell,\ldots,r]}\right) 
 \pf\left(Z_{[r-2\ell,\ldots, r-\ell -1, r-\ell+1,  \ldots,r]}\right)\right)\cdot  \\
& \; \; \; \; \; \;\; \; \; \left(\pf\left(Z_{[1,\ldots,r-1]}\right)
\pf\left(Z_{[2,\ldots,r]}\right) 
 \pf\left(Z_{[1,\ldots, b ,b+2,\ldots,r]}\right)\right),
\end{align*}
for $u\leq r/2$.
If  $r$ is even, we set $b=r/2$, and we let
\begin{align*}
&f_{2u}(Z)=\\ 
&\left( \prod^{b-1}_{\ell=u}
 \pf\left(Z_{[1,\ldots,2\ell]}\right)
\pf\left(Z_{[1,\ldots,\ell, \ell+2, \ldots, 2\ell+1]}\right) 
 \pf\left(Z_{[r+1-2\ell,\ldots,r]}\right) 
 \pf\left(Z_{[r-2\ell\ldots, r-\ell -1, r-\ell+1  \ldots,r]}\right)\right)
\pf(Z),
\end{align*}
for $u\leq r/2$.
We consider the lexicographical monomial order on $R$ induced by
$$z_{1,r}>z_{1,r-1}>\ldots >z_{1,2}>z_{2,r}>\ldots >z_{r-1,r}.$$
\end{setup}

\begin{remark}\label{sqftPf}
For any $t\in \NN$ we note that  the initial form $\IN_<(f_{2t}(Z))$ is a square-free monomial.
\end{remark}

In the following result we show ideals of Pfaffians are symbolic $F$-split.

\begin{theorem}\label{ThmSFPPf}
Assuming Setup \ref{setup pfaffians}, the ideal $P_{2t}(Z)$ is symbolic $F$-split.
\end{theorem}
\begin{proof}
We set $h=\frac{(r-2t+1)(r-2t+2)}{2}=\Ht(P_{2t}(Z))$ \cite[Theorem 2.3]{HtSkew}. Let $b=\lfloor r/2 \rfloor$, and $f=f_{2t}(Z)$. If $r$ is odd, 
we have that 
\begin{align*}
f& \in \left( \prod^{b-1}_{\ell=t} P_{2\ell}(Z)\right)^4 \cdot P_{2b}(Z)^3 &\\
&\subseteq \left( \prod^{b-1}_{\ell=t}  P_{2t}(Z)^{(\ell-t+1)}\right)^4 \cdot
\left(  P_{2t}(Z)^{(b-t+1)}\right)^3&
\quad \hbox{\cite[Theorem 4.6]{JMV}}\\
&= \left( \prod^{b-t}_{a=1}  P_{2t}(Z)^{(a)}\right)^4 \cdot
\left(  P_{2t}(Z)^{(b-t+1)}\right)^3&\\
&\subseteq   P_{2t}(Z)^{(2(b-t)(b-t+1))} 
\cdot P_{2t}(Z)^{(3(b-t+1))}&\\
&\subseteq   P_{2t}(Z)^{(2(b-t)(b-t+1)+3(b-t+1))}& \\
&=  P_{2t}(Z)^{(h)}.&
\end{align*}

On the other hand, if $r$ is even we have that
\begin{align*}
f& \in \left( \prod^{b-1}_{\ell=t} P_{2\ell}(Z)\right)^4 \cdot P_{r}(Z)& \\
&\subseteq \left( \prod^{b-1}_{\ell=t}  P_{2t}(Z)^{(\ell-t+1)}\right)^4 \cdot
P_{2t}(Z)^{(b-t+1)}&\quad \hbox{ \cite[Theorem 4.6]{JMV}}\\
&= \left( \prod^{b-t}_{a=1}  P_{2t}(Z)^{(a)}\right)^4 \cdot
P_{2t}(Z)^{(b-t+1)}&\\
&\subseteq   P_{2t}(Z)^{(2(b-t)(b-t+1))} 
\cdot P_{2t}(Z)^{(b-t+1)}&\\
&\subseteq   P_{2t}(Z)^{(2(b-t)(b-t+1)+(b-t+1))}& \\
&=   P_{2t}(Z)^{(h)}.&
\end{align*}
The conclusion follows from Remarks \ref{sqftPf} and Lemma \ref{Lemma initial square-free} (2).
\end{proof}


The previous result leads to the following homological consequences. 

\begin{theorem}\label{mainDetRegPf}
Assuming Setup \ref{setup pfaffians},  the limit $$\lim\limits_{n\to\infty} \frac{\reg(R/P_{2t}(Z)^{(n)})}{n}$$
  exists and $$ \depth(R/P_{2t}(Z)^{(n)})$$ stabilizes for $n\gg 0$,
  Furthermore, 
  $$\lim\limits_{n\to\infty} \depth(R/P_{2t}(Z)^{(n)})=\dim(R)- \dim\left(\R^s\left(P_{2t}(Z)\right)/\m \R^s\left(P_{2t}(Z)\right)\right)=t(2t-1)-1.$$
\end{theorem}
\begin{proof}
We recall that  $\R^s\left(P_{2t}(Z)\right)$ is Noetherian \cite{YoungPf} (see also \cite[Section 3]{baetica98}). 
  Hence, the result follows by combining  Theorem \ref{ThmSFPPf}  and  Theorem \ref{mainRegDepth}.
Since  $\R^s\left(P_{2t}(Z)\right)$ is Cohen-Macaulay  \cite[Corollary 3.2]{baetica98}, we have that  
\begin{align*}
\lim\limits_{n\to\infty} \depth\big(R/P_{2t}(Z)^{(n)}\big) &=\dim(R)- \dim\left(\R^s\left(P_{2t}(Z)\right)/\m \R^s\left(P_{2t}(Z)\right)\right)&\quad \hbox{ by Remark \ref{CM},}\\
&=\min\{ \depth\big(R/P_{2t}(Z)^{(n)}\big)\}&\quad \hbox{ by Theorem \ref{mainRegDepth},}\\
 & =\grade \big(\m\gr^s(P_{2t}(Z))\big) &\quad \hbox{\cite[Proposition 9.23]{BookDet}}.
\end{align*}
It remains to show that $\grade \big(\m\gr^s(P_{2t}(Z))\big) =t(2t-1) -1 $. This computation is already known in the generic case in arbitrary characteristic  \cite[Proposition 10.8]{BookDet}, we adapt the proof for ideals of Pfaffians. 
We let $H$ be the poset of  all the Pfaffians of $Z$, and consider the partial order induced by  
$$
\pf\left(Z_{[i_1,\ldots, i_{2u}]}\right) \ls \pf\left(Z_{[a_1,\ldots, a_{2v}]}\right) \Longleftrightarrow u\gs v \mbox{ and } i_s\ls a_s \text{ for all } 1\ls s\ls 2v.$$
We let $\Omega$ be the subposet  of $H$  consisting of the $2s$-Pfaffians with $s\gs t$.  We note that $\Omega$ is also given by 
$$
\Omega=
   \{ \delta\in H\mid \delta \ls [r-2t+1,\ldots, r ]\}.
$$
Since $\R^s(P_{2t}(Z))$ is Cohen-Macaulay then so is 
  $\gr^s(P_{2t}(Z))$ \cite[Proof of Proposition 2.4]{Varb}. Then, we have that  $\grade \big(\m\gr^s(P_{2t}(Z))\big)=\rk(H)-\rk(\Omega)$  \cite[Proof of Proposition 10.8]{BookDet}, where the  rank of a poset $P$ is defined as 
$$\rk(P) = \max\{i\mid \mbox{ there exists a chain } p_1<p_2<\cdots <p_i  \mbox{ of elements of } P\}.$$   We  note that every maximal chain of $H$ has length $\rk(H) = \dim(R)= \binom{r}{2}$ \cite[Lemma 5.13(d) and Proposition 5.10]{BookDet}.  We also note that a maximal chain of  $\Omega$ can be extended to a maximal chain of $H$ by adjoining a maximal chain of  Pfaffians of the submatrix of $Z$ with rows $\{r-t+1,\ldots, r\}$ and columns $\{r-t+1,\ldots, r\}$. Then, $\rk(H) -\rk(\Omega) = \binom{r}{2}-\left(\binom{r}{2}-\binom{2t}{2}+1\right)=t(2t-1)-1$, finishing the proof.
\end{proof}

We now show that  $\R^s(P_{2t}(Z))$   and  $\gr^s(P_{2t}(Z))$ are strongly $F$-regular.

\begin{theorem}\label{ThmSymbFregPf}
Assuming Setup \ref{setup pfaffians}, the rings $\R^s(P_{2t}(Z))$ and $\gr^s(P_{2t}(Z))$ are strongly $F$-regular.
\end{theorem}
\begin{proof}
 We know that 
 $\R^s(P_{2t}(Z))$ and $\gr^s(P_{2t}(Z))$ are Noetherian \cite{YoungPf} (see also \cite[Section 3]{baetica98}).
We proceed by induction on $t$.
If $t=1$, then the result follows because  $P_{2t}(Z)=\m$.
We now assume the result is true for  the ideal of $(2t-2)$-Pfaffians.

Let $f=f_{2t}(Z)$, and note that $\IN_<(f)$ is a square-free monomial which is not divisible by $z_{1,2}$, because $t\geq 2$.
Let $g=\frac{\prod_{i<j} z_{i,j}}{z_{1,2}\IN_<(f)}$.
We note that $\IN(f^{p-1} g^{p-1})=\frac{\prod_{i<j} z_{i,j}^{p-1}}{z_{1,2}^{p-1}}$, and so, $f^{p-1} g^{p-1}\not\in \m^{[p]}$.
Let  $\phi=\Phi(f^{(p-1)/p} g^{(p-1)/p}-),$ where $\Phi: R^{1/p}\to R$ denotes the trace map introduced in Remark \ref{Trace}. 
We note that $\varphi(z_{1,2}^{(p-1)/p})=1$.

We have 
$f\in P_{2t}(Z)^{(\Ht(P_{2t}(Z))}$, as shown in the proof of Theorem \ref{ThmSFPPf}.
We therefore have that $f^{p-1}\in \left( P_{2t}(Z)^{(n+1)}\right)^{[p]}:P_{2t}(Z)^{(np+1)}$ for every $n\in\ZZ_{\geq 0}$  \cite[Lemma 2.6]{GrifoHuneke} (see also  proof of Corollary \ref{CorH}). Thus, $\phi$ induces  maps $\Psi: \R^s(P_{2t}(Z))^{1/p} \to \R^s(P_{2t}(Z)) $ and $\overline{\Psi}: \gr^s(P_{2t}(Z))^{1/p} \to \gr^s(P_{2t}(Z)) $ which satisfy 
\begin{equation}\label{xrpf}
\Psi\left(z_{1,2}^{(p-1)/p}\right)=1\quad  \text{and} \quad \overline{\Psi}\left(\overline{z_{1,2}}^{(p-1)/p}\right)=1.
\end{equation}

Let $A=K[U]$, where $U=(u_{i,j})_{3\ls i<j\ls \ls r}$ be a generic skew-symmetric matrix of size $(r-2)\times (r-2)$, and let
$S= A[z_{1,3},\ldots,z_{1,r},z_{2,3},\ldots,z_{2,r}]$.
We have an isomorphism  $\gamma: R[z^{-1}_{1,2}]\to S[z^{-1}_{1,2}]$ defined by
$z_{i,j}\mapsto u_{i,j } - z_{1,j}z_{2,i} z^{-1}_{1,2}+ z_{1,i}z_{2,j} z^{-1}_{1,2}$ for $3\ls i<j\ls \ls r$, 
$z_{1,i}\mapsto z_{1,i}$ for $i\gs 2$, 
and  
$z_{2,i}\mapsto z_{2,i}$ for $i\gs 3$. 
Furthermore, $\gamma(P_{2t}(Z))R[z^{-1}_{1,2}])=P_{2t-2}(U) S[z^{-1}_{1,2}]$, and then $\gamma(P_{2t}(Z))^{(n)}R[z^{-1}_{1,2}])=\Psi(P_{2t-2}(U))^{(n)} S[z^{-1}_{1,2}])$  for every $n\in\ZZ_{\geq 0}$ \cite[Lemma 1.2]{HtSkew} (see also  \cite[Lemma 10.1]{BookDet}). 
By the induction hypothesis, the rings
$\R^s(P_{2t-2}(U))$ and $\gr^s(P_{2t-2}(U))$ are strongly $F$-regular.
It follows that
$\R^s(P_{2t-2}(U))\otimes_A S[z^{-1}_{1,2}]$ and $\gr^s(P_{2t-2}(U))\otimes_A S[z^{-1}_{1,2}]$ are strongly $F$-regular,
because strong $F$-regularity is preserved by adding variables and localizing. 
Therefore, thanks to the isomorphism $\gamma$, the rings $\R^s(P_{2t}(Z))\otimes_R  R[z^{-1}_{1,2}]$ and $\gr^s(P_{2t}(Z))\otimes_R R[z^{-1}_{1,2}]$ are also strongly $F$-regular. From this and Equation \eqref{xrpf}, we  conclude that  $\R^s(P_{2t}(Z))$ and $\gr^s(P_{2t}(Z))$ are strongly $F$-regular \cite[Theorem 3.3]{HoHuStrong}.
\end{proof}

In the following result we show that  ordinary  Rees algebras of ideals of Pfaffians are also  $F$-split.

\begin{theorem}\label{ThmReesFpurePf}
In addition to assuming Setup \ref{setup pfaffians}, suppose that $p>\min\{2t,r-2t\}$. Then the Rees algebra $\R(P_{2t}(Z))$ is  $F$-split.
\end{theorem}
\begin{proof}

Let $f=f_{2}(Z)$, and note that $f\in P_{2 \ell}(Z)^{\Ht(P_{2 \ell}(Z))}$, as shown in the proof of Theorem \ref{ThmSFPPf}. It follows that $f^{p-1}\in \left( P_{2\ell}(Z)^{(n+1)}\right)^{[p]}: P_{2\ell}(Z) ^{(np+1)}$ for every $\ell\leq r/2$ and $n\in\ZZ_{\geq 0}$   \cite[Lemma 2.6]{GrifoHuneke} (cf.  proof of Corollary \ref{CorH}). 
Thus, 
$$
f^{p-1} P_{2\ell} (Z)^{((n+1)p)}\subseteq f^{p-1} I_\ell (X)^{(np+1)}\subseteq \left(  P_{2\ell} (Z)^{(n+1)}\right)^{[p]}
$$
for every $\ell \leq r/2$ and $n\in\ZZ_{\geq 0}$. We then get that
\begin{align*}
f^{p-1}   P_{2t} (Z)^{np}& =f^{p-1}\left(   \bigcap^{ t}_{\ell=1}    P_{2\ell} (Z) ^{((t-\ell+1)np)}  \right)&\quad \hbox{\cite[Proposition 2.6]{DNPaf}}\\
& \subseteq   \bigcap^{ t}_{\ell=1} f^{p-1}\left(     P_{2\ell} (Z)^{((t-\ell+1)np)}  \right)&\\
& \subseteq   \bigcap^{ t}_{\ell=1} \left(    P_{2\ell} (Z)^{((t-\ell+1)n)}  \right)^{[p]}&\\
& \subseteq   \left(  \bigcap^{ t}_{\ell=1}  P_{2\ell} (Z)^{((t-\ell+1)n)}  \right)^{[p]}&\\
& =  \left(  P_{2t} (Z)^{n}\right)^{[p]}& \quad \hbox{\cite[Proposition 2.6]{DNPaf}}.
\end{align*}
The result follows from Remark \ref{sqftPf} and Lemma \ref{Lemma initial square-free} (2).
\end{proof}

As in the previous subsections, we end with the following results about degrees of  defining equations for ordinary blowup algebras for ideals of Pfaffians of generic skew-symmetric matrices.

\begin{theorem}\label{ThmDegPfOrd}
Assume Setup \ref{setup pfaffians}.  
 Set  $\mu = \binom{r}{2t}$. 
\begin{enumerate} 
\item  Suppose $\deg(z_{i,j})=0$ for every $i,j$, 
then the defining equations of $\R(P_{2t} (Z))$  over $R$ have degree at most $\min\{ \binom{r}{2}+1, \mu\}$.
\item   Suppose $\deg(z_{i,j})=1$ for every $i,j$, 
then the defining equations of $\R(P_{2t} (Z))$  over $K$ have total degree at most  $\binom{r}{2}+\mu (t+1).$
\end{enumerate}
\end{theorem}
\begin{proof}
The result follows from Theorems  \ref{ThmReesFpurePf} and \ref{defEqblowup1}, and Proposition \ref{dimOfblow}.
\end{proof}

\begin{theorem}\label{ThmDegPfSymb}
Assume Setup \ref{setup pfaffians}.  
For $j=t,\ldots, \lfloor r \rfloor$, set  $\mu_j= \binom{r}{2j}$. 
\begin{enumerate} 
\item  Suppose $\deg(z_{i,j})=0$ for every $i,j$,
then the defining equations of $\R^s(P_{2t} (Z))$  over $R$ have degree at most 
$\min\{\binom{r}{2}+1+\sum^{r}_{j=t+1}\mu_{j}(j-t),\sum^{r}_{j=t}\mu_{j}(j-t+1) \}$, and of $\gr^s(P_{2t} (Z))$ over $R/P_{2t} (Z)$, have degree at most $\min\{\binom{r}{2}+\sum^{\lfloor r/2 \rfloor}_{j=t+1}\mu_{j}(j-t),\sum^{\lfloor r/2 \rfloor}_{j=t}\mu_{j} (j-t+1)\}$.
\item Suppose  $\deg(z_{i,j})=1$  for every $i,j$,
then the defining equations of $\R^s(P_{2t} (Z))$  and $\gr^s(P_{2t} (Z))$  over $K$ have total degree at most  
$\binom{r}{2}+\sum^{\lfloor r/2 \rfloor}_{j=t}\mu_{j}(2j-t+1)$.
\end{enumerate}
\end{theorem}
\begin{proof}
By Theorem \ref{ThmSFPPf} and Theorem \ref{mainCharP}, the algebras $\R^s(P_{2t} (Z))$   and  $\gr^s(P_{2t} (Z))$ are  $F$-split. Both parts of the result now follow from Theorem \ref{defEqblowup2}, Proposition \ref{dimOfblow}, and the equality 
\[
\R^s(P_{2t}(Z)) =R[P_{2t}(Z)T,  P_{2t+2}(Z)T^2,\ldots, P_{2\lfloor r/2 \rfloor}(Z)T^{\lfloor r/2\rfloor }] \quad \text{ \cite{YoungPf} (see also \cite[Section 3]{baetica98}).}\qedhere
\]

\end{proof}

\subsection{Ideals of minors of a Hankel matrix}

We first recall the definition of Hankel matrix.
\begin{definition}
Let $j,c\in \ZZ_{>0}$, with $j \leq c$. Let $w_1,\ldots, w_c$ be variables. We denote by $W^c_j$ the $j\times (c+1-j) $ {\it Hankel matrix}, which has the following entries 
$$W^c_j=\begin{pmatrix}
w_1&w_2&\cdots&w_{c+1-j}\\
w_2&w_3&\cdots&\cdots\\
w_3&\cdots&\cdots&\cdots\\
\vdots&\vdots&\vdots&\vdots\\
w_j&\cdots&\cdots&w_c
\end{pmatrix}.$$
\end{definition}

\begin{setup} \label{setup hankel}
Let $j,c \in \ZZ_{>0}$, with $j \leq c$, and $W^c_j$ be the $j \times (c+1-j)$ Hankel matrix. Let $K$ be an $F$-finite field of  characteristic $p>0$, $R=K[W^c_j]$, and $\m=(w_1,\ldots,w_c)$. For $t \in \ZZ_{>0}$ with $t \leq \min\{j, c+1-j\}$, we denote by $I_t(W_j^c)$ the ideal generated by the minors of size $t$ of $W^c_j$. If $c$ is odd, we set $m=\frac{c+1}{2}$, and we let 
\[
f_{\rm odd}(W_j^c)= \det\left( W^c_m\right)\det\left( (W^c_m)^{[2,m]}_{[1,m-1]}\right).
\]
If $c$ is even, we set $m=\frac{c}{2}$, and we let
\[
f_{\rm even}(W_j^c)= \det \left((W^c_m)^{[1,m]}_{[1,m]}\right) \det \left( (W^c_m)^{[1,m]}_{[2,m+1]}\right).  
\]
Consider the lexicographical monomial order on $R$ induced by
$$w_1>w_3>\ldots >w_c>w_2>w_4>\ldots >w_{c-1}.$$
\end{setup}

\begin{remark}\label{sqftHank}
We note that the initial forms $\IN_<(f_{\rm odd}(W_j^c))$ and $\IN_<(f_{\rm even}(W_j^c))$ are   square-free monomials.
\end{remark}

\begin{remark}\label{RemChangeHankel}
 It is well-known that $I_t(W^c_j)$ only depends on $c$ and $t$, that is, $I_t(W_j^c)=I_t(W_t^c)$ for every $t\ls \min\{j, c+1-j\}$.
\end{remark}

\begin{theorem}\label{ThmSFPHankel}
Assuming Setup \ref{setup hankel}, the ideal $I_t(W^c_j)$ is symbolic $F$-split for every $t\leq \min\{j,c+1-j\}$.  In particular, the rings $\R^s(I_t(W^c_j))$ and $\gr^s(I_t(W^c_j))$ are  $F$-split.
\end{theorem}
\begin{proof}
Let $m=\lfloor\frac{c+1}{2} \rfloor$, and observe that $t \leq m$. Moreover, we have that $I_t(W^c_j)=I_t(W^c_m)$ by Remark \ref{RemChangeHankel}. 

If $c$ is odd, let $f=f_{\rm odd}(W^c_j)$, and observe that $h=\Ht(I_t(W^c_j ))=c-2t+2=2m-2t+1$. We then have
\begin{align*}
f& \in  I_m(W^c_m)I_{m-1}(W^c_m)\\
&\subseteq  I_t(W^c_m)^{(m-t+1)}  I_{t}(W^c_m)^{(m-t)}   &\quad \hbox{\cite[Theorem 3.16 (a)]{conca98}}\\
&\subseteq  I_{t}(W^c_m)^{(2m-2t+1)} =  I_t(W^c_m)^{(h)}.
\end{align*}
If $c$ is even we let $f=f_{\rm even}(W^c_j)$, and we observe that $h=\Ht(I_t(W^c_j ))=c-2t+2=2m-2t+2$. In this case we have that 
\begin{align*}
f& \in  I_m(W^c_m)I_{m}(W^c_m)\\
&\subseteq   I_t(W^c_m)^{(m-t+1)}  I_{t}(W^c_m)^{(m-t+1)}  & \quad \hbox{\cite[Theorem 3.16 (a)]{conca98}}\\
&\subseteq  I_{t}(W^c_m) ^{(2m-2t+2)} =  I_t(W^c_m)^{(h)}.
\end{align*}
In both cases, we have shown that $f \in I_t(W^c_m)^{(h)}$. The first statement now follows from Remark \ref{sqftHank} and Lemma \ref{Lemma initial square-free} (1), and the second statement from Theorem \ref{mainCharP}.
\end{proof}

We obtain the following homological consequences.

\begin{theorem}\label{mainDetRegHankel}
Assuming Setup \ref{setup symmetric}, the limit $$\lim\limits_{n\to\infty} \frac{\reg(R/I_t(W^c_j )^{(n)})}{n}$$
  exists and $$ \depth(R/I_t(W^c_j )^{(n)})$$ stabilizes for $n\gg 0$.
\end{theorem}
\begin{proof}
Since   $\R^s\left(I_t(W^c_j )\right)$ is Noetherian \cite[Theorem 4.1]{conca98}, 
 the result follows by combining   Theorems \ref{ThmSFPHankel} and \ref{mainRegDepth}.
\end{proof}

We now show that  ordinary  Rees algebras of a  determinantal ideals  of Hankel matrices are  $F$-split.

\begin{theorem}\label{ThmReesFpureHankel}
Assume Setup \ref{setup hankel}. Then, the Rees algebra $\R(I_t(W^c_j))$ is  $F$-split.
\end{theorem}
\begin{proof}

Let $m=\lfloor\frac{c+1}{2}\rfloor$, and observe that $t\leq m$.
We have that $I_\ell(W^c_j)=I_\ell(W^c_m)$ for every $\ell\leq m$ by Remark \ref{RemChangeHankel}. 
If $c$ is odd, we set $f=f_{\rm odd}(W^c_j)$, otherwise we set $f=f_{\rm even}(W^c_j)$. 
From the proof of Theorem  \ref{ThmSFPHankel}, we see that $f\in I_\ell(W^c_m)  ^{(\Ht(I_\ell(W^c_m) ))}$ for every $\ell \ls  m$. It follows that $f^{p-1}\in \left( I_\ell(W^c_m)^{(n+1)}\right)^{[p]}:\left( I_\ell(W^c_m) \right)^{(np+1)}$ for every $\ell \ls  m$ and $n\in\ZZ_{\geq 0}$ \cite[Lemma 2.6]{GrifoHuneke} (cf.  proof of Corollary \ref{CorH}). Thus, 
$$
f^{p-1} I_\ell (W^c_m)^{((n+1)p)}\subseteq f^{p-1} I_\ell (W^c_m)^{(np+1)}\subseteq \left( I_\ell(W^c_m)^{(n+1)}\right)^{[p]}.
$$
for all $\ell \leq m$ and $n \in \NN$.  
Then,
\begin{align*}
f^{p-1}   I_t(W^c_m)^{np}& =f^{p-1} \left( \bigcap^t_{\ell=1}    I_\ell(W^c_m)^{(np(t+1-\ell))}\right) & \quad \hbox{\cite[Theorem 3.16 (a)]{conca98}}\\
& \subseteq \bigcap^t_{\ell=1} \left(f^{p-1}     I_\ell(W^c_m)^{(np(t+1-\ell))} \right)\\
& \subseteq \bigcap^t_{\ell=1} \left(  I_\ell(W^c_m)^{(n(t+1-\ell))} \right)^{[p]}\\
& =\left(  \bigcap^t_{\ell=1}  I_\ell(W^c_m)^{(n(t+1-\ell))} \right)^{[p]}\\
&=\left(   I_t(W^c_m)^{n}\right)^{[p]}& \quad \hbox{\cite[Theorem 3.16 (a)]{conca98}}.
\end{align*}
The conclusion follows from Remarks \ref{sqftHank} and Lemma \ref{Lemma initial square-free} (2).
\end{proof}

Finally, we prove the following results about degrees of defining equations for ordinary Rees and associated graded algebra for ideals of minors of generic Hankel matrices.

\begin{theorem}\label{ThmDegHankelOrd}
Assume Setup \ref{setup hankel}.  
 Set  $\mu =\binom{c+1-t}{t}$. 
\begin{enumerate} 
\item  Suppose $\deg(w_i)=0$ for every $i$, 
then the defining equations of $\R(I_t(W^c_j))$  over $R$ have degree at most $\min\{c, \mu\}$.
\item   Suppose $\deg(w_i)=1$ for every $i$, 
then the defining equations of $\R(I_t(W^c_j))$  over $K$ have total degree at most  $c+\mu (t+1).$
\end{enumerate}
\end{theorem}
\begin{proof}
The result follows from Theorems \ref{ThmReesFpureHankel} and \ref{defEqblowup1}, and Proposition \ref{dimOfblow}.
\end{proof}

\begin{theorem}\label{ThmDegHankelSymb}
Assume Setup \ref{setup hankel}.  
For $j=t,\ldots, m$, set  $\mu_j = \binom{c+1-j}{j}$. 
\begin{enumerate} 
\item  Suppose $\deg(w_{i})=0$ for every $i$. The defining equations of $\R^s(I_t(W^c_j))$  over $R$ have degree at most 
$\min\{c+1+\sum^{m}_{j=t+1}\mu_{j}(j-t),\sum^{m}_{j=t}\mu_{j}(j-t+1) \}$, and the defining equations of $\gr^s(I_t(W^c_j))$ over $R/I_t(W^c_j)$ have degree at most $\min\{c+\sum^{m}_{j=t+1}\mu_{j}(j-t),\sum^{m}_{j=t}\mu_{j}(j-t+1) \}$.
\item Suppose  $\deg(w_{i})=1$  for every $i$,
then the defining equations of $\R^s(I_t(W^c_j))$  and $\gr^s(I_t(W^c_j))$  over $K$ have total degree at most  
$c+\sum^{m}_{j=t}\mu_{j}(2j-t+1)$.
\end{enumerate}
\end{theorem}
\begin{proof}
The result follows from  Theorem \ref{ThmSFPHankel},  Theorem \ref{defEqblowup2}, Proposition \ref{dimOfblow}, and the equality 
\[
\R^s\left(I_t(W^c_j)\right)=R[I_t(W^c_j) T,I_{t+1}(W^c_j) T^2,\ldots, I_{m}(W^c_j) T^{m-t+1}] \quad \text{ \cite[Proposition 4.1]{conca98}.}\qedhere
\]
\end{proof}

\subsection{Binomial edge ideals}

We now give another example of symbolic $F$-split ideals,  the binomial edge ideals, which  are generated by minors of certain matrices related to graphs. 

	\begin{definition}[{\cites{HHHKR10,Ohtani}}]
		Let $G=(V(G),E(G))$ be a simple graph such that $V(G)=[n]=\{1,2,\ldots,n\}$.  Let $K$ be a field and $S=K[x_1,\ldots,x_n,y_1,\ldots,y_n]$ the ring of polynomials in $2n$ variables. The {\it binomial edge ideal}, $\cJ_G$, of $G$  is defined by
		$$
		\cJ_G=\left(x_iy_j-x_jy_i\; | \; \hbox{ for }\{i,j\}\in E(G)\right).
		$$
	\end{definition}

\begin{definition}[{\cite{HHHKR10}}]
A graph $G$ on $[n]$ is {\it closed} if $G$ has a labeling of the vertices such that 
for all edges $\{i,j\}$ and $\{k,l\}$ with $i<j$ and $k<l$ one has $\{j,l\}\in E(G)$ if $i=k$, and $\{i,k\}\in E(G)$ if $j=l$.
\end{definition}


The binomial ideals of closed graphs class of graphs can be characterized via Gr{\"o}bner bases for binomial edge ideals \cite{HHHKR10}. This class of binomial edge  ideals has  been studied in several works   \cite{HHHKR10, CRPI, EHH,EneHerzogSymbolic}.  For example, it is know that, for closed graphs, a binomial edge ideal is equidimensional if and only if it is Cohen-Macaulay \cite[Theorem 3.1]{EHH}. Since their initial ideals correspond to a bipartite graphs, we have that the ordinary and symbolic powers of closed binomial edge ideals coincide  \cite[Corollary 3.4]{EneHerzogSymbolic}. This follows from the analogous result for monomial edge ideals of bipartite graphs \cite[Theorem 5.9]{SVV}.

\begin{proposition}\label{propBinomialEdge}
Let $G$ be a closed connected  graph such that $\cJ_G$ is equidimensional. Then, 
 $\cJ_G$ is symbolic $F$-split.  In particular, , the rings $\R^s(\cJ_G)$ and $\gr^s(\cJ_G)$ are  $F$-split.
\end{proposition}
\begin{proof}
Since $G$ is connected and $S/\cJ_G$ is equidimensional we have that  $\bh(\cJ_G)=n-1$ (\cite[Theorem 3.1]{EHH}, \cite[Corollary 3.4]{HHHKR10}). We also  have that $\cJ_G^{(n)}=\cJ_G^{n}$ \cite[Corollary 3.4]{EneHerzogSymbolic}.
Now, since a closed graph is a {\it proper interval graph}, it contains a Hamiltonian graph \cite{CRPI, BH}.
We assume with out loss of generality that this path is given by $1,\ldots, n$, in this order.
We set $f=\prod^{n-1}_{i=1}(x_iy_{i+1}-x_{i+1}y_i).$
Then, $f^{p-1}\in \cJ^{((n-1)(p-1))}\smallsetminus (x^p_1,\ldots,x^p_n,y^p_1,\ldots,y^p_n)$.
Therefore, 
 $\cJ_G$ is symbolic $F$-split by Corollary \ref{CorH}. The second statement follows from Theorem \ref{mainCharP}.
\end{proof}

\section{Examples of monomial $F$-split filtrations}\label{SecDet}



In this section we present several classes of filtrations of monomial ideals that are  $F$-split. The list of examples include, symbolic powers and rational powers of squarefree monomial ideals, and initial ideals of symbolic and ordinary powers of  determinantal ideals of generic and Hankel matrices of variables, and of Pfaffians of generic skew-symmetric matrices.

Throughout this section  we assume the following setup.

\begin{setup}\label{setupMonomialFpure}
Let $R$ be a standard graded polynomial ring $R=K[x_1,\ldots, x_d]$  over an $F$-finite field $K$ of characteristic $p>0$. For a vector $\fn=(n_1,\ldots, n_d)\in \NN^d$ we set $\fx^\fn=x_1^{n_1}\cdots x_d^{n_d}$.
\end{setup}

\subsection{ $F$-split filtrations obtained from  monomial valuations}
Assuming Setup \ref{setupMonomialFpure},  let $\fa=(a_1,\ldots, a_d)\in \NN^d$ and consider the following function on the set of monomials in $R$ 
$$v(\fx^\fn)=\fn\cdot \fa.$$
We extend $v$ to the entire $R$ by setting 
\begin{equation*}\label{monVal}
v(f)=v(\sum c_i \fx^{\fn_i} ):=\min\{v(\fx^{\fn_i})\}
\end{equation*}
for a polynomial  $f=\sum c_i \fx^{\fn_i} \in R$ with $0\neq c_i\in K$. Such a function is called a {\it monomial valuation} of $R$ \cite[Definition 6.1.4]
{huneke2006integral}. 

The following is the main result of this subsection.

\begin{theorem}\label{thm_F-pureMon}
Assuming Setup  \ref{setupMonomialFpure}, let $v_1,\ldots, v_r$ be monomial valuations of $R$. For each $n\in \NN$ we set  $$I_n = \{f\in R\mid v_i(f)\gs n, \text{ for every }1\ls i\ls r\}.$$ 
\begin{enumerate}
\item The sequence $\II=\{I_n\}_{n\in \NN}$ is an  $F$-split filtration of monomial ideals.
\item $\R(\II)$ is  Noetherian and strongly $F$-regular.
\end{enumerate}

\end{theorem}
\begin{proof}
We begin with the proof of (1). 
We note that each $I_n$ is a monomial ideal by the definition of valuation  \cite[Definition 6.1.1]{huneke2006integral}. 
Let $\rho: K^{1/p}\to K$ be an splitting. 
Let $\phi: R^{1/p}\to R$ be the $K$-linear map defined on the  monomials of $R$ as:
$$\phi((c\fx^\fn)^{1/p})=\phi(c^{1/p}(x_1^{n_1}\cdots x_d^{n_d})^{1/p})=
\begin{cases}
\rho(c^{1/p})x_1^{n_1/p}\cdots x_d^{n_d/p}, &\mbox{ if }n_1\equiv \cdots \equiv n_d\equiv 0\pmod{p},\\
0.& \mbox{otherwise;}
\end{cases}
$$
In particular, $\phi(c^{1/p}(\fx^\fn)^{1/p})=\phi(c^{1/p})\phi((\fx^\fn)^{1/p})$. 
If ${\widetilde{c}},c\in K$ and $\fx^{\widetilde{\fn}},\fx^\fn\in R$ are monomial, then
$$
\phi( ({\widetilde{c}} \fx^{\widetilde{\fn}}) c^{1/p} \fx^{\fn/p})=\phi( {\widetilde{c}}  c^{1/p} )\phi(\fx^{\widetilde{\fn}}\fx^{\fn/p})
={\widetilde{c}} \phi(  c^{1/p}) \fx^{\widetilde{\fn}}\phi(\fx^{\fn/p})
={\widetilde{c}}  \fx^{\widetilde{\fn}}\phi(  c^{1/p} \fx^{\fn/p}).
$$
Then,
\begin{equation}\label{R-hom}
\phi(f(g^{1/p}))=f\phi(g^{1/p})\quad\text{for every}\quad f,g\in R.
\end{equation}

We have that $\phi$ is an $R$-homomorphism and thus a splitting of the natural inclusion $R\hookrightarrow R^{1/p}$. 


Now, let $n\in \NN$ be arbitrary and $\fx^\fn=x_1^{n_1}\cdots x_d^{n_d}\in I_{np+1}$  be such that $\phi((\fx^\fn)^{1/p})\neq 0$. Therefore, $p|n_i$  for every $i=1,\ldots, d$ and then $$v_i\big(\phi((\fx^\fn)^{1/p})\big)=\sum_{j=1}^d\frac{n_j}{p}v_i(x_j)\gs \left\lceil \frac{np+1}{p}\right\rceil =n+1,$$ for every $1\ls i\ls r$. 
It follows that $\phi((\fx^\fn)^{1/p})\in I_{n+1}$. Therefore $\phi( (I_{np+1})^{1/p})\in I_{n+1}$ for every $n\in \NN$, which implies $\II$ is an  $F$-split filtration.

We continue with the proof of (2). Let $\fa_i=(v_i(x_1),\ldots, v_i(x_d))\in \NN^d$ and let $\cM_i$ be the {\it affine semigroup} $$\cM_i=\{(\fn,n)\in \NN^{d+1}\mid (\fa_i,-1)\cdot (\fn,n)\gs 0 \}\subseteq \NN^{d+1}$$
for every $1\ls i\ls r$. Then, $\cM=\cM_1\cap \ldots \cap \cM_r$ is a finitely generated affine semigroup \cite[Theorem 1.1 and Corollary 1.2]{VCA}. 
We  note that $\R(\II)=K[\cM]$, and so, $\R(\II)$ is a finitely generated $K$-algebra.

Since $\R(\II)=K[\cM]$ is  $F$-split regardless of the characteristic of the field by part (1), we have that 
$\cM$ is a normal monoid \cite[Corollary 6.3]{SemiNormal}.
As a consequence, $\R(\II)$ is a strongly $F$-regular ring \cite[Theorem 1]{HochsterTori}, finishing the proof.
\end{proof}

In the following example we include several well-studied filtrations of monomial ideals covered by Theorem \ref{thm_F-pureMon}.

\begin{example}\label{monExamp}
Some examples of  $F$-split filtrations of monomial ideals.
\begin{enumerate}
\item ({\it Rational powers of monomial ideals}) Let $I$ be a monomial ideal and $u_1,\ldots, u_r$ its {\it Rees valuations}, which in this setting are also monomial valuations (see \cite[Proposition 10.3.4]{huneke2006integral}). For each $i$, set $u_i(I)=\min\{u_i(f)\mid f\in I\}$ and let $u=\mcm(u_1(I), \ldots, u_r(I))$.  For each $1\ls i\ls, r$,  consider the monomial valuation $v_i=\frac{u}{u_i(I)}u_i$. Then, the monomial ideal  $$I_n = \{f\in R\mid v_i(f)\gs n, \text{ for every }1\ls i\ls r\}$$ is the {\it $\frac{n}{u}$-rational power} of $I$ (see \cite[Proposition 10.5.5]{huneke2006integral}, \cite{Lewis20}). Therefore, by Theorem \ref{thm_F-pureMon} (1), $\II=\{I_n\}_{n\in \NN}$ is an  $F$-split filtration. Thus,  $\R(\II)$  is Noetherian and strongly $F$-regular  by Theorem \ref{thm_F-pureMon} (2)  and $\gr(\II)$ are  $F$-split by Theorem \ref{mainCharP}. Since these are algebras are Noetherian, the conclusions of Theorem \ref{mainRegDepth} hold for $\II$. 

The sequence of {\it integral closure powers} $\{\overline{I^n}\}_{n\in \NN}$ is a subsequence of  the the rational powers. Indeed, one has  $\overline{I^{n}}=I_{nu}$ \cite[Proposition 10.5.2 (5)]{huneke2006integral}. Thus, the conclusions of Theorem \ref{mainRegDepth} hold for $\{\overline{I^n}\}_{n\in \NN}$ and, since direct summands of strongly $F$-regular  rings are strongly $F$-regular, the {\it normal Rees algebra} $\overline{\R(I)}=\bigoplus_{n\in \NN}\overline{I^n}T^n$ is also strongly $F$-regular.

\item ({\it Symbolic powers of squarefree monomial ideals}) Let $I$ be a squarefree monomial ideal and $\fp_1,\ldots, \fp_r$ be its minimal primes. The functions $v_i(f)=\max\{n\mid f\in \fp_i^n\}$ are monomial valuations, therefore, the {\it symbolic powers} of $I$
$$I^{(n)} = \{g\in R\mid v_i(g)\gs n, \text{ for every }1\ls i\ls r\}$$
form an  $F$-split filtration. 
\end{enumerate}
\end{example}

\subsection{ $F$-split filtrations obtained from  initial ideals}

\begin{setup}\label{setupInitial}
Assume Setup \ref{setupMonomialFpure} and suppose that $R$ is equipped with a monomial order $<$. 
If $\II$ is a filtration, then $\IN_<(\II)=\{ \IN(I_n)\}_{n\in\NN}$ is also a filtration. 
\end{setup}



In the following proposition, we provide sufficient conditions for certain filtrations and algebras obtained from initial ideals to be  $F$-split.

\begin{proposition}\label{propInSymbOrd}
Assuming Setup \ref{setupInitial}, let $I\subseteq R$ be a homogeneous equidimensional radical ideal of height $h$   such that $\IN_<(I)$ is radical. 
\begin{enumerate}
\item If there exists $f\in I^{(h)}$ such that $\IN_<(f)$ is square-free, then the filtration $\{ \IN_<(I^{(n)})\}_{n\in\NN}$ is  $F$-split.
\item If there exists $f\in \bigcap_{n\in \NN} \left( I^n\right)^{[p]}:I^{np}$ such that $\IN_<(f)$ is square-free, then $\R(\{\IN_<(I^{n})\})$ is  $F$-split.
\end{enumerate}

\end{proposition}
\begin{proof}
We begin with (1). Since  $ f\in I^{(h)}$, we have that $f^{p-1}\in \left( I^{(n+1)}\right)^{[p]}:I^{(np+1)}$ for every $n\in\NN$ \cite[Lemma 2.6]{GrifoHuneke} (cf.  proof of Corollary \ref{CorH}).
Thus,
\begin{align*}
\IN_<(f^{p-1})\IN_<(I^{(np+1))})&=\IN_<(f^{p-1} I^{(np+1)})\\
&\subseteq \IN_<( (I^{(n+1)})^{[p]})\\
&=\IN_<(I^{(n+1)})^{[p]}
\end{align*} 
Then, $\IN_<(f^{p-1})\not\in \m^{[p]}$ and 
$\IN_<(f^{p-1})\in  \left( \IN_<(I^{(n+1)})\right)^{[p]}:\IN_<(I^{(np+1)})$ for every $n\in\NN$.
It follows that $\{ \IN_<(I^{(n)})\}_{n\in\NN}$ is an  $F$-split filtration by Proposition \ref{PropFedder}.

We continue with (2).  The assumptions of this part guarantee that
\begin{align*}
\IN_<(f^{p-1})\IN_<(I^{np})&=\IN_<(f^{p-1} I^{np})\\
&\subseteq \IN_<( (I^{(n)})^{[p]})\\
&=\IN_<(I^{n})^{[p]}.
\end{align*}
Then, $\IN_<(f^{p-1})\not\in \m^{[p]}$ and 
$\IN_<(f^{p-1})\in  \left( \IN_<(I^{n})\right)^{[p]}:\IN_<(I^{np})$ for every $n\in\NN$.
We conclude that    $\R(\{\IN_<(I^n)\})$ is  $F$-split proceeding as in the proof of Lemma \ref{Lemma initial square-free} (2).
\end{proof}

Our next goal is to prove a technical result, Theorem \ref{Theorem initial torsion free}, which is crucially used in the rest of this section. First we need two lemmas.

\begin{lemma} \label{lemma EM} 
Assume Setup \ref{setupInitial} with  $K$ a perfect field. Let  $P \subseteq R$ be a homogeneous prime ideal  such that $\IN_<(P)$ is radical, and let $Q=(x_1,\ldots,x_h)$. Suppose that  $\IN_<(P)\subseteq Q$. 
Let $\{ x^{\fn_1} T^{b_1},\ldots,  x^{\fn_m} T^{b_m}\}$ be a set of generators of $S=\bigoplus_{n\gs 0}  \IN_<(P^{(n)}) T^n$ as an $R$-algebra. Write $\fn_i=(n_{i,1}\ldots,n_{i,d})$ for each $1\ls i\ls m$ and 
suppose that $p\geq \max\{n_{i,j}, b_i\}_{1\ls i\ls m,\, 1\ls j\ls d}$.
Then, $\IN_<(P^{(n+1)}) \subseteq  Q \IN_<(P^{(n)})$ for every $n \in\NN$.
\end{lemma}
\begin{proof}
Set $\cJ:=\bigoplus_{n\gs 0}  \IN_<(P^{(n+1)}) T^n\subseteq S$.  
We have that   $\mathcal{A}:=\{ x^{\fn_1} T^{b_1-1},\ldots,  x^{\fn_m} T^{b_m-1}\}$ generates $\cJ$ as an $S$-ideal.
 Fix $ x^{\fn_i} T^{b_i-1}\in\mathcal{A}$. We claim that there exists $1\ls j\ls h$ such that $n_{i,j}\gs 1$. If not, we would get that $ x^{\fn_i} T^{b_i-1} R_Q=R_Q$, and therefore $\IN_<(P^{(b_i)})\not\subseteq Q$. However, this contradicts the assumption that $\IN_<(P)\subseteq Q$. 
It follows from this claim and by the assumption on $p$ that $\partial_j (x^{\fn_i} T^{b_i-1})\neq 0$.
Let $g\in P^{(b_i)}$ be such that $\IN_< (g)=x^{\fn_i}$, and set $\partial_j=\frac{\partial}{\partial x_j}$. We have that
$$
0\neq \partial_j (x^{\fn_i} T^{b_i-1})=\partial_j (\IN_< (g) T^{b_i-1})
= \IN_< (\partial_j(g)) T^{b_i-1}
\in \IN_<(P^{(b_i-1)})T^{b_i-1}
$$
by the characterization of symbolic powers in terms of differential operators \cite{Zariski, Nagata, SurveySymbPowers,ZariskiNagata}.
Hence
$$x^{\fn_i} T^{b_i-1}\in  x_j \IN_<(P^{(b_i-1)})T^{b_i-1}\subseteq Q \IN_<(P^{(b_i-1)})T^{b_i-1}\subseteq Q\cdot  S.$$
We have that $\cJ\subseteq Q\cdot S$, and therefore $\IN_<(P^{(n+1)}) \subseteq  Q \IN_<(P^{(n)})$ for every $n \in\NN$.
\end{proof}

\begin{lemma} \label{GisEquid} 
Let $(S,\n)$ be a universally catenary local domain. Let $\II$ be a  filtration of non-zero  $S$-ideals, such that the Rees algebra $\R(\II)$ is finitely generated as $S$-algebra. Then $\gr(\II)$ is  equidimensional of dimension $\dim(S)$.
\end{lemma}
\begin{proof}
Consider the {\it extended Rees algebra} $A:=S[\II T,T^{-1} ]=\oplus_{n\in \ZZ} I_n T^n$ where $I_n=S$ for $n\ls0$. Thus, there exists $\ell \in \ZZ_{>0} $ such that  $A$ is an integral extension of $S[I_\ell T, T^-1]=\oplus_{n\in \ZZ} I_\ell^n T^n$; see Equation \eqref{eventSG}.  Therefore, $A$ has dimension $\dim(S) +1$ \cite[Theorem 2.2.5, Theorem 5.1.4(1)]{huneke2006integral}. Now, $T^{-1}$ is a homogeneous non-zero element of $A$ and then every minimal prime of $(t^{-1})$ has height one.  We also have that $A/(T^{-1})\cong \gr(\II)$.  
The result now follows by noticing that  $A$ is a catenary graded domain which has a unique homogeneous maximal ideal $\oplus_{n<0}I_nt^n\oplus \n \oplus_{n>0}I_nt^n$, and thus for any homogeneous ideal $\J\subseteq A$ one has $\dim(A/\J)+\Ht(\J)=\dim(A)$. 
\end{proof}

\begin{remark} \label{remark Sullivant} Assuming Setup \ref{setupInitial} with $K$ algebraically closed, Sullivant proved that for all $n \in \NN$ and all radical ideals $I$ such that $\IN_<(I)$ is radical one has ${\rm in}_<(I^{(n)}) \subseteq {\rm in}_<(I)^{(n)}$ \cite{Sull}. We point out that his proof works, more generally, if $K$ is just any perfect field.
\end{remark}

We are now ready to present the technical theorem (cf.  \cite[Theorem 1.2]{HSV89}).

\begin{theorem} \label{Theorem initial torsion free}
Assume Setup \ref{setupInitial} with  $K$ a perfect field. Let  $P \subseteq R$ be a homogeneous prime ideal  such that $\IN_<(P)$ is radical. Let $S=\bigoplus_{n\gs 0}  \IN_<(P^{(n)}) T^n$, $\J = \bigoplus_{n\gs 0}  \IN_<(P^{(n+1)}) T^n\subseteq S$, and $G=S/\J$. Assume that $S$ is Noetherian and let $b_1,\ldots, b_m$ be the generating degrees of $S$ and an $R$-algebra. Assume that $p> \lcm(b_1,\ldots, b_m)$ and that  $G$ is reduced. Then, there is a one-to-one correspondence between primes of $R$ minimal over $\IN_<(P)$ and minimal primes of $G$. 

More specifically, if $\mathfrak{q} \in \Spec(R)$ is minimal over $\IN_<(P)$, then $Q=\ker(G \to G \otimes_R R_q)$ is a minimal prime of $G$ such that $Q \cap R = q$, and every minimal prime of $G$ is of this form. 
\end{theorem}
\begin{proof}
For all $n\gs 1$, set $J_n= \IN_<(P^{(n)}) $. Without loss of generality, we may assume that $K$ is infinite. For all $n \geq 1$, we have $J_1^n \subseteq J_n$, and by Remark \ref{remark Sullivant} we also have $J_n \subseteq J_1^{(n)}$. Let $\mathfrak{q}  \in \Spec(R)$ be minimal over $J_1$. After localizing at $\mathfrak{q} $, the above inclusions all become equalities. Moreover, since $J_1$ is radical, we have $J_1R_{\mathfrak{q} } = \mathfrak{q} R_{\mathfrak{q} }$, and therefore $J_nR_{\mathfrak{q} } = (\mathfrak{q} R_{\mathfrak{q} })^n$. It follows that $G \otimes_R R_{\mathfrak{q} } \cong \gr_{ \mathfrak{q} R_{\mathfrak{q} }}(R_{\mathfrak{q} })$, and since $R_{\mathfrak{q} }$ is regular, this associated graded ring is a domain. If we let $Q = \ker(G \to G \otimes_R R_{\mathfrak{q} })$, then $G/Q$ is a subring of a domain, hence a domain itself. It follows that $Q$ is a prime ideal of $G$, and it is easy to see that $Q \cap R = \ker(R/J_1 \to R_{\mathfrak{q} }/\mathfrak{q} R_{\mathfrak{q} }) = \mathfrak{q} $. Finally, the map $G \to G\otimes_R R_{\mathfrak{q} }$ becomes an isomorphism when localized at $Q$. Therefore $QG_Q = 0$, i.e., $Q$ is a minimal prime of $G$.

Now let $\ov{Q}$ be a minimal prime of $G$. Let $Q$ be a lift of $\ov{Q}$ to $S$, so that $Q$ is a prime of $S$ which is minimal over $\J$. Let $q = Q \cap R$, which is a monomial ideal. Hence, $\mathfrak{q}$ generated by variables,  say $x_1,\ldots,x_h$. 
Consider the multiplicative system $W=K[x_{h+1},\ldots,x_d]\smallsetminus \{0\}\subseteq R$, let
 $K'=K(x_{h+1},\ldots,x_d)=W^{-1}K[x_{h+1},\ldots,x_d]$ and $R'=W^{-1}R=K'[x_1,\ldots,x_h]$.
In addition, $W^{-1}q= W^{-1}Q' \cap W^{-1}R = (x_1,\ldots,x_h) R'$.
We replace $K$ by $K'$, and may assume that
 $Q \cap R = \m=(x_1,\ldots,x_h)$. By Lemma \ref{GisEquid} we have that $\dim(S/Q) = \dim(G) = d$.
We want to show that $\m$ is minimal over $J_1$.

First we want to show that, in our current setup, the monomial ideal $J_1$ is generated by variables. If not, after possibly relabeling the indeterminates, we may find integers $1 \leq \ell_1 \leq \ell_2 \leq d$, a square-free monomial ideal $A \subseteq (x_1,\ldots,x_{\ell_1})^2$, and an ideal $L=(x_{\ell_1+1},\ldots,x_{\ell_2})$ generated by variables such that $J_1 = A + L$.

By Lemma \ref{lemma EM} we have that $J_{n+1} \subseteq \m J_n$ for all $n \geq 0$, and thus $\J \subseteq \m S \subseteq Q$. Since $G = S/\J$ is reduced, and $Q$ is a minimal prime of $\J$, we have that $\J S_Q = QS_Q$, and thus $\J S_Q = \m S_Q$. In particular, there exist $n \geq 0$ and $fT^n \in S \smallsetminus Q$ such that $fT^n(\m S) \subseteq \J$. Thus, $f \in J_n$ is an element such that $f\m \subseteq J_{n+1}$, and $fT^n \notin Q$. In our assumptions, if $s=\lcm(b_1,\ldots, b_m)$, then we can find homogeneous elements $a_1T^s,\ldots,a_dT^s$ which form a full system of parameters for the finitely generated graded $K$-algebra $S/\m S$. In particular, notice that $a_iT^s \notin Q$ for all $i$, since $Q$ is a minimal prime of $\m S$. 

First, we claim that $a_i \notin \m^{s+1}$ for all $i=1,\ldots,d$, i.e., each $a_i$ has degree at most $s$. By way of contradiction, assume that $a_i \in \m^{s+1}$ for some $i$. We have that
\[
a_i T^s (fT^n)^{s+1}  \subseteq (f\m)^{s+1}T^{(n+1)(s+1)-1} \subseteq J_{(n+1)(s+1)} T^{(n+1)(s+1)-1} \subseteq \J \subseteq Q,
\]
which contradicts the fact that $a_iT^s$ and $fT^n$ do not belong to $Q$.

We now claim that $J_1^{(s)}  = (A+L)^{(s)} \subseteq (x_1,\ldots,x_{\ell_1})^{s+1} + L$. Let $R_1= K[x_1,\ldots,x_{\ell_1}]$, $\m_1 = (x_1,\ldots,x_{\ell_1})R_1$ and $B = A \cap R_1$. Since we have that $(A+L)^{(s)} \subseteq A^{(s)} + L$, it suffices to show that $B^{(s)} \subseteq \m_1^{s+1}$ in $R_1$. Since $p = {\rm char}(K) > s$, every $K$-linear differential operator $\partial$ of $R'$ of order at most $s$ can be written as $\partial_{x_i} \circ \partial'$ for some $K$-linear differential operator $\partial'$ of order at most $s-1$, and some $1 \leq i \leq \ell_1$. We have that
\[
\partial(B^{(s)}) = \partial_{x_i} (\partial'(B^{(s)})) \subseteq \partial_{x_i}(B) \subseteq \partial_{x_i}(\m_1^2) \subseteq \m_1 \ \text{(for instance, \cite[Proposition 2.14]{SurveySymbPowers}}).
\]
Therefore we have that $B^{(s)} \subseteq \m_1^{(s+1)} = \m_1^{s+1}$ \cite[Proposition 2.14]{SurveySymbPowers}. At this, point, we have shown that $(a_1,\ldots,a_d) \subseteq J_s \subseteq J_1^{(s)} \subseteq (x_1,\ldots,x_{\ell_1})^{s+1} + L$. Since each $a_i$ is homogeneous, and has degree at most $s$, we must have $(a_1,\ldots,a_d) \subseteq L$.  

Finally, let $I=(a_1,\ldots,a_d)$; we claim that $\sqrt{I} = J_1$. Once we have shown this, we have that $J_1  = \sqrt{I} \subseteq L \subseteq J_1$, which implies that $J_1=  L$ is generated by variables. Since $a_1T^s,\ldots,a_dT^s$ are a full system of parameters for $S/\m S$, we can find an integer $N \gg 0$ such that $(J_s^N T^{Ns}) S/\m S \subseteq (a_1T^s,\ldots,a_dT^s) S/\m S$, so that $J_sT^{Ns} \subseteq (a_1T^s,\ldots,a_dT^s) J_{Ns-s}T^{Ns} + (\m S)_{Ns}T^{Ns}$. In particular, we have a containment $J_s^N \subseteq (a_1,\ldots,a_d) + \m J_{Ns}$. Because $S$ is generated in degree at most $s$, we have that $J_{Ns} = J_s^N$, and we conclude that $J_{Ns} \subseteq (a_1,\ldots,a_d) + \m J_{Ns}$. It follows from graded Nakayama's Lemma that $J_{Ns} \subseteq (a_1,\ldots,a_d)$, and since $J_1^{Ns} \subseteq J_{Ns}$ we conclude that $J_1 \subseteq \sqrt{I}$. Since $I \subseteq J_1$, the other inclusion is trivial.
 
To conclude the proof, observe that since $J_1$ is generated by variables, we have that $J_1^n = J_1^{(n)}$ for all $n$. 
It follows that $G=\bigoplus_{n \geq 0} J_n/J_{n+1} =  \bigoplus_{n \geq 0} J_1^n/J_1^{n+1} = \gr_{J_1}(R)$ is reduced.
Then,  $\m$ is a minimal prime over $J_1$, by the one-to-one correspondence between minimal primes of $G$ and $R/J_1$ already established in this case \cite[Theorem 1.2]{HSV89}.
\end{proof}

From the previous theorem we obtain the following useful corollary.

\begin{corollary}\label{CorInSymbEq}
Assume Setup \ref{setupInitial} with  $K$ a perfect field. Let  $P \subseteq R$ be a homogeneous prime ideal  such that $\IN_<(P)$ is radical. Let $S=\bigoplus_{n\gs 0}  \IN_<(P^{(n)}) T^n$, $\J = \bigoplus_{n\gs 0}  \IN_<(P^{(n+1)}) T^n\subseteq S$, and $G=S/\J$. Assume  $S$ is Noetherian and let $b_1,\ldots, b_m$ be the generating degrees of $S$ and an $R$-algebra. Assume  $p> \lcm(b_1,\ldots, b_m)$ and that  $G$ is reduced. 
Then,  $\IN_<(P^{(n)}) = \IN_<(P)^{(n)}$ for all $n \geq 1$.
\end{corollary}
\begin{proof}
For all $n\gs 1$, set $J_n= \IN_<(P^{(n)}) $. By Theorem \ref{Theorem initial torsion free} we have that $G = \bigoplus_{n \geq 0} J_n/J_{n+1}$ is a torsion-free $R/J_1$-module. The same argument used in Remark \ref{remark torsion free} shows that 
$\Ass_R(R/J_n) \subseteq {\rm Min}(J_1)$ for all $n \geq 1$. Thus, since $J_n$ contains $J_1^n$, it must in fact contain $J_1^{(n)}$. Finally, because the containment $J_n \subseteq J_1^{(n)}$ always holds, we obtained the desired equality.
\end{proof}

The following observation shows how close is the equality $\IN_<(P^{(n)}) = \IN_<(P)^{(n)}$, for every $n\in\NN$, to $I$ being  symbolic $F$-split.
\begin{remark}
Assume Setup \ref{setupInitial} with  $K$ a perfect field. Let  $P \subseteq R$ be a homogeneous prime ideal  such that $\IN_<(P)$ is radical and $\IN_<(P^{(n)}) = \IN_<(P)^{(n)}$ for all $n \geq 1$. Then, 
$x_1\cdots x_d\in  \IN_<(P)^{(h)}=\IN_<(P^{(h)})$, where $h=\height(P)$.
Let $f\in P^{(h)}$ be an homogeneous polynomial such that $\IN_<(f)=x_1\cdots x_d.$ Thus, $f\not\in\m^{[p]},$ and so, $P$ is symbolic $F$-split
by Corollary \ref{CorH}.
\end{remark}

\subsection{Results in characteristic zero}

Let $J \subseteq R=\QQ[x_1,\ldots,x_d]$, and let $A=\ZZ[x_1,\ldots,x_d]$. Let $J_A = J \cap A$, and let $f_1,\ldots,f_s \in A$ be generators of $J_A$. Note that $(f_1,\ldots,f_s)R = J$. For every prime $p \in \ZZ$ we let $J(p) = J_A \cdot A(p)$, where $A(p) = \ZZ/(p)[x_1,\ldots,x_d]$. Note that, if $A/J_A$ is flat over $A$, then $J_A \otimes_\ZZ \ZZ/(p)$ can be identified with the ideal $J(p)$ of $A(p) \cong A \otimes_\ZZ \ZZ/(p)$.


\begin{lemma} \label{lemma symb}
Given an integer $n \in \NN$ and an ideal $J \subseteq R = \QQ[x_1,\ldots,x_d]$, we have that $(J(p))^{(n)} = ((J_A)^{(n)})(p)$ for all primes $p  \gg 0$.
\end{lemma}
\begin{proof}
Consider a minimal primary decomposition $(J_A)^n = I_1 \cap \ldots \cap I_s$ in $A$. We collect the primary components and write $(J_A)^n = (J_A)^{(n)} \cap I$, where $(J_A)^{(n)} = I_1 \cap \ldots \cap I_t$ is the $n$-th symbolic power of $J_A$ in $A$. Let $P_j = \sqrt{I_j}$. By generic freeness \cite[Lemma 8.1]{HoRo}, there exists  an element $a \in \ZZ$ such that all the modules $(A/I_j)_a$, $(A/J_A^n)_a$, $(A/J_A^{(n)})_a$ and $(A/J_A)_a$ are free over $\ZZ_a$. Since we are seeking to get the equality $(J(p))^{(n)} = ((J_A)^{(n)})(p)$ only for $p \gg 0$, without loss of generality we directly assume that all the above modules are free over $\ZZ$. By flatness, we have that $J_A^n \otimes_\ZZ \ZZ/(p) \cong J_A^n(p) = (J(p))^n = I_1(p) \cap \ldots \cap I_s(p)$ as ideals of $A(p)$ and, in particular, $((J_A)^{(n)})(p) = I_1(p) \cap \ldots I_t(p)$. It is left to show that $I_1(p) \cap \ldots \cap I_t(p) = (J(p))^{(n)}$. Since $(J_A)^{(n)} \otimes_\ZZ \QQ \cong J^{(n)}$, we have that for $p\gg 0$ there is no associated prime of $I_1(p) \cap \ldots \cap I_t(p)$ which is embedded \cite[Theorem 2.3.9]{HHCharZero}, and the desired equality follows.
\end{proof}

\begin{remark} \label{remark in} Assume that $<$ is a monomial order on $R = \QQ[x_1,\ldots,x_d]$, and let $J \subseteq R$ be an ideal. We have that $\IN_<(J_A)(p) = \IN_<(J(p))$ for all $p \gg 0$ \cite[Lemma 2.3]{SecciaHankel}. Moreover, any minimal monomial generating set of $\IN_<(J)$ is a minimal monomial generating set of $\IN_<(J(p))$ for $p \gg 0$.
\end{remark}

\begin{remark} \label{remark equal}
If $I \subseteq J$ are two ideals of $R=\QQ[x_1,\ldots,x_d]$ such that $I(p) = J(p)$ for all $p \gg 0$, then $I=J$. In fact, after localizing at a non-zero element $a \in \ZZ$ we may assume that $(J_A/I_A)_a$ is a free $\ZZ_a$-module, by generic freeness. Our assumptions guarantee that there is a sufficiently large prime integer $p$ such that $(J_A/I_A)_a \otimes_{\ZZ_a} \ZZ_a/(p) \cong J_A/I_A \otimes_\ZZ \ZZ/(p) \cong J(p)/I(p) = 0$, and since $(J_A/I_A)_a$ is free over $\ZZ_a$ this implies that $(I_A)_a=(J_A)_a$. In particular, $I=J$.
\end{remark}

\begin{theorem} \label{ThmInitialChar0}
Let $B = \QQ[x_1,\ldots,x_d]$ be equipped with a monomial order $<$. Let $f_1,\ldots,f_s \in A=\ZZ[x_1,\ldots,x_d]$ be homogeneous elements, and $Q=(f_1,\ldots,f_s)B$. Assume that $Q$ is prime, and that $\IN_<(Q)$ is radical. 
For a prime integer $p$ we let $S(p) = \bigoplus_{n\gs 0}  \IN_<(Q(p)^{(n)}) T^n$, $\J(p) = \bigoplus_{n\gs 0}  \IN_<(Q(p)^{(n+1)}) T^n$ and $G(p) = S(p)/\J(p)$. Assume that $S(p)$ is Noetherian and that $G(p)$ is reduced for all $p \gg 0$. If $K$ is any field of characteristic zero, $R=K[x_1,\ldots,x_d]$ is equipped with the same monomial order $<$ as $B$, and $P=QR$, then $\IN_<(P^{(n)}) = \IN_<(P)^{(n)}$ for all $n \geq 1$.
\end{theorem}
\begin{proof}
First we prove the statement for $K=\QQ$, that is, for $R=B$ and $P=Q$. Let $J=Q_A$ and fix $n \in \NN$; by Remarks \ref{remark Sullivant} and \ref{remark equal} we only have to show that $\IN_<(J)^{(n)}(p) = \IN_<(J^{(n)})(p)$ for all $p \gg 0$. By Lemma \ref{lemma symb} and Remark \ref{remark in}, for $p \gg 0$ we have that
\[
\IN_<(J)^{(n)}(p) = (\IN_<(J)(p))^{(n)} = \IN_<(Q(p))^{(n)}
\]
and
\[
\IN_<(J^{(n)})(p) =  \IN_<(J^{(n)}(p))= \IN_<(Q(p)^{(n)}).
\]
Moreover, by Remark \ref{remark in} we have that $\IN_<(Q(p))$ is square-free for all $p \gg 0$, given that $\IN_<(Q)$ is square-free by assumption. 
Since $S(p)$ is finitely generated and $G(p)$ is reduced for all $p \gg 0$, we conclude by Corollary \ref{CorInSymbEq} that $\IN_<(Q(p))^{(n)} = \IN_<(Q(p)^{(n)})$ for all $p \gg 0$, and the proof is complete in this case.

Now let $K$ be any field of characteristic zero and fix $n \in \NN$. 
Since Buchberger's algorithm is stable under base extensions, we have that $\IN_<(I)R= \IN_<(IR)$ for any ideal $I \subseteq B$. Moreover, as the natural inclusion $B \hookrightarrow R \cong B \otimes_\QQ K$ is flat and $\QQ \to K$ is separable, we have that $I^{(n)}R = (IR)^{(n)}$ for any radical ideal $I \subseteq B$. By what we have already shown we finally get that
\[
\IN_<(P^{(n)}) = \IN_<((QR)^{(n)}) = \IN_<(Q^{(n)})R = (\IN_<(Q)^{(n)})R = \IN_<(QR)^{(n)} =  \IN_<(P)^{(n)}.  \qedhere
\]
\end{proof}

\subsection{Main results of this section.} We are ready to present the main results of this section in the context of determinantal ideals. In the generic case, the equality between initial ideals of symbolic powers and symbolic powers of initial ideals was proved by Bruns and Conca \cite[Lemma 7.2]{BCInitial}. The methods developed in this article allow us to recover this result.

\begin{theorem}\label{ThmInitialGen}
Assume Setup \ref{setupDetGen} and $p>\min\{t,r-t \}$. 
Then, the filtration
 $\{\IN_<(I_t(X)^{(n)})\}_{n\in \NN}$  and the algebra $\R(\{\IN_<(I_t(X)^n)\})$  are   $F$-split. Moreover,  
$$\IN_<(I_t(X)^{(n)})=\IN_<(I_t(X))^{(n)}$$
for every $n\in\NN$.
\end{theorem}
\begin{proof}
We consider $f_u(X)$ as in Notation \ref{setupDetGen}.
We have that $\IN_<(f_1(X))$ is square-free  and $f_1(X)\in  \bigcap_{n\in \NN} \left( I_t(X)^n\right)^{[p]}:I_t(X)^{np}$ 
  by the proof of Theorem \ref{ThmReesFpureGen}.
 Then, $\R(\{\IN_<(I_t(X)^n)\})$ is   $F$-split by Proposition \ref{propInSymbOrd} (2). We also have that $f_t(X)\in I_t(X)^{(\Ht(I_t(X)))}$ by the proof of Theorem  \ref{ThmSFPGen}. 
Then, $\{\IN_<(I_t(X)^{(n)})\}_{n\in \NN}$ is an  $F$-split filtration by Proposition \ref{propInSymbOrd} (1).
Thus,
 $$
 G:=\bigoplus_{n\in\NN}\frac{ \IN_<(I_t(X)^{(n)})}{ \IN_<(I_t(X)^{(n+1)})}
 $$
is an  $F$-split ring by Theorem \ref{mainCharP}, and so, it is reduced. 
Since $\R(\{\IN_< (I_t(X)^{(n)})\})$ is a finitely generated algebra \cite[Lemma 7.1.]{BCInitial}, we conclude that 
$$\IN_<(I_t(X)^{(n)})=\IN_<(I_t(X))^{(n)}$$
for every $n\in\NN$
by Corollary \ref{CorInSymbEq}.
\end{proof}

\begin{corollary} \label{CorInitialGen} Let $K$ be a field of characteristic zero, $X$ be a generic $r \times s$ matrix of variables, and $R=K[X]$. For every $t \leq \min \{r,s\}$ and every $n \in \NN$ we have that
\[
\IN_<(I_t(X)^{(n)})=\IN_<(I_t(X))^{(n)}.
\]
\end{corollary}
\begin{proof}
This is immediate consequence of Theorems  \ref{ThmInitialChar0} and \ref{ThmInitialGen}.
\end{proof}

Finally, we now turn our attention to the case of Pfaffians which, to the best of our knowledge, was not previously known.

\begin{theorem}\label{ThmInitialPf}
Assume Setup \ref{setup pfaffians} and $p>\min\{2t,r-2t \}$. 
Then, the filtration
$\{\IN_<(P_{2t}(Z)^{(n)})\}_{n\in\NN}$  and the algebra $\R(\{\IN_<(P_{2t}(Z)^n)\})$  are   $F$-split. Moreover,  
$$\IN_<(P_{2t}(Z)^{(n)})=\IN_<(P_{2t}(Z))^{(n)}$$
for every $n\in\NN$.
\end{theorem}
\begin{proof}
We consider $f_{2u}(Z)$ as in Notation \ref{setup pfaffians}.
We have that $\IN_<(f_2)$ is a square-free monomial and $f_2(Z)\in  \bigcap_{n\in \NN} \left( P_{2t}(Z)^n\right)^{[p]}:P_{2t}(Z)^{np}$ 
  by the proof of Theorem \ref{ThmReesFpurePf}.
 Then, $\R(\{\IN_<(P_{2t}(Z)^n)\})$ is $F$-split  by  Proposition \ref{propInSymbOrd} (2). We also have that  $f_{2t}(Z)\in P_{2t}(Z)^{(\Ht(P_{2t}(Z)))}$ by the proof of Theorem  \ref{ThmSFPPf}, therefore $\{\IN_<(P_{2t}(Z)^{(n)})\}_{n\in\NN}$ is an  $F$-split filtration by Proposition \ref{propInSymbOrd} (1). In particular,
 $$
 G=\bigoplus_{n\in\NN}\frac{ \IN_<(P_{2t}(Z)^{(n)})}{ \IN_<(P_{2t}(Z)^{(n+1)})}.
 $$
is a $F$-split by Theorem \ref{mainCharP}, and so, it is reduced. 
As $\R(\{\IN_<(P_{2t}(Z)^{(n)})\})$ is a finitely generated algebra \cite[Proof of Proposition 3.1]{baetica98},  
we conclude that 
$$\IN_<(P_{2t}(Z)^{(n)})=\IN_<(P_{2t}(Z))^{(n)}$$
for every $n\in\NN$
by Corollary \ref{CorInSymbEq}.
\end{proof}

\begin{corollary} \label{CorInitialPf} Let $K$ be a field of characteristic zero, $Z$ be a generic $r \times r$ skew-symmetric matrix, and $R=K[Z]$. For every $t \leq \lfloor \frac{r}{2} \rfloor$ and every $n \in \NN$ we have that
\[
\IN_<(P_{2t}(Z)^{(n)})=\IN_<(P_{2t}(Z))^{(n)}.
\]
\end{corollary}
\begin{proof}
This is immediate consequence of Theorems  \ref{ThmInitialChar0} and \ref{ThmInitialPf}.
\end{proof}

The case of Hankel matrix was also known, and it is due to Conca \cite[Lemma 3.5 and Theorem 3.8]{conca98}. We recover it here.

\begin{theorem}\label{ThmInitialHankel}
Assume Setup \ref{setup hankel} and $p>\min\{t,r-t \}$. 
Then, the filtration
 $\{\IN_<\big(I_t(W_j^c\big)^{(n)})\}_{n\in \NN}$  and the algebra $\R(\{\IN_<(I_t(W_j^c)^n)\})$  are   $F$-split. Moreover,  
$$\IN_<\big(I_t(W_j^c)^{(n)}\big)=\IN_<\big(I_t(W_j^c)\big)^{(n)}$$
for every $n\in\NN$.
\end{theorem}
\begin{proof}
We set $ f=f_{\rm odd}(W_j^c)$ if $d$ is odd and  $ f=f_{\rm even}(W_j^c)$ if $d$ is even  as in Notation \ref{setup hankel}, and $W=W_j^c$.
We have that $\IN_<(f)$ is square-free and  $f\in  \bigcap_{n\in \NN} \left( I_t(W)^n\right)^{[p]}:I_t(W)^{np}$ by the proof of Theorem \ref{ThmReesFpureHankel}. It follows that $\R(\{\IN_<(I_t(W)^n)\})$ is an  $F$-split ring by Theorem \ref{mainCharP}, and so, it is reduced.  We also have that $f\in I_t(W)^{(\Ht(I_t(W)))}$ by the proof of Theorem   \ref{ThmSFPHankel}, and therefore $\{\IN_<(I_t(W)^{(n)})\}_{n\in \NN}$ is an  $F$-split filtration by Proposition \ref{propInSymbOrd} (1). We also have that
 $$
 G=\bigoplus_{n\in\NN}\frac{ \IN_<(I_t(W)^{(n)})}{ \IN_<(I_t(W)^{(n+1)})}
 $$
is a  $F$-split, and so, it is reduced. Since $\R(\{I_t(W)^{(n)})\})$ is a finitely generated algebra \cite[Theorem 4.1]{conca98}, we conclude that 
$$\IN_<(I_t(W)^{(n)})=\IN_<(I_t(W))^{(n)}$$
for every $n\in\NN$
by Corollary \ref{CorInSymbEq}.
\end{proof}

\begin{corollary} \label{CorInitialHankel} Let $K$ be a field of characteristic zero, $W^c_j$ be a $j \times c+1-j$ Hankel matrix of variables, and $R=K[W^c_j]$. For every $t \leq \min\{j,c+1-j\}$ and every $n \in \NN$ we have that
\[
\IN_<(I_t(W^c_j)^{(n)})=\IN_<(I_t(W^c_j))^{(n)}.
\]
\end{corollary}
\begin{proof}
This is immediate consequence of Theorems  \ref{ThmInitialChar0} and \ref{ThmInitialHankel}.
\end{proof}

\begin{remark} \label{remark symmetric initial}
In the case of minors of a generic symmetric matrix $Y$, it is not known whether the algebra $\R(\{\IN_<(I_t(Y)^{(n)})\})$ is finitely generated. For this reason, we cannot use the same strategy used above for the other three types of determinantal ideals.
\end{remark}

\begin{remark}\label{RemW}
If $R$ is a standard graded polynomial ring over a field $K$ and  $I\subseteq R$ is a homogeneous ideal, we denote by $\alpha(I)$ the 
smallest degree of a minimal generator of $I$. 
Let $<$ be a monomial order on $R$.
We note that $\alpha(I)=\alpha(\IN_<(I))$. In particular, if $I$ 
and $\IN_<(I)$ are radical and $\IN_<(I^{(n)})=\IN_<(I)^{(n)}$ for every $n$, then their Waldschmidt constants coincide. Specifically,
$$
\widehat{\alpha}(I)=\lim\limits_{n\to\infty}\frac{\alpha(I^{(n)})}{n} =\lim\limits_{n\to\infty}\frac{\alpha(\IN_<(I)^{(n)})}{n} =\widehat{\alpha}(\IN_<(I)).
$$
In particular, Theorems \ref{ThmInitialGen}, \ref{ThmInitialPf}, and \ref{ThmInitialHankel} allow us to compute the Waldschmidt constant  of certain determinantal rings via their initial ideals, for which a formula has already been proved \cite{WCMonomial}. 
We point out that one can also compute directly that $\widehat{\alpha}(I_t(X))=\frac{r}{r-t+1}$  \cite[Lemma 7.1.]{BCInitial},
$\widehat{\alpha}(P_{2t}(Z))=\frac{\lfloor r/2\rfloor }{\lfloor r/2\rfloor -t+1}$    \cite[Proof of Proposition 3.1]{baetica98},  and
$\widehat{\alpha}(I_t(W^c_j))=\frac{\lfloor ( c+1)/2\rfloor }{\lfloor (c+1)/2\rfloor -t+1}$   \cite[Theorem 4.1]{conca98}.
\end{remark}

\begin{remark}\label{RemResurgence}
If $R$ is a polynomial ring, $<$ is a monomial order, and  $I\subseteq R$ is a homogeneous ideal, then 
$I^{(a)}\subseteq I^b$ implies that $\IN_<(I^{(a)})\subseteq \IN_<(I^b)$.
We recall that the resurgence of $I$ is defined by
$\rho(I)=\sup\{\frac{a}{b}\; |\;  I^{(a)}\not\subseteq I^b\} $.
If  $\IN_<(I)$ are radical, $\IN_<(I^{(n)})=\IN_<(I)^{(n)}$, and $\IN_<(I^{n})=\IN_<(I)^{n}$ for every $n$, then
$\rho(\IN_<(I))\leq \rho(I)$.
In particular,
this case occurs for ideals of minors of Hankel matrices (see Theorem \ref{ThmInitialHankel} and \cite[Theorem 3.16 (b)]{conca98}).
\end{remark}

Hoa and Trung showed that the  limit above exists for square-free monomial ideals.  In fact, they showed a stronger version for the $a$-invariants \cite[Theorems 4.7 and 4.9]{HoaTrung}.

\begin{corollary}\label{CorInitialTHms}
Assume Setup \ref{setup hankel}.   
 \begin{enumerate}
\item Assume Setup \ref{setupDetGen}.  Then  $\lim\limits_{n\to\infty} \frac{\reg\left(K[X]/\IN_{<}\left( I_t(X)^{(n)}\right)\right)}{n}$ exists. Moreover, for $n \gg 0$ we have that $\depth\left(K[X]/\IN_<\left( I_t(X)^{(n)}\right)\right)$ stabilizes.
\item Assume Setup \ref{setup pfaffians}.  Then  $\lim\limits_{n\to\infty} \frac{\reg\left(K[Z]/\IN_{<}\left( P_{2t}(Z)^{(n)}\right)\right)}{n}$ exists. Moreover, for $n \gg 0$ we have that $\depth\left(K[Z]/\IN_<\left( P_{2t}(Z)^{(n)}\right)\right)$ stabilizes.
\item Assume Setup \ref{setup hankel}.  Then  $\lim\limits_{n\to\infty} \frac{\reg\left(K[W^c_j]/\IN_{<}\left( I_t(W^c_j)^{(n)}\right)\right)}{n}$ exists. Moreover, for $n \gg 0$ we have that $\depth\left(K[W^c_j]/\IN_<\left( I_t(W^c_j)^{(n)}\right)\right)$ stabilizes.
\end{enumerate}
\end{corollary}
\begin{proof}
If $J$ is a square-free monomial ideal in a polynomial ring $R$, it is already known that $\lim\limits_{n\to\infty} \frac{\reg(R/J^{(n)})}{n}$ exists, and that $\depth(R/J^{(n)})$ stabilizes \cite{HoaTrung}. Then, the result follows from Theorems \ref{ThmInitialGen}, \ref{ThmInitialPf}, and \ref{ThmInitialHankel}.
\end{proof}

\section*{Acknowledgments}
The authors would like to  thank Winfried Bruns, Giulio Caviglia, Aldo Conca, Elo\'isa Grifo, Craig Huneke, Delio Jaramillo,  Jack Jeffries, Manoj Kummini, Emmy Lewis, Pedro Ram\'irez-Moreno, Sandra Sandoval, and Matteo Varbaro for helpful discussions and comments. The first author was partially supported by the PRIN 2020 project 2020355B8Y ``Squarefree Gr{\"o}bner degenerations, special varieties and related topics'', by the MIUR Excellence Department Project CUP D33C23001110001, and by INdAM-GNSAGA. The second author was  supported by NSF Grant DMS \#2001645/2303605.
The third author was supported   CONACyT Grant 284598 and C\'atedras Marcos Moshinsky

\bibliographystyle{alpha}
\bibliography{References}

\def\cprime{$'$} \def\cprime{$'$}
\begin{bibdiv}
\begin{biblist}

\bib{YoungPf}{article}{
      author={Abeasis, S.},
      author={Del~Fra, A.},
       title={Young diagrams and ideals of {P}faffians},
        date={1980},
        ISSN={0001-8708},
     journal={Adv. in Math.},
      volume={35},
      number={2},
       pages={158\ndash 178},
         url={https://doi.org/10.1016/0001-8708(80)90046-8},
      review={\MR{560133}},
}

\bib{baetica98}{article}{
      author={Baetica, C.},
       title={Rees algebra of ideals generated by {P}faffians},
        date={1998},
        ISSN={0092-7872},
     journal={Comm. Algebra},
      volume={26},
      number={6},
       pages={1769\ndash 1778},
         url={https://doi.org/10.1080/00927879808826238},
      review={\MR{1621743}},
}

\bib{HC1}{incollection}{
      author={Bauer, Thomas},
      author={Di~Rocco, Sandra},
      author={Harbourne, Brian},
      author={Kapustka, Micha\l},
      author={Knutsen, Andreas},
      author={Syzdek, Wioletta},
      author={Szemberg, Tomasz},
       title={A primer on {S}eshadri constants},
        date={2009},
   booktitle={Interactions of classical and numerical algebraic geometry},
      series={Contemp. Math.},
      volume={496},
   publisher={Amer. Math. Soc., Providence, RI},
       pages={33\ndash 70},
         url={https://doi.org/10.1090/conm/496/09718},
      review={\MR{2555949}},
}

\bib{BH}{article}{
      author={Bertossi, Alan~A.},
       title={Finding {H}amiltonian circuits in proper interval graphs},
        date={1983},
        ISSN={0020-0190},
     journal={Inform. Process. Lett.},
      volume={17},
      number={2},
       pages={97\ndash 101},
         url={https://doi.org/10.1016/0020-0190(83)90078-9},
      review={\MR{731128}},
}

\bib{WCMonomial}{article}{
      author={Bocci, Cristiano},
      author={Cooper, Susan},
      author={Guardo, Elena},
      author={Harbourne, Brian},
      author={Janssen, Mike},
      author={Nagel, Uwe},
      author={Seceleanu, Alexandra},
      author={Van~Tuyl, Adam},
      author={Vu, Thanh},
       title={The {W}aldschmidt constant for squarefree monomial ideals},
        date={2016},
        ISSN={0925-9899},
     journal={J. Algebraic Combin.},
      volume={44},
      number={4},
       pages={875\ndash 904},
         url={https://doi.org/10.1007/s10801-016-0693-7},
      review={\MR{3566223}},
}

\bib{Brion_Kumar_book}{book}{
      author={Brion, Michel},
      author={Kumar, Shrawan},
       title={Frobenius splitting methods in geometry and representation
  theory},
      series={Progress in Mathematics},
   publisher={Birkh\"{a}user Boston, Inc., Boston, MA},
        date={2005},
      volume={231},
        ISBN={0-8176-4191-2},
      review={\MR{2107324}},
}

\bib{BroSharp}{book}{
      author={Brodmann, M.~P.},
      author={Sharp, R.~Y.},
       title={Local cohomology},
     edition={Second},
      series={Cambridge Studies in Advanced Mathematics},
   publisher={Cambridge University Press, Cambridge},
        date={2013},
      volume={136},
        ISBN={978-0-521-51363-0},
        note={An algebraic introduction with geometric applications},
      review={\MR{3014449}},
}

\bib{FRatGeneric}{article}{
      author={Bruns, Winfried},
      author={Conca, Aldo},
       title={{$F$}-rationality of determinantal rings and their {R}ees rings},
        date={1998},
        ISSN={0026-2285},
     journal={Michigan Math. J.},
      volume={45},
      number={2},
       pages={291\ndash 299},
         url={https://doi.org/10.1307/mmj/1030132183},
      review={\MR{1637654}},
}

\bib{brunsconca98}{article}{
      author={Bruns, Winfried},
      author={Conca, Aldo},
       title={K{RS} and powers of determinantal ideals},
        date={1998},
        ISSN={0010-437X},
     journal={Compositio Math.},
      volume={111},
      number={1},
       pages={111\ndash 122},
         url={https://doi.org/10.1023/A:1000287107308},
      review={\MR{1611055}},
}

\bib{BCInitial}{incollection}{
      author={Bruns, Winfried},
      author={Conca, Aldo},
       title={Gr\"{o}bner bases and determinantal ideals},
        date={2003},
   booktitle={Commutative algebra, singularities and computer algebra
  ({S}inaia, 2002)},
      series={NATO Sci. Ser. II Math. Phys. Chem.},
      volume={115},
   publisher={Kluwer Acad. Publ., Dordrecht},
       pages={9\ndash 66},
      review={\MR{2030262}},
}

\bib{BCRV}{book}{
      author={Bruns, Winfried},
      author={Conca, Aldo},
      author={Raicu, Claudiu},
      author={Varbaro, Matteo},
       title={Determinants, {G}r{"o}bner {B}ases and {C}ohomology},
      series={Springer Monographs in Mathematics},
   publisher={Springer},
        date={2022},
        ISBN={978-3-031-05479-2},
         url={https://doi.org/10.1007/978-3-031-05480-8},
}

\bib{BrHe}{book}{
      author={Bruns, Winfried},
      author={Herzog, J{\"u}rgen},
       title={Cohen-{M}acaulay rings},
      series={Cambridge Studies in Advanced Mathematics},
   publisher={Cambridge University Press, Cambridge},
        date={1993},
      volume={39},
        ISBN={0-521-41068-1},
      review={\MR{1251956 (95h:13020)}},
}

\bib{SemiNormal}{article}{
      author={Bruns, Winfried},
      author={Li, Ping},
      author={R\"{o}mer, Tim},
       title={On seminormal monoid rings},
        date={2006},
        ISSN={0021-8693},
     journal={J. Algebra},
      volume={302},
      number={1},
       pages={361\ndash 386},
         url={https://doi.org/10.1016/j.jalgebra.2005.11.012},
      review={\MR{2236607}},
}

\bib{BookDet}{book}{
      author={Bruns, Winfried},
      author={Vetter, Udo},
       title={Determinantal rings},
      series={Lecture Notes in Mathematics},
   publisher={Springer-Verlag, Berlin},
        date={1988},
      volume={1327},
        ISBN={3-540-19468-1},
         url={https://doi.org/10.1007/BFb0080378},
      review={\MR{953963}},
}

\bib{conca98}{article}{
      author={Conca, Aldo},
       title={Straightening law and powers of determinantal ideals of {H}ankel
  matrices},
        date={1998},
        ISSN={0001-8708},
     journal={Adv. Math.},
      volume={138},
      number={2},
       pages={263\ndash 292},
         url={https://doi.org/10.1006/aima.1998.1740},
      review={\MR{1645574}},
}

\bib{Sagbi}{article}{
      author={Conca, Aldo},
      author={Herzog, J\"{u}rgen},
      author={Valla, Giuseppe},
       title={Sagbi bases with applications to blow-up algebras},
        date={1996},
        ISSN={0075-4102},
     journal={J. Reine Angew. Math.},
      volume={474},
       pages={113\ndash 138},
         url={https://doi.org/10.1515/crll.1996.474.113},
      review={\MR{1390693}},
}

\bib{Hankel}{article}{
      author={Conca, Aldo},
      author={Mostafazadehfard, Maral},
      author={Singh, Anurag~K.},
      author={Varbaro, Matteo},
       title={Hankel determinantal rings have rational singularities},
        date={2018},
        ISSN={0001-8708},
     journal={Adv. Math.},
      volume={335},
       pages={111\ndash 129},
         url={https://doi.org/10.1016/j.aim.2018.06.011},
      review={\MR{3836659}},
}

\bib{CC}{article}{
      author={Cornu{\'e}jols, G{\'e}rard},
      author={Conforti, Michele},
       title={A decomposition theorem for balanced matrices},
        date={1990},
     journal={Integer Programming and Combinatorial Optimization},
      volume={74},
       pages={147\ndash 169},
}

\bib{CRPI}{article}{
      author={Crupi, Marilena},
      author={Rinaldo, Giancarlo},
       title={Closed graphs are proper interval graphs},
        date={2014},
        ISSN={1224-1784},
     journal={An. \c{S}tiin\c{t}. Univ. ``Ovidius'' Constan\c{t}a Ser. Mat.},
      volume={22},
      number={3},
       pages={37\ndash 44},
         url={https://doi.org/10.2478/auom-2014-0048},
      review={\MR{3215896}},
}

\bib{Cutkosky_irrational}{article}{
      author={Cutkosky, S.~Dale},
       title={Irrational asymptotic behaviour of {C}astelnuovo-{M}umford
  regularity},
        date={2000},
        ISSN={0075-4102,1435-5345},
     journal={J. Reine Angew. Math.},
      volume={522},
       pages={93\ndash 103},
         url={https://doi.org/10.1515/crll.2000.043},
      review={\MR{1758577}},
}

\bib{CutkoskyKurano}{article}{
      author={Cutkosky, Steven~Dale},
      author={Kurano, Kazuhiko},
       title={Asymptotic regularity of powers of ideals of points in a weighted
  projective plane},
        date={2011},
        ISSN={2156-2261},
     journal={Kyoto J. Math.},
      volume={51},
      number={1},
       pages={25\ndash 45},
         url={https://doi.org/10.1215/0023608X-2010-019},
      review={\MR{2784746}},
}

\bib{DalzottoSbarra}{article}{
      author={Dalzotto, Giorgio},
      author={Sbarra, Enrico},
       title={On non-standard graded algebras},
        date={2008},
        ISSN={1880-6015},
     journal={Toyama Math. J.},
      volume={31},
       pages={33\ndash 57},
      review={\MR{2510523}},
}

\bib{SurveySymbPowers}{incollection}{
      author={Dao, Hailong},
      author={De~Stefani, Alessandro},
      author={Grifo, Elo\'{\i}sa},
      author={Huneke, Craig},
      author={N\'{u}\~{n}ez Betancourt, Luis},
       title={Symbolic powers of ideals},
        date={2018},
   booktitle={Singularities and foliations. geometry, topology and
  applications},
      series={Springer Proc. Math. Stat.},
      volume={222},
   publisher={Springer, Cham},
       pages={387\ndash 432},
         url={https://doi.org/10.1007/978-3-319-73639-6_13},
      review={\MR{3779569}},
}

\bib{dCEP}{article}{
      author={de~Concini, C.},
      author={Eisenbud, David},
      author={Procesi, C.},
       title={Young diagrams and determinantal varieties},
        date={1980},
        ISSN={0020-9910},
     journal={Invent. Math.},
      volume={56},
      number={2},
       pages={129\ndash 165},
         url={https://doi.org/10.1007/BF01392548},
      review={\MR{558865}},
}

\bib{dCP}{article}{
      author={de~Concini, C.},
      author={Procesi, C.},
       title={A characteristic free approach to invariant theory},
        date={1976},
        ISSN={0001-8708},
     journal={Advances in Math.},
      volume={21},
      number={3},
       pages={330\ndash 354},
         url={https://doi.org/10.1016/S0001-8708(76)80003-5},
      review={\MR{422314}},
}

\bib{DNPaf}{article}{
      author={De~Negri, Emanuela},
       title={{$K$}-algebras generated by {P}faffians},
        date={1996},
        ISSN={0916-6009},
     journal={Math. J. Toyama Univ.},
      volume={19},
       pages={105\ndash 114},
      review={\MR{1427689}},
}

\bib{ZariskiNagata}{article}{
      author={De~Stefani, Alessandro},
      author={Grifo, Elo\'{\i}sa},
      author={Jeffries, Jack},
       title={A {Z}ariski-{N}agata theorem for smooth {$\Bbb{Z}$}-algebras},
        date={2020},
        ISSN={0075-4102},
     journal={J. Reine Angew. Math.},
      volume={761},
       pages={123\ndash 140},
         url={https://doi.org/10.1515/crelle-2018-0012},
      review={\MR{4080246}},
}

\bib{DSGNB}{article}{
      author={De~Stefani, Alessandro},
      author={Grifo, Eloísa},
      author={N{\'u}{\~{n}}ez-Betancourt, Luis},
       title={Local cohomology and {L}yubeznik numbers of {F}-pure rings},
        date={2018},
        ISSN={0021-8693},
     journal={J. Algebra},
  url={http://www.sciencedirect.com/science/article/pii/S0021869318304265},
}

\bib{DRS}{article}{
      author={Doubilet, Peter},
      author={Rota, Gian-Carlo},
      author={Stein, Joel},
       title={On the foundations of combinatorial theory. {IX}. {C}ombinatorial
  methods in invariant theory},
        date={1974},
        ISSN={0022-2526},
     journal={Studies in Appl. Math.},
      volume={53},
       pages={185\ndash 216},
         url={https://doi.org/10.1002/sapm1974533185},
      review={\MR{498650}},
}

\bib{ExNotLin}{article}{
      author={Dung, Le~Xuan},
      author={Hien, Truong~Thi},
      author={Nguyen, Hop~D.},
      author={Trung, Tran~Nam},
       title={Regularity and koszul property of symbolic powers of monomial
  ideals},
        date={2019},
     journal={arXiv preprint arXiv:1903.09026},
}

\bib{EneHerzogSymbolic}{article}{
      author={Ene, Viviana},
      author={Herzog, J{\"{u}}rgen},
       title={On the symbolic powers of binomial edge ideals},
        date={2020},
     journal={arXiv:2009.03156},
}

\bib{EHH}{article}{
      author={Ene, Viviana},
      author={Herzog, J\"{u}rgen},
      author={Hibi, Takayuki},
       title={Cohen-{M}acaulay binomial edge ideals},
        date={2011},
        ISSN={0027-7630},
     journal={Nagoya Math. J.},
      volume={204},
       pages={57\ndash 68},
         url={https://doi.org/10.1215/00277630-1431831},
      review={\MR{2863365}},
}

\bib{FedderFputityFsing}{article}{
      author={Fedder, Richard},
       title={{$F$}-purity and rational singularity},
        date={1983},
        ISSN={0002-9947},
     journal={Trans. Amer. Math. Soc.},
      volume={278},
      number={2},
       pages={461\ndash 480},
         url={http://dx.doi.org.proxy.its.virginia.edu/10.2307/1999165},
      review={\MR{701505 (84h:13031)}},
}

\bib{FK2015}{article}{
      author={Fouli, Louiza},
      author={Lin, Kuei-Nuan},
       title={Rees algebras of square-free monomial ideals},
        date={2015},
        ISSN={1939-0807},
     journal={J. Commut. Algebra},
      volume={7},
      number={1},
       pages={25\ndash 54},
         url={https://doi.org/10.1216/JCA-2015-7-1-25},
      review={\MR{3316984}},
}

\bib{grayson2002macaulay}{misc}{
      author={Grayson, Daniel~R},
      author={Stillman, Michael~E},
       title={Macaulay 2, a software system for research in algebraic
  geometry},
        date={2002},
}

\bib{GrifoHuneke}{article}{
      author={Grifo, Elo\'{\i}sa},
      author={Huneke, Craig},
       title={Symbolic powers of ideals defining {F}-pure and strongly
  {F}-regular rings},
        date={2019},
        ISSN={1073-7928},
     journal={Int. Math. Res. Not. IMRN},
      number={10},
       pages={2999\ndash 3014},
         url={https://doi.org/10.1093/imrn/rnx213},
      review={\MR{3952556}},
}

\bib{GrifoSeceleanu}{article}{
      author={Grifo, Elo{\`i}sa},
      author={Seceleanu, Alexandra},
       title={Symbolic rees algebras},
        date={2020},
     journal={arXiv preprint arXiv:2012.01617},
}

\bib{Ha02}{article}{
      author={H\`a, Huy~T\`ai},
       title={Box-shaped matrices and the defining ideal of certain blowup
  surfaces},
        date={2002},
        ISSN={0022-4049},
     journal={J. Pure Appl. Algebra},
      volume={167},
      number={2-3},
       pages={203\ndash 224},
         url={https://doi.org/10.1016/S0022-4049(01)00032-9},
      review={\MR{1874542}},
}

\bib{H98}{article}{
      author={Hara, Nobuo},
       title={A characterization of rational singularities in terms of
  injectivity of {F}robenius maps},
        date={1998},
        ISSN={0002-9327},
     journal={Amer. J. Math.},
      volume={120},
      number={5},
       pages={981\ndash 996},
  url={http://muse.jhu.edu/journals/american_journal_of_mathematics/v120/120.5hara.pdf},
      review={\MR{1646049}},
}

\bib{HW02}{article}{
      author={Hara, Nobuo},
      author={Watanabe, Kei-Ichi},
       title={F-regular and {F}-pure rings vs. log terminal and log canonical
  singularities},
        date={2002},
        ISSN={1056-3911},
     journal={J. Algebraic Geom.},
      volume={11},
      number={2},
       pages={363\ndash 392},
         url={https://doi.org/10.1090/S1056-3911-01-00306-X},
      review={\MR{1874118}},
}

\bib{HC2}{article}{
      author={Harbourne, Brian},
      author={Huneke, Craig},
       title={Are symbolic powers highly evolved?},
        date={2013},
        ISSN={0970-1249},
     journal={J. Ramanujan Math. Soc.},
      volume={28A},
       pages={247\ndash 266},
      review={\MR{3115195}},
}

\bib{HeHi}{article}{
      author={Herzog, J\"{u}rgen},
      author={Hibi, Takayuki},
       title={The depth of powers of an ideal},
        date={2005},
        ISSN={0021-8693},
     journal={J. Algebra},
      volume={291},
      number={2},
       pages={534\ndash 550},
         url={https://doi.org/10.1016/j.jalgebra.2005.04.007},
      review={\MR{2163482}},
}

\bib{HHHKR10}{article}{
      author={Herzog, J\"{u}rgen},
      author={Hibi, Takayuki},
      author={Hreinsd\'{o}ttir, Freyja},
      author={Kahle, Thomas},
      author={Rauh, Johannes},
       title={Binomial edge ideals and conditional independence statements},
        date={2010},
        ISSN={0196-8858},
     journal={Adv. in Appl. Math.},
      volume={45},
      number={3},
       pages={317\ndash 333},
         url={https://doi.org/10.1016/j.aam.2010.01.003},
      review={\MR{2669070}},
}

\bib{VCA}{article}{
      author={Herzog, J\"{u}rgen},
      author={Hibi, Takayuki},
      author={Trung, Ng\^{o}~Vi\^{e}t},
       title={Symbolic powers of monomial ideals and vertex cover algebras},
        date={2007},
        ISSN={0001-8708},
     journal={Adv. Math.},
      volume={210},
      number={1},
       pages={304\ndash 322},
         url={https://doi.org/10.1016/j.aim.2006.06.007},
      review={\MR{2298826}},
}

\bib{HHT}{article}{
      author={Herzog, J{\"u}rgen},
      author={Hoa, L{\^e}~Tu{\^a}n},
      author={Trung, Ng{\^o}~Vi{\^e}t},
       title={Asymptotic linear bounds for the {C}astelnuovo-{M}umford
  regularity},
        date={2002},
        ISSN={0002-9947},
     journal={Trans. Amer. Math. Soc.},
      volume={354},
      number={5},
       pages={1793\ndash 1809},
         url={http://dx.doi.org/10.1090/S0002-9947-02-02932-X},
      review={\MR{1881017}},
}

\bib{HoaTrung}{article}{
      author={Hoa, L\^{e}~Tu\^{a}n},
      author={Trung, Tr\^{a}n~Nam},
       title={Partial {C}astelnuovo-{M}umford regularities of sums and
  intersections of powers of monomial ideals},
        date={2010},
        ISSN={0305-0041},
     journal={Math. Proc. Cambridge Philos. Soc.},
      volume={149},
      number={2},
       pages={229\ndash 246},
         url={https://doi.org/10.1017/S0305004110000071},
      review={\MR{2670214}},
}

\bib{HochsterTori}{article}{
      author={Hochster, M.},
       title={Rings of invariants of tori, {C}ohen-{M}acaulay rings generated
  by monomials, and polytopes},
        date={1972},
        ISSN={0003-486X},
     journal={Ann. of Math. (2)},
      volume={96},
       pages={318\ndash 337},
         url={https://doi.org/10.2307/1970791},
      review={\MR{304376}},
}

\bib{HoHuStrong}{article}{
      author={Hochster, Melvin},
      author={Huneke, Craig},
       title={Tight closure and strong {$F$}-regularity},
        date={1989},
        ISSN={0037-9484},
     journal={M\'em. Soc. Math. France (N.S.)},
      number={38},
       pages={119\ndash 133},
        note={Colloque en l'honneur de Pierre Samuel (Orsay, 1987)},
      review={\MR{1044348 (91i:13025)}},
}

\bib{HoHu1}{article}{
      author={Hochster, Melvin},
      author={Huneke, Craig},
       title={Tight closure, invariant theory, and the {B}rian\c con-{S}koda
  theorem},
        date={1990},
        ISSN={0894-0347},
     journal={J. Amer. Math. Soc.},
      volume={3},
      number={1},
       pages={31\ndash 116},
         url={http://dx.doi.org/10.2307/1990984},
      review={\MR{1017784 (91g:13010)}},
}

\bib{HoHu2}{article}{
      author={Hochster, Melvin},
      author={Huneke, Craig},
       title={{$F$}-regularity, test elements, and smooth base change},
        date={1994},
        ISSN={0002-9947},
     journal={Trans. Amer. Math. Soc.},
      volume={346},
      number={1},
       pages={1\ndash 62},
         url={http://dx.doi.org/10.2307/2154942},
      review={\MR{1273534 (95d:13007)}},
}

\bib{HoHu3}{article}{
      author={Hochster, Melvin},
      author={Huneke, Craig},
       title={Tight closure of parameter ideals and splitting in module-finite
  extensions},
        date={1994},
        ISSN={1056-3911},
     journal={J. Algebraic Geom.},
      volume={3},
      number={4},
       pages={599\ndash 670},
      review={\MR{1297848 (95k:13002)}},
}

\bib{HHCharZero}{unpublished}{
      author={Hochster, Melvin},
      author={Huneke, Craig},
       title={Tight closure in equal characteristic zero},
        date={1999},
}

\bib{HoRo}{article}{
      author={Hochster, Melvin},
      author={Roberts, Joel~L.},
       title={Rings of invariants of reductive groups acting on regular rings
  are {C}ohen-{M}acaulay},
        date={1974},
        ISSN={0001-8708},
     journal={Advances in Math.},
      volume={13},
       pages={115\ndash 175},
      review={\MR{0347810 (50 \#311)}},
}

\bib{HRFpurity}{article}{
      author={Hochster, Melvin},
      author={Roberts, Joel~L.},
       title={The purity of the {F}robenius and local cohomology},
        date={1976},
        ISSN={0001-8708},
     journal={Advances in Math.},
      volume={21},
      number={2},
       pages={117\ndash 172},
      review={\MR{0417172 (54 \#5230)}},
}

\bib{HSV12}{article}{
      author={Hong, Jooyoun},
      author={Simis, Aron},
      author={Vasconcelos, Wolmer~V.},
       title={The equations of almost complete intersections},
        date={2012},
        ISSN={1678-7544},
     journal={Bull. Braz. Math. Soc. (N.S.)},
      volume={43},
      number={2},
       pages={171\ndash 199},
         url={https://doi.org/10.1007/s00574-012-0009-z},
      review={\MR{2928168}},
}

\bib{Huckaba}{article}{
      author={Huckaba, Sam},
       title={On complete {$d$}-sequences and the defining ideals of {R}ees
  algebras},
        date={1989},
        ISSN={0305-0041},
     journal={Math. Proc. Cambridge Philos. Soc.},
      volume={106},
      number={3},
       pages={445\ndash 458},
         url={https://doi.org/10.1017/S0305004100068183},
      review={\MR{1010369}},
}

\bib{HSV89}{incollection}{
      author={Huneke, Craig},
      author={Simis, Aron},
      author={Vasconcelos, Wolmer},
       title={Reduced normal cones are domains},
        date={1989},
   booktitle={Invariant theory ({D}enton, {TX}, 1986)},
      series={Contemp. Math.},
      volume={88},
   publisher={Amer. Math. Soc., Providence, RI},
       pages={95\ndash 101},
         url={https://doi.org/10.1090/conm/088/999985},
      review={\MR{999985}},
}

\bib{huneke2006integral}{book}{
      author={Huneke, Craig},
      author={Swanson, Irena},
       title={Integral closure of ideals, rings, and modules},
   publisher={Cambridge University Press},
        date={2006},
      volume={13},
}

\bib{JMV}{article}{
      author={Jeffries, Jack},
      author={Monta\~{n}o, Jonathan},
      author={Varbaro, Matteo},
       title={Multiplicities of classical varieties},
        date={2015},
        ISSN={0024-6115},
     journal={Proc. Lond. Math. Soc. (3)},
      volume={110},
      number={4},
       pages={1033\ndash 1055},
         url={https://doi.org/10.1112/plms/pdv005},
      review={\MR{3335294}},
}

\bib{HtSymm}{article}{
      author={J\'{o}zefiak, Tadeusz},
       title={Ideals generated by minors of a symmetric matrix},
        date={1978},
        ISSN={0010-2571},
     journal={Comment. Math. Helv.},
      volume={53},
      number={4},
       pages={595\ndash 607},
         url={https://doi.org/10.1007/BF02566100},
      review={\MR{511849}},
}

\bib{HtSkew}{article}{
      author={J\'{o}zefiak, Tadeusz},
      author={Pragacz, Piotr},
       title={Ideals generated by {P}faffians},
        date={1979},
        ISSN={0021-8693},
     journal={J. Algebra},
      volume={61},
      number={1},
       pages={189\ndash 198},
         url={https://doi.org/10.1016/0021-8693(79)90313-2},
      review={\MR{554859}},
}

\bib{Kady18}{article}{
      author={Kadyrsizova, Zhibek},
       title={Nearly commuting matrices},
        date={2018},
        ISSN={0021-8693},
     journal={J. Algebra},
      volume={497},
       pages={199\ndash 218},
         url={https://doi.org/10.1016/j.jalgebra.2017.10.019},
      review={\MR{3743180}},
}

\bib{S1}{incollection}{
      author={Kapur, Deepak},
      author={Madlener, Klaus},
       title={A completion procedure for computing a canonical basis for a
  {$k$}-subalgebra},
        date={1989},
   booktitle={Computers and mathematics ({C}ambridge, {MA}, 1989)},
   publisher={Springer, New York},
       pages={1\ndash 11},
      review={\MR{1005954}},
}

\bib{KPU2011}{article}{
      author={Kustin, Andrew~R.},
      author={Polini, Claudia},
      author={Ulrich, Bernd},
       title={Rational normal scrolls and the defining equations of {R}ees
  algebras},
        date={2011},
        ISSN={0075-4102},
     journal={J. Reine Angew. Math.},
      volume={650},
       pages={23\ndash 65},
         url={https://doi.org/10.1515/CRELLE.2011.002},
      review={\MR{2770555}},
}

\bib{LauritzenThomsen}{article}{
      author={Lauritzen, Niels},
      author={Thomsen, Jesper~Funch},
       title={Maximal compatible splitting and diagonals of {K}empf varieties},
        date={2011},
        ISSN={0373-0956,1777-5310},
     journal={Ann. Inst. Fourier (Grenoble)},
      volume={61},
      number={6},
       pages={2543\ndash 2575},
         url={https://doi.org/10.5802/aif.2682},
      review={\MR{2976320}},
}

\bib{Lewis20}{article}{
      author={Lewis, James},
       title={Limit behavior of the rational powers of monomial ideals},
        date={2020},
     journal={arXiv preprint arXiv:2009.05173},
}

\bib{LP2014}{article}{
      author={Lin, Kuei-Nuan},
      author={Polini, Claudia},
       title={Rees algebras of truncations of complete intersections},
        date={2014},
        ISSN={0021-8693},
     journal={J. Algebra},
      volume={410},
       pages={36\ndash 52},
         url={https://doi.org/10.1016/j.jalgebra.2014.03.022},
      review={\MR{3201047}},
}

\bib{MS97}{article}{
      author={Mehta, V.~B.},
      author={Srinivas, V.},
       title={A characterization of rational singularities},
        date={1997},
        ISSN={1093-6106},
     journal={Asian J. Math.},
      volume={1},
      number={2},
       pages={249\ndash 271},
         url={https://doi.org/10.4310/AJM.1997.v1.n2.a4},
      review={\MR{1491985}},
}

\bib{MVSymm}{article}{
      author={Micali, Artibano},
      author={Villamayor, Orlando~E.},
       title={Sur les alg\`ebres de {C}lifford},
        date={1968},
        ISSN={0012-9593},
     journal={Ann. Sci. \'{E}cole Norm. Sup. (4)},
      volume={1},
       pages={271\ndash 304},
  url={https://www-numdam-org.biblio.cimat.remotexs.co/item?id=ASENS_1968_4_1_2_271_0},
      review={\MR{255583}},
}

\bib{Morey97}{article}{
      author={Morey, Susan},
       title={Equations of blowups of ideals of codimension two and three},
        date={1996},
        ISSN={0022-4049},
     journal={J. Pure Appl. Algebra},
      volume={109},
      number={2},
       pages={197\ndash 211},
         url={https://doi.org/10.1016/0022-4049(95)00087-9},
      review={\MR{1387739}},
}

\bib{MU96}{article}{
      author={Morey, Susan},
      author={Ulrich, Bernd},
       title={Rees algebras of ideals with low codimension},
        date={1996},
        ISSN={0002-9939},
     journal={Proc. Amer. Math. Soc.},
      volume={124},
      number={12},
       pages={3653\ndash 3661},
         url={https://doi.org/10.1090/S0002-9939-96-03470-3},
      review={\MR{1343713}},
}

\bib{MP2013}{article}{
      author={Mui\~{n}os, Ferran},
      author={Planas-Vilanova, Francesc},
       title={The equations of {R}ees algebras of equimultiple ideals of
  deviation one},
        date={2013},
        ISSN={0002-9939},
     journal={Proc. Amer. Math. Soc.},
      volume={141},
      number={4},
       pages={1241\ndash 1254},
         url={https://doi.org/10.1090/S0002-9939-2012-11398-X},
      review={\MR{3008872}},
}

\bib{Nagata}{book}{
      author={Nagata, Masayoshi},
       title={{Local rings}},
   publisher={Interscience},
        date={1962},
}

\bib{Ohtani}{article}{
      author={Ohtani, Masahiro},
       title={Graphs and ideals generated by some 2-minors},
        date={2011},
        ISSN={0092-7872},
     journal={Comm. Algebra},
      volume={39},
      number={3},
       pages={905\ndash 917},
         url={https://doi.org/10.1080/00927870903527584},
      review={\MR{2782571}},
}

\bib{Perlman}{article}{
      author={Perlman, Michael},
       title={Regularity and cohomology of pfaffian thickenings},
        date={2019},
     journal={J. Commut. Algebra},
         url={https://projecteuclid.org:443/euclid.jca/1561363248},
        note={Advance publication},
}

\bib{Raicu}{article}{
      author={Raicu, Claudiu},
       title={Regularity and cohomology of determinantal thickenings},
        date={2018},
        ISSN={0024-6115},
     journal={Proc. Lond. Math. Soc. (3)},
      volume={116},
      number={2},
       pages={248\ndash 280},
         url={https://doi.org/10.1112/plms.12071},
      review={\MR{3764061}},
}

\bib{ratliff1979notes}{article}{
      author={Ratliff, L.~J., Jr.},
       title={Notes on essentially powers filtrations},
        date={1979},
        ISSN={0026-2285},
     journal={Michigan Math. J.},
      volume={26},
      number={3},
       pages={313\ndash 324},
         url={http://projecteuclid.org/euclid.mmj/1029002263},
      review={\MR{544599}},
}

\bib{S2}{incollection}{
      author={Robbiano, Lorenzo},
      author={Sweedler, Moss},
       title={Subalgebra bases},
        date={1990},
   booktitle={Commutative algebra ({S}alvador, 1988)},
      series={Lecture Notes in Math.},
      volume={1430},
   publisher={Springer, Berlin},
       pages={61\ndash 87},
         url={https://doi.org/10.1007/BFb0085537},
      review={\MR{1068324}},
}

\bib{Roberts}{article}{
      author={Roberts, Paul~C.},
       title={A prime ideal in a polynomial ring whose symbolic blow-up is not
  {N}oetherian},
        date={1985},
        ISSN={0002-9939},
     journal={Proc. Amer. Math. Soc.},
      volume={94},
      number={4},
       pages={589\ndash 592},
         url={https://doi.org/10.2307/2044869},
      review={\MR{792266}},
}

\bib{SurveyTestIdeals}{incollection}{
      author={Schwede, Karl},
      author={Tucker, Kevin},
       title={A survey of test ideals},
        date={2012},
   booktitle={Progress in commutative algebra 2},
   publisher={Walter de Gruyter, Berlin},
       pages={39\ndash 99},
      review={\MR{2932591}},
}

\bib{SecciaHankel}{article}{
      author={Seccia, Lisa},
       title={Knutson ideals and determinantal ideals of {H}ankel matrices},
        date={2021},
        ISSN={0022-4049},
     journal={J. Pure Appl. Algebra},
      volume={225},
      number={12},
       pages={106788, 17},
         url={https://doi.org/10.1016/j.jpaa.2021.106788},
      review={\MR{4260034}},
}

\bib{SecciaGen}{article}{
      author={Seccia, Lisa},
       title={Knutson ideals of generic matrices},
        date={2021},
     journal={arXiv preprint arXiv:2101.06496},
}

\bib{STV98}{article}{
      author={Simis, A.},
      author={Trung, N.~V.},
      author={Valla, G.},
       title={The diagonal subalgebra of a blow-up algebra},
        date={1998},
        ISSN={0022-4049},
     journal={J. Pure Appl. Algebra},
      volume={125},
      number={1-3},
       pages={305\ndash 328},
         url={https://doi.org/10.1016/S0022-4049(96)00127-2},
      review={\MR{1600032}},
}

\bib{SUV95}{article}{
      author={Simis, Aron},
      author={Ulrich, Bernd},
      author={Vasconcelos, Wolmer~V.},
       title={Cohen-{M}acaulay {R}ees algebras and degrees of polynomial
  relations},
        date={1995},
        ISSN={0025-5831},
     journal={Math. Ann.},
      volume={301},
      number={3},
       pages={421\ndash 444},
         url={https://doi.org/10.1007/BF01446637},
      review={\MR{1324518}},
}

\bib{SVV}{article}{
      author={Simis, Aron},
      author={Vasconcelos, Wolmer~V.},
      author={Villarreal, Rafael~H.},
       title={On the ideal theory of graphs},
        date={1994},
        ISSN={0021-8693},
     journal={J. Algebra},
      volume={167},
      number={2},
       pages={389\ndash 416},
         url={https://doi.org/10.1006/jabr.1994.1192},
      review={\MR{1283294}},
}

\bib{singh99}{article}{
      author={Singh, Anurag~K.},
       title={{$F$}-regularity does not deform},
        date={1999},
        ISSN={0002-9327},
     journal={Amer. J. Math.},
      volume={121},
      number={4},
       pages={919\ndash 929},
  url={http://muse.jhu.edu/journals/american_journal_of_mathematics/v121/121.4singh.pdf},
      review={\MR{1704481}},
}

\bib{DModFSplit}{article}{
      author={Smith, Karen~E.},
       title={The {$D$}-module structure of {$F$}-split rings},
        date={1995},
        ISSN={1073-2780},
     journal={Math. Res. Lett.},
      volume={2},
      number={4},
       pages={377\ndash 386},
      review={\MR{1355702 (96j:13024)}},
}

\bib{SmithFrational}{article}{
      author={Smith, Karen~E.},
       title={{$F$}-rational rings have rational singularities},
        date={1997},
        ISSN={0002-9327},
     journal={Amer. J. Math.},
      volume={119},
      number={1},
       pages={159\ndash 180},
  url={http://muse.jhu.edu.proxy.lib.umich.edu/journals/american_journal_of_mathematics/v119/119.1smith.pdf},
      review={\MR{1428062 (97k:13004)}},
}

\bib{Smith97}{incollection}{
      author={Smith, Karen~E.},
       title={Vanishing, singularities and effective bounds via prime
  characteristic local algebra},
        date={1997},
   booktitle={Algebraic geometry---{S}anta {C}ruz 1995},
      series={Proc. Sympos. Pure Math.},
      volume={62},
   publisher={Amer. Math. Soc., Providence, RI},
       pages={289\ndash 325},
      review={\MR{1492526}},
}

\bib{SmithComm}{article}{
      author={Smith, Karen~E.},
       title={The multiplier ideal is a universal test ideal},
        date={2000},
        ISSN={0092-7872},
     journal={Comm. Algebra},
      volume={28},
      number={12},
       pages={5915\ndash 5929},
         url={http://dx.doi.org/10.1080/00927870008827196},
        note={Special issue in honor of Robin Hartshorne},
      review={\MR{1808611}},
}

\bib{SKRS}{article}{
      author={Sturmfels, Bernd},
       title={Gr\"{o}bner bases and {S}tanley decompositions of determinantal
  rings},
        date={1990},
        ISSN={0025-5874},
     journal={Math. Z.},
      volume={205},
      number={1},
       pages={137\ndash 144},
         url={https://doi.org/10.1007/BF02571229},
      review={\MR{1069489}},
}

\bib{Sull}{article}{
      author={Sullivant, Seth},
       title={Combinatorial symbolic powers},
        date={2008},
        ISSN={0021-8693},
     journal={J. Algebra},
      volume={319},
      number={1},
       pages={115\ndash 142},
         url={https://doi.org/10.1016/j.jalgebra.2007.09.024},
      review={\MR{2378064}},
}

\bib{Ha2002}{article}{
      author={T\`ai, H\`a~Huy},
       title={On the {R}ees algebra of certain codimension two perfect ideals},
        date={2002},
        ISSN={0025-2611},
     journal={Manuscripta Math.},
      volume={107},
      number={4},
       pages={479\ndash 501},
         url={https://doi.org/10.1007/s002290200247},
      review={\MR{1906772}},
}

\bib{Trung_Wang}{article}{
      author={Trung, Ng\^{o}~Vi\^{e}t},
      author={Wang, Hsin-Ju},
       title={On the asymptotic linearity of {C}astelnuovo-{M}umford
  regularity},
        date={2005},
        ISSN={0022-4049},
     journal={J. Pure Appl. Algebra},
      volume={201},
      number={1-3},
       pages={42\ndash 48},
         url={https://doi.org/10.1016/j.jpaa.2004.12.043},
      review={\MR{2158746}},
}

\bib{Varb}{article}{
      author={Varbaro, Matteo},
       title={Symbolic powers and matroids},
        date={2011},
        ISSN={0002-9939},
     journal={Proc. Amer. Math. Soc.},
      volume={139},
      number={7},
       pages={2357\ndash 2366},
         url={https://doi.org/10.1090/S0002-9939-2010-10685-8},
      review={\MR{2784800}},
}

\bib{V91}{article}{
      author={Vasconcelos, Wolmer~V.},
       title={On the equations of {R}ees algebras},
        date={1991},
        ISSN={0075-4102},
     journal={J. Reine Angew. Math.},
      volume={418},
       pages={189\ndash 218},
         url={https://doi.org/10.1515/crll.1991.418.189},
      review={\MR{1111206}},
}

\bib{Vasconcelos}{book}{
      author={Vasconcelos, Wolmer~V.},
       title={Arithmetic of blowup algebras},
      series={London Mathematical Society Lecture Note Series},
   publisher={Cambridge University Press, Cambridge},
        date={1994},
      volume={195},
        ISBN={0-521-45484-0},
         url={https://doi.org/10.1017/CBO9780511574726},
      review={\MR{1275840}},
}

\bib{Zariski}{article}{
      author={Zariski, Oscar},
       title={A fundamental lemma from the theory of holomorphic functions on
  an algebraic variety},
        date={1949},
        ISSN={0003-4622},
     journal={Ann. Mat. Pura Appl. (4)},
      volume={29},
       pages={187\ndash 198},
         url={http://dx.doi.org/10.1007/BF02413926},
      review={\MR{0041488}},
}

\end{biblist}
\end{bibdiv}

\end{document}